\documentclass[final]{siamart190516}
\usepackage{amsfonts}
\usepackage{amsfonts,amsmath,amssymb}
\usepackage{mathrsfs,mathtools,stmaryrd,wasysym}
\usepackage{enumerate}
\usepackage{enumitem}
\usepackage{esint}
\usepackage{graphicx,psfrag}
\usepackage{hyperref}
\usepackage{amsmath,stackengine}
\usepackage{yfonts}
\DeclareMathAlphabet{\mathpzc}{OT1}{pzc}{m}{it}

\newcommand{\EO}[1]{{\color{black}#1}}
\newcommand{\DQ}[1]{{\color{black}#1}}

\newcommand{\T}{\mathscr{T}}

\newsiamremark{remark}{Remark}

\newcommand{\TheTitle}{Bilinear optimal control for the fractional Laplacian: analysis and discretization}
\newcommand{\ShortTitle}{Bilinear optimal control for fractional diffusion}
\newcommand{\TheAuthors}{F. Bersetche, F. Fuica, E. Ot\'arola, D. Quero}

\headers{\ShortTitle}{\TheAuthors}

\title{{\TheTitle}\thanks{FB is supported by ANID through FONDECYT grant 3220254 and \EO{Programa Regional MATH-AmSud  AMSUD210013.} FF is supported by ANID through FONDECYT grant 3230126. EO is partially supported by ANID through FONDECYT grant 1220156 and \EO{Programa Regional MATH-AmSud  AMSUD210013}. DQ is supported by ANID through Subdirecci\'on del Capital Humano/Doctorado Nacional/2021--21210988 and by UTFSM through Programa de Incentivo a la Investigaci\'on Cient\'ifica (PIIC).}}

\author{Francisco Bersetche\thanks{Departamento de Matem\'atica, Universidad T\'ecnica Federico Santa Mar\'ia, Valpara\'iso, Chile and Departamento de Matem\'atica, FCEyN, Universidad de Buenos Aires \& IMAS CONICET,
Pabell\'on I, Ciudad Universitaria 1428, Buenos Aires, Argentina. \email{francisco.bersetche@usm.cl}}
\and
Francisco Fuica\thanks{Facultad de Matem\'aticas, Pontificia Universidad Cat\'olica de Chile, Avenida Vicu\~{n}a Mackenna 4860, Santiago, Chile.
\email{francisco.fuica@mat.uc.cl}}
\and
Enrique Ot\'arola\thanks{Departamento de Matem\'atica, Universidad T\'ecnica Federico Santa Mar\'ia, Valpara\'iso, Chile. \email{enrique.otarola@usm.cl}}
\and
Daniel Quero\thanks{Departamento de Matem\'atica, Universidad T\'ecnica Federico Santa Mar\'ia, Valpara\'iso, Chile \email{daniel.quero@alumnos.usm.cl}}}

\ifpdf
\hypersetup{
  pdftitle={\TheTitle},
  pdfauthor={\TheAuthors}
}
\fi

\date{Draft version of \today.}

\begin{document}

\maketitle
\begin{abstract}
We adopt the \emph{integral} definition of the fractional Laplace operator and study an optimal control problem on \emph{Lipschitz domains} that involves a fractional elliptic partial differential equation (PDE) as state equation and a control variable that enters the state equation as a coefficient; pointwise constraints on the control variable are considered as well. We establish the existence of optimal solutions and analyze first and, necessary and sufficient, second order optimality conditions. Regularity estimates for optimal variables are also analyzed. We develop two finite element discretization strategies: a semidiscrete scheme in which the control variable is not discretized, and a fully discrete scheme in which the control variable is discretized with piecewise constant functions. For both schemes, we analyze the convergence properties of discretizations and derive error estimates.
\end{abstract}


\section{Introduction}\label{sec:intro}
In this paper, we are interested in the analysis and discretization of an optimal control problem involving a fractional elliptic PDE as state equation and a control variable that enters the state equation as a coefficient. To make matters precise, we let $\Omega$ be an open and bounded domain in $\mathbb{R}^{d}$ $(d \geq 2)$ with Lipschitz boundary $\partial\Omega$. Given a desired state $u_{\Omega}\in L^{2}(\Omega)$ and a regularization parameter $\lambda>0$, we introduce the cost functional
\begin{equation}\label{def:cost_functional}
J(u,q):=\frac{1}{2}\|u - u_{\Omega}\|_{L^{2}(\Omega)}^{2}+\frac{\lambda}{2}\|q\|_{L^{2}(\Omega)}^{2}.
\end{equation} 
Let $f$ be a fixed function in $H^{-s}(\Omega)$. We shall be concerned with the following optimal control problem: Find $\min J(u,q)$ subject to the \emph{fractional} and \emph{elliptic} PDE
\begin{equation}\label{def:state_eq}
(-\Delta)^{s}u + qu = f \text{ in } \Omega,
\qquad
u=0  \text{ in } \Omega^{c},
\end{equation}
where $\Omega^{c} = \mathbb{R}^{d}\setminus\Omega$, and the \emph{control constraints} 
\begin{equation}\label{def:box_constraints}
q \in \mathbb{Q}_{ad},
\qquad
\mathbb{Q}_{ad}:=\{v \in L^{\infty}(\Omega):~a \leq v(x) \leq b~\mathrm{a.e.}~x\in \Omega \}.
\end{equation}
The control bounds $a$ and $b$ are such that $0 < a < b$. The operator $(-\Delta)^s$ corresponds to the \emph{integral definition} of the fractional Laplace operator; its definition can be found in section \ref{sec:the_frac_lap_operator}.

In problem \eqref{def:cost_functional}--\eqref{def:box_constraints} the control variable $q$ enters the state equation as a coefficient and not as a source term. This creates the \emph{nonlinear} coupling $qu$ in \eqref{def:state_eq}, the presence of which has led to this type of problems being referred to as \emph{bilinear optimal control problems} or \emph{control-affine problems} \cite{MR3878305}. Mathematically, the coupling $qu$ in \eqref{def:state_eq} complicates both the analysis and the discretization. In particular, the solution of the state equation depends nonlinearly on the control, so the uniqueness of solutions of \eqref{def:cost_functional}--\eqref{def:box_constraints} cannot be guaranteed \cite{MR2536007}; a complete optimization study therefore requires the analysis of second order optimality conditions \cite{MR3878305}. In terms of applications, it should be mentioned that several processes in engineering, biology, and ecology, to name just a few, can be modeled by bilinear systems; see, for instance, \cite{MR0414174,MR2986407,MR0332249}. We would also like to note that problem \eqref{def:cost_functional}--\eqref{def:box_constraints} can be interpreted as a particular instance of a \emph{coefficient identification problem}. Such problems become particularly relevant in scenarios where coefficients or source terms remain uncertain or unknown. It may be that we have a sparse and/or noisy measurement of the state of the system or an output of interest that we wish to match, and/or that a priori information about some model coefficients is available. In such cases, we can resort to solving a control problem or an inverse problem to recover the unknown parameters and define a more accurate, data-driven mathematical model.

Another important feature of \eqref{def:cost_functional}--\eqref{def:box_constraints} is that it is governed by the \emph{fractional} and \emph{nonlocal} operator $(-\Delta)^s$. Although the use of nonlocal models to describe natural and social phenomena has been of interest for a long time, this interest has only increased significantly in recent years. This is mainly due to the numerous applications of these models, which demonstrate their relevance and versatility. It is therefore only natural that an interest in efficient approximation schemes for nonlocal models arises \cite{MR3893441,MR4189291} and that one might be interested in their \emph{control} \cite{MR3158780,MR3429730,MR3504977}. By nonlocal models, we mean here model descriptions that at a given point in space depend on the state of the system at points at a far distance from that given point. An important example from the family of nonlocal operators is the \emph{integral} fractional Laplacian $(-\Delta)^s$ ($0<s<1$), which corresponds to the infinitesimal generator of a stable L\'evy process. It can be shown that $(-\Delta)^s$ converges in a suitable sense to the Laplace operator supplemented with homogeneous Dirichlet boundary conditions and to the identity operator when $s \uparrow 1$ and $s \downarrow 0$, respectively; see \cite{Hitchhikers} for details.

For the \emph{particular case} $s=1$, there are several works in the literature that provide error estimates for finite element discretizations  of \eqref{def:cost_functional}--\eqref{def:box_constraints}. To the best of our knowledge, the first work that provides an analysis for suitable finite element discretizations is \cite{MR2536007}. In this work, the authors propose two schemes that discretize the admissible control set with piecewise constant and piecewise linear functions, and provide bounds for the error committed within the approximation of a control variable \cite[Corollaries 5.6 and 5.10]{MR2536007}. These results were later extended to mixed and stabilized finite element methods in \cite{MR3103238} and \cite{MR3693332}, respectively. Recently, bilinear optimal control problems whose objective functionals do not depend on the controls have been analyzed in \cite{MR3878305}. In particular, the authors investigate sufficient second order conditions for bang--bang controls and derive error estimates in $L^1$-norms for a suitable finite element discretization. Regarding the analysis of \emph{a posteriori} error estimates for problem \eqref{def:cost_functional}--\eqref{def:box_constraints} with $s=1$, we refer the reader to \cite{MR2680928} and \cite{MR4450052}. We conclude this paragraph with a reference to the work \cite{MR2373479}, in which the authors provide upper bounds for discretization errors with respect to a cost functional and with respect to a given quantity of interest based on a posteriori error estimators.

To the best of our knowledge, the only work available in the literature that provides an advance on the \emph{parabolic} version of \eqref{def:cost_functional}--\eqref{def:box_constraints} is the very recent manuscript \cite{warma_2022}. In this paper, the authors provide an analysis for the continuous problem including the existence of solutions and first and second order optimality conditions; the second order sufficiency requires an additional assumption on the problem data (see also \cite[Remark 2.21]{MR2536007}). In contrast to this work and to the best of our knowledge, this exposition is the first to study approximation techniques for the optimal control problem \eqref{def:cost_functional}--\eqref{def:box_constraints}; discretization techniques for the problem where the control variable enters \eqref{def:state_eq} as a source term are available in \cite{MR3990191,MR4358465,MR4599045,fractional_sparse,MR4250571}. In the following, we list what we consider to be the most important contributions of our work:

\begin{itemize}
\item[(i)] 
\emph{Existence of optimal solutions}: We prove that the optimal control problem \eqref{def:cost_functional}--\eqref{def:box_constraints} admits at least one optimal solution; see Theorem \ref{thm:existence_opt_sol}.
\item[(ii)] 
\emph{Optimality conditions}: We derive first and necessary and sufficient second order optimality conditions with a minimal gap; see \S \ref{sec:1st_opt_condition} and \S \ref{sec:2nd_order_opt_cond}.
\item[(iii)] \emph{Regularity estimates}: We analyze regularity properties for optimal variables. In particular, we prove that $\bar u, \bar p, \bar q \in H^{s+\kappa -\epsilon}(\Omega)$; see Theorem \ref{thm:regul_control} for details.
\item[(iv)] \emph{Finite element discretizations}: We propose two different strategies: a semidiscrete scheme, where the control set is not discretized, and a fully discrete scheme, where such a set is discretized with piecewise constant functions.
\item[(v)] \emph{Convergence of discretizations}: We prove the existence of subsequences of discrete global solutions that converge to global solutions of \eqref{def:cost_functional}--\eqref{def:box_constraints}. We also prove that continuous strict local solutions can be approximated by local minima of discrete problems; see section \ref{sec:convergence_fullydiscrete} and section \ref{sec:convergence_semidiscrete}.
\item[(vi)] \emph{Error estimates}: For each discretization scheme, we provide estimates for the error that occurs when approximating optimal variables; see \S\ref{sec:error_estimates}. To obtain these results, we have assumed that solutions to suitable finite element discretizations of \eqref{def:state_eq} are uniformly bounded in $L^\infty(\Omega)$.
\end{itemize}

We structure our presentation as follows. In section \ref{sec:notation_and_prel}, we introduce the notation and collect some known facts that shall be useful for our purposes. In section \ref{sec:the_state_eq}, we give an overview of regularity and finite element approximation results for problem \eqref{def:state_eq}.  Section \ref{sec:the_ocp} is our first original contribution. We study a weak version of the control problem \eqref{def:cost_functional}--\eqref{def:box_constraints}. More precisely, we show the \emph{existence} of optimal solutions and derive \emph{first} and \emph{second} order optimality conditions. The study of numerical schemes begins in section \ref{sec:fem_control_problem}, where we develop two finite element schemes for \eqref{def:cost_functional}--\eqref{def:box_constraints} and analyze convergence properties. In section \ref{sec:error_estimates}, we derive error bounds. We conclude with section \ref{sec:numerical_exp}, where we provide several numerical examples to illustrate our theory.


\section{Notation and preliminary remarks}
\label{sec:notation_and_prel}

Let us establish the notation and recall some facts that will be useful later.
 
\subsection{Notation}
In the course of this work, let $d \geq 2$ and $\Omega \subset \mathbb{R}^d$ be an open and bounded domain with Lipschitz boundary $\partial \Omega$. We denote by $\Omega^{c}$ the complement of $\Omega$. If $\mathscr{X}$ and $\mathscr{Y}$ are Banach function spaces, we write $\mathscr{X}\hookrightarrow \mathscr{Y}$ to denote that $\mathscr{X}$ is continuously embedded in $\mathscr{Y}$. We denote by $\mathscr{X}'$ and $\| \cdot \|_{\mathscr{X}}$ the dual and the norm of $\mathscr{X}$, respectively. We denote by $\langle \cdot,\cdot \rangle_{\mathscr{X}',\mathscr{X}}$ the duality pairing between $\mathscr{X}'$ and $\mathscr{X}$ and simply write $\langle \cdot,\cdot \rangle$ if the spaces $\mathscr{X}'$ and $\mathscr{X}$ are clear from the context. The relation $\mathfrak{a} \lesssim \mathfrak{b}$ indicates that $\mathfrak{a} \leq C \mathfrak{b}$, with a positive constant $C$ that does not depend on $\mathfrak{a}$, $\mathfrak{b}$, or the discretization parameters, but may depend on $s$, $d$, and $\Omega$. The value of $C$ might change at each occurrence.


\subsection{Function spaces}\label{sec:function_spaces}
 
Fractional Sobolev spaces provide a natural framework for analyzing the state equation \eqref{def:state_eq}. A family of fractional Sobolev spaces can be defined based on the Fourier transform $\mathcal{F}$: For any $s \geq 0$, we define $H^s(\mathbb{R}^d)$, a Sobolev space of order $s$ over $\mathbb{R}^{d}$, by \cite[Definition 15.7]{MR2328004}, \cite[Chapter 1, Section 7]{MR0350178}
\begin{equation*}
H^{s}(\mathbb{R}^{d}) := \{ v \in L^2(\mathbb{R}^{d}) :  (1 + |\xi|^2)^{\frac{s}{2}} \mathcal{F}(v) \in  L^2(\mathbb{R}^{d})\},
\end{equation*}
endowed with the norm $\|v\|_{H^{s}(\mathbb{R}^{d})}:= \|(1 + |\xi|^2)^{\frac{s}{2}} \mathcal{F}(v)\|_{L^2(\mathbb{R}^{d})}$. We define $\tilde H^s(\Omega)$ as the closure of $C_0^{\infty}(\Omega)$ in $H^{s}(\mathbb{R}^{d})$ \cite[page 77]{MR1742312} and note that it can be equivalently characterized as the following space of zero-extension functions \cite[Theorem 3.29]{MR1742312}:
\begin{equation*}
 \tilde{H}^{s}(\Omega)=\{v|_{\Omega}^{} : v\in H^{s}(\mathbb{R}^{d}), \text{ supp }v\subset\bar{\Omega}\}.
\end{equation*}
We endow the space $\tilde{H}^{s}(\Omega)$ with the following inner product and norm \cite[page 75]{MR1742312}:
\begin{equation*}
\label{eq:inner_product}
( v, w )_{\tilde{H}^{s}(\Omega)} :=\int_{\mathbb{R}^{d}}\int_{\mathbb{R}^{d}}\frac{(v(x) - v(y))(w(x) - w(y))}{|x - y|^{d+2s}}\mathrm{d}x\mathrm{d}y,
\quad
| v |_{\tilde{H}^s(\Omega)}:= ( v, v )^{\frac{1}{2}}_{\tilde{H}^{s}(\Omega)}.
\end{equation*}
We also introduce $H^{-s}(\Omega)$ as the dual space of the fractional Sobolev space $\tilde{H}^{s}(\Omega)$.

Let us review a continuity of the product property in fractional Sobolev spaces.

\begin{lemma}[continuity of the product]
Let $\mathfrak{t} \in (0,\infty)$ and let $\varphi,\phi \in H^{\mathfrak{t}}(\mathbb{R}^{d})\cap L^{\infty}(\mathbb{R}^{d})$. Then, the product $\varphi\phi$ belongs to the space $H^{\mathfrak{t}}(\mathbb{R}^{d})\cap L^{\infty}(\mathbb{R}^{d})$.
\label{lemma:product_of_func}
\end{lemma}
\begin{proof}
The result follows from a direct application of the Runst-Sickel lemma \cite[Lemma 4.1]{MR1877265} (\cite[Section 5.3.7]{MR1419319}) with $s = \mathfrak{t}$, $p_{1} = p_{2} = p = q = 2$, and $r_{1} = r_{2} = \infty$. We note that, as stated in \cite[page 390]{MR1877265}, the Triebel-Lizorkin space $\tilde{\mathcal{F}}_{2,2}^{\mathfrak{t}}$ coincides with
$H^\mathfrak{t}(\mathbb{R}^{d})$; see also \cite[Chapter 2, Section 2.3.5]{MR781540}.
\end{proof}

We conclude this section with the following Sobolev embedding results.

\begin{lemma}[embedding results]\label{lemma:embedding_result}
Let $s\in(0,1)$. If $\mathfrak{r}\in [1,2d/(d - 2s)]$, then $H^{s}(\Omega)\hookrightarrow L^{\mathfrak{r}}(\Omega)$. If $\mathfrak{r}\in [1,2d/(d - 2s))$, then $H^{s}(\Omega)\hookrightarrow L^{\mathfrak{r}}(\Omega)$ is compact.
\end{lemma}
\begin{proof}
A proof of $H^{s}(\Omega)\hookrightarrow L^{\mathfrak{r}}(\Omega)$ can be found in \cite[Theorem 7.34]{MR2424078}. The fact that the embedding is compact for $\mathfrak{r}<2d/(d - 2s)$ follows from \cite[Corollary 7.2]{Hitchhikers}.
\end{proof}


\subsection{The fractional Laplace operator}
\label{sec:the_frac_lap_operator}

For $s \in (0,1)$ and smooth functions $w:\mathbb{R}^{d} \to \mathbb{R}$, there are several equivalent definitions of $(-\Delta)^s$ in $\mathbb{R}^d$. In fact, $(-\Delta)^{s}$ can be naturally defined by means of the following pointwise formula:
\begin{equation}\label{eq:pointwise_formula}
(-\Delta)^{s}w(x):= C(d,s)\mathrm{p.v.}\int_{\mathbb{R}^{d}}\frac{w(x) - w(y)}{|x - y|^{d+2s}} \mathrm{d}y,\quad C(d,s):=\frac{2^{2s}s\Gamma(s+\frac{d}{2})}{\pi^{\frac{d}{2}}\Gamma(1-s)},
\end{equation}
where p.v.~stands for the \emph{Cauchy principal value} and $C(d,s)$ is a normalization constant. The constant $C(d,s)$ is introduced to ensure that the definition \eqref{eq:pointwise_formula} is equivalent to the following one via Fourier transform: $\mathcal{F}((-\Delta)^{s}w)(\xi)=|\xi|^{2s}\mathcal{F}(w)(\xi)$ for all $\xi \in \mathbb{R}^d$; see \cite[chapter 1, section 1]{MR0350027} for details. In addition to these two definitions, there are several other equivalent definitions of $(-\Delta)^{s}$ in $\mathbb{R}^{d}$ in the literature. For a discussion, we refer the reader to \cite{MR3613319}.

In bounded domains, for functions supported in $\bar{\Omega}$, we may use the integral representation \eqref{eq:pointwise_formula} to define $(-\Delta)^{s}$. This gives rise to the so-called \emph{restricted} or \emph{integral} fractional Laplacian, which, from now on, we shall simply refer to as the \emph{integral fractional Laplacian}. Note that we have materialized a zero Dirichlet condition by restricting the operator to acting only on functions that are zero outside $\Omega$.

To present suitable weak formulations for problems involving $(-\Delta)^{s}$, we define 
\begin{equation*}
\mathcal{A}:\tilde{H}^{s}(\Omega)\times\tilde{H}^{s}(\Omega) \to \mathbb{R},
\qquad
\mathcal{A}(v,w):=\tfrac{C(d,s)}{2}(v,w)_{\tilde{H}^s(\Omega)}.
\end{equation*}
We note that $\mathcal{A}$ is just a multiple of the inner product in $\tilde{H}^s(\Omega)$ introduced in section \ref{sec:function_spaces}; $\mathcal{A}$ is bilinear and bounded. We denote by $\|\cdot\|_s$ the norm induced by $\mathcal{A}$: 
\[
 \|v\|_{s}:=\sqrt{\mathcal{A}(v,v)}=\mathfrak{C}(d,s)|v|_{\tilde{H}^s(\Omega)},
 \qquad 
 \mathfrak{C}(d,s)= \sqrt{C(d,s)/2}.
\]


\section{The state equation}
\label{sec:the_state_eq}

In this section, we present a suitable weak formulation for \eqref{def:state_eq} and give a brief overview of results concerning the well-posedness of such a formulation, regularity estimates for its solution, and finite element approximations.


\subsection{Weak formulation}
\label{sec:weak_formulation_st_eq}

Let $s\in(0,1)$, let $\mathfrak{f}\in H^{-s}(\Omega)$ be a given forcing term, and let $\mathfrak{q}$ be an arbitrary element in $\mathbb{Q}_{ad}$. Under these conditions, we introduce the following weak formulation of the state equation \eqref{def:state_eq}: Find $\mathfrak{u} \in \tilde{H}^{s}(\Omega)$ such that
\begin{equation}\label{eq:weak_aux_eq}
\mathcal{A}(\mathfrak{u},v) + (\mathfrak{q}\mathfrak{u},v)_{L^2(\Omega)} = \langle \mathfrak{f},v \rangle \quad \forall v\in \tilde{H}^{s}(\Omega).
\end{equation}
We note that, since $\mathfrak{q} \in \mathbb{Q}_{ad} \subset L^{\infty}(\Omega)$ and $\mathfrak{u},v\in \tilde{H}^{s}(\Omega) \subset L^2(\Omega)$, all terms involved in \eqref{eq:weak_aux_eq} are well-defined. The well-posedness of problem \eqref{eq:weak_aux_eq} follows from the Lax-Milgram lemma. In particular, we have the \emph{stability bound} $\|\mathfrak{u}\|_{s} \lesssim \|\mathfrak{f}\|_{H^{-s}(\Omega)}$.


\subsection{Regularity estimates}
\label{sec:reg_st_eq}

We present the following regularity result.

\begin{theorem}[Sobolev regularity]\label{thm:sobolev_reg}
Let $s\in(0,1)$, $\mathfrak{f}\in L^2(\Omega)$, and $\mathfrak{q} \in \mathbb{Q}_{ad}$. Then, the solution $\mathfrak{u}$ to problem \eqref{eq:weak_aux_eq} belongs to $H^{s + \kappa -\varepsilon}(\Omega)$ for all $0 < \varepsilon  < s$, where $\kappa = \frac{1}{2}$ for $\frac{1}{2} < s < 1$ and $\kappa = s - \varepsilon$ for $0 < s \leq \frac{1}{2}$. In addition, we have the bound
\begin{equation*}\label{eq:estimate_frac_Lap}
\|\mathfrak{u}\|_{H^{s + \kappa -\varepsilon}(\Omega)} 
\leq
C\varepsilon^{-\nu}\|\mathfrak{f}\|_{L^2(\Omega)} \quad \forall \varepsilon \in (0,s),
\end{equation*}
where $\nu = \frac{1}{2}$ for $\frac{1}{2} < s < 1$ and $\nu = \frac{1}{2} + \nu_{0}$ for $0 < s \leq \frac{1}{2}$. Here, $\nu_0$ and $C$ denote positive constants that depend on $\Omega$ and $d$ and $\Omega$, $d$, $s$, and $\mathfrak{q}$, respectively.
\end{theorem}
\begin{proof}
Since $\Omega$ is Lipschitz and $\mathfrak{f} - \mathfrak{q}\mathfrak{u} \in L^2(\Omega)$, the proof follows immediately as an application of \cite[Theorem 2.1 and inequality (2.6)]{MR4283703}; see also \cite[Remark 6]{MR4530901}.
\end{proof}

If we assume a higher integrability assumption for $\mathfrak{f}$, it is possible to obtain an $L^{\infty}(\Omega)$-regularity result for the solution $\mathfrak{u}$ of problem \eqref{eq:weak_aux_eq}.

\begin{theorem}[$L^{\infty}(\Omega)$--regularity]\label{thm:L_infty_reg}
Let $s\in(0,1)$ and $r > d/2s$. If $\mathfrak{f}\in L^{r}(\Omega)$, then the solution $\mathfrak{u}$ to problem \eqref{eq:weak_aux_eq} belongs to $\tilde{H}^{s}(\Omega)\cap L^{\infty}(\Omega)$ and
\begin{equation*}\label{eq:estimate_L_infty}
\|\mathfrak{u}\|_{s} + \|\mathfrak{u}\|_{L^{\infty}(\Omega)} \lesssim \|\mathfrak{f}\|_{L^r(\Omega)},
\end{equation*}
with a hidden constant that is independent of $\mathfrak{u}$ and the problem data.
\end{theorem}
\begin{proof}
See \cite[Theorem 3.1]{MR4358465}.
\end{proof}


\subsection{Finite element approximation}\label{sec:fem}

Under the additional assumption that $\Omega$ is a Lipschitz polytope, we now introduce a finite element-like scheme to approximate the solution to \eqref{eq:weak_aux_eq}. For this purpose, we assume that we have at hand a collection of conforming and quasi-uniform meshes $\{ \mathscr{T}_h\}_{h>0}$ of $\bar{\Omega}$ made of closed simplices. We denote by $h :=\max\{ h_T: T \in \mathscr{T}_h \}$ the mesh-size of $\T_h = \{ T \}$, where $h_T := \text{diam}(T)$.

Given a mesh $\mathscr{T}_{h}$, we introduce the finite element space 
\begin{equation}\label{def:piecewise_linear_set}
\mathbb{V}_{h}:=\{v_{h}\in C(\bar{\Omega}): v_{h}|_T\in \mathbb{P}_{1}(T) \ \forall T\in \T_{h}, v_h = 0 \text{ on } \partial\Omega\}.
\end{equation}
The following comments are now appropriate. First, for every $s \in (0,1)$, $\mathbb{V}_{h} \subset\tilde{H}^s(\Omega)$. Second, we enforce a classical homogeneous Dirichlet boundary condition at $\partial \Omega$. Note that discrete functions are trivially extended by zero to $\Omega^c$.

With $\mathbb{V}_h$ at hand, we introduce an approximation of the solution to \eqref{eq:weak_aux_eq}: 
\begin{equation}\label{eq:weak_aux_eq_discrete}
\mathfrak{u}_{h} \in \mathbb{V}_{h}: 
\quad
\mathcal{A}(\mathfrak{u}_{h},v_{h}) + (\mathfrak{q}\mathfrak{u}_{h},v_{h})_{L^2(\Omega)} = \langle \mathfrak{f},v_{h} \rangle \quad \forall v_h \in \mathbb{V}_{h}.
\end{equation}
The existence and uniqueness of $\mathfrak{u}_{h} \in \mathbb{V}_{h}$ follows from the Lax-Milgram lemma. In particular, for every $h>0$, we have the \emph{discrete  stability bound} $\|\mathfrak{u}_{h}\|_{s} \lesssim \|\mathfrak{f}\|_{H^{-s}(\Omega)}$.

We present the following a priori error estimates.

\begin{theorem}[error estimates]\label{thm:error_estimates_frac_Lap}
Let $s\in(0,1)$, $\mathfrak{f}\in L^{2}(\Omega)$, and $\mathfrak{q}\in\mathbb{Q}_{ad}$. Let $\Omega$ be a Lipschitz polytope. Let $\mathfrak{u} \in \tilde{H}^{s}(\Omega)$ be the solution to \eqref{eq:weak_aux_eq} and let $\mathfrak{u}_{h} \in \mathbb{V}_{h}$ be its finite element approximation obtained as in \eqref{eq:weak_aux_eq_discrete}. Then, we have the error bounds
\begin{align}
\label{eq:error_in_norm_s}
\|\mathfrak{u} - \mathfrak{u}_{h}\|_{s} 
& \lesssim
h^{\gamma}|\log h|^{\varphi}\|\mathfrak{f}\|_{L^2(\Omega)},
\quad
\gamma = \min\{s,\tfrac{1}{2}\},
\\
\label{eq:error_in_norm_L2}
\|\mathfrak{u} - \mathfrak{u}_{h}\|_{L^2(\Omega)} 
& \lesssim
h^{2\gamma}|\log h|^{2\varphi}\|\mathfrak{f}\|_{L^2(\Omega)},
\end{align}
where $\varphi = \nu$ if $s\neq \frac{1}{2}$, $\varphi = 1 +\nu$ if $s=\frac{1}{2}$, and $\nu \geq \frac{1}{2}$ is the constant in Theorem \ref{thm:sobolev_reg}. 
\end{theorem}
\begin{proof}
The proof follows from the regularity estimates for $\mathfrak{u}$ given in Theorem \ref{thm:sobolev_reg} and the arguments developed in the proofs of \cite[Theorem 3.5]{MR4283703} and \cite[Proposition 3.8]{MR4283703}. For the sake of brevity, we omit the details.
\end{proof}


\section{The optimal control problem}\label{sec:the_ocp}

In this section, we analyze the following weak formulation of the optimal control problem \eqref{def:cost_functional}--\eqref{def:box_constraints}: Find
\begin{equation}\label{eq:weak_min_problem}
\min\{ J(u,q): (u,q) \in  \tilde{H}^{s}(\Omega ) \times  \mathbb{Q}_{ad}\} 
\end{equation}
subject to the \emph{fractional} and \emph{elliptic} state equation
\begin{equation}\label{eq:weak_st_eq}
\mathcal{A}(u,v) + (qu,v)_{L^2(\Omega)} = (f,v)_{L^2(\Omega)}  \quad \forall v\in \tilde{H}^{s}(\Omega),
\end{equation}
where $\mathbb{Q}_{ad}$ is defined in \eqref{def:box_constraints} and $f \in L^2(\Omega)$.


\subsection{Existence of optimal controls}\label{sec:existence_of_opt_cont}
The existence of at least one optimal solution follows from the direct method of calculus of variations \cite[Chapter 1]{MR1201152}. 

\begin{theorem}[existence of an optimal solution]\label{thm:existence_opt_sol}
The optimal control problem \eqref{eq:weak_min_problem}--\eqref{eq:weak_st_eq} admits at least one global solution $(\bar{u},\bar{q})\in \tilde{H}^{s}(\Omega)\times \mathbb{Q}_{ad}$.
\end{theorem}
\begin{proof}
Let $\mathfrak{i}:=\inf\{ J(u,q):(u,q)\in \tilde{H}^{s}(\Omega) \times \mathbb{Q}_{ad}\}$. Let, for $k \in \mathbb{N}$, $q_{k} \in \mathbb{Q}_{ad}$ and $u_k \in \tilde{H}^s(\Omega)$ be such that $J(u_{k},q_{k}) \rightarrow \mathfrak{i}$ as $k\uparrow \infty$ and
\begin{equation}\label{eq:state_eq_with_k}
\mathcal{A}(u_{k},v) + (q_{k}u_{k},v)_{L^2(\Omega)} = (f,v)_{L^2(\Omega)} \quad \forall v\in \tilde{H}^{s}(\Omega),
\end{equation}
i.e., $\{ (u_k,q_k) \}_{k \in \mathbb{N}} \subset \tilde{H}^s(\Omega) \times \mathbb{Q}_{ad}$ is a minimizing sequence. Let us now invoke the fact that  $\mathbb{Q}_{ad}$ is bounded in $L^{\infty}(\Omega)$ to deduce the existence of a nonrelabeled subsequence $\{q_{k}\}_{k\in\mathbb{N}}\subset\mathbb{Q}_{ad}$ such that $q_k \mathrel{\ensurestackMath{\stackon[1pt]{\rightharpoonup}{\scriptstyle\ast}}} \bar{q}$ in $L^{\infty}(\Omega)$ as $k \uparrow \infty$; $\bar{q} \in \mathbb{Q}_{ad}$. On the other hand, since for every $k \in \mathbb{N}$ $q_k \in \mathbb{Q}_{ad}$, the well-posedness of \eqref{eq:state_eq_with_k} reveals that $\{u_{k}\}_{k\in\mathbb{N}}$ is uniformly bounded in $\tilde{H}^{s}(\Omega)$. More precisely: $\| u_k \|_s \lesssim \| f \|_{L^2(\Omega)}$ for every $k \in \mathbb{N}$. We can thus conclude the existence of a nonrelabeled subsequence such that $u_{k}\rightharpoonup \bar{u}$ in $\tilde{H}^{s}(\Omega)$ as $k\uparrow \infty$; $\bar{u}$ being the natural candidate for an optimal state. We now prove that $\bar{u}$ solves problem \eqref{eq:weak_st_eq}, where $q$ is replaced by $\bar{q}$, and that $(\bar{u},\bar{q})$ is optimal.

Since $u_{k}\rightharpoonup \bar{u}$ in $\tilde{H}^{s}(\Omega)$ as $k\uparrow \infty$, we immediately obtain that $\mathcal{A}(u_{k},v) \to \mathcal{A}(\bar{u},v)$ as $k\uparrow \infty$ for every $v \in \tilde{H}^{s}(\Omega)$. To analyze the convergence of the bilinear term, we utilize that $u_{k}\rightharpoonup \bar{u}$ in $\tilde{H}^{s}(\Omega)$ as $k\uparrow \infty$, the compact embedding $H^{s}(\Omega)\hookrightarrow L^{2}(\Omega)$ of Lemma \ref{lemma:embedding_result}, which holds because $2d/(d-2s) > 2$, the convergence $q_k \mathrel{\ensurestackMath{\stackon[1pt]{\rightharpoonup}{\scriptstyle\ast}}} \bar{q}$ in $L^{\infty}(\Omega)$ as $k \uparrow \infty$, and the fact that $\bar{u}v\in L^1(\Omega)$. These arguments allow us to obtain
\begin{multline*}
|(\bar{q}\bar{u},v)_{L^2(\Omega)} - (q_{k}u_{k},v)_{L^2(\Omega)}| \leq | ([\bar{q}-q_{k}]\bar{u},v)_{L^2(\Omega)}| + |(q_{k}[\bar{u}-u_{k}],v)_{L^2(\Omega)}|\\
\leq | (\bar{q}-q_{k},\bar{u}v)_{L^2(\Omega)}| + b\|\bar{u}-u_{k}\|_{L^2(\Omega)}\|v\|_{L^2(\Omega)} \rightarrow 0, \quad k\uparrow \infty.
\end{multline*}

Finally, we prove that $(\bar{u},\bar{q})$ is optimal. To do so, we utilize the strong convergence $u_{k}\rightarrow \bar{u}$ in $L^2(\Omega)$ as $k\uparrow \infty$ and the fact that the square of $\| \cdot \|_{L^2(\Omega)}$ is weakly lower semicontinuous in $L^2(\Omega)$. With these results, we can therefore conclude that $J(\bar{u},\bar{q}) \leq \liminf_{k}J(u_{k},q_{k}) = \mathfrak{i}$. Consequently, $(\bar{u}, \bar{q})$ is optimal.
\end{proof}

Since the control problem \eqref{eq:weak_min_problem}--\eqref{eq:weak_st_eq} is not convex, in the following analysis we will discuss optimality conditions in the context of local solutions in $L^{2}(\Omega)$ \cite[page 207]{Troltzsch}: We say that $\bar{q} \in \mathbb{Q}_{ad}$ is \emph{locally optimal} in the sense of $L^{2}(\Omega)$ for \eqref{eq:weak_min_problem}--\eqref{eq:weak_st_eq} if there exists $\eta > 0$ such that $J(\bar{u},\bar{q}) \leq J(u,q)$ for all $(u,q) \in \tilde{H}^{s}(\Omega) \times \mathbb{Q}_{ad}$ such that $\|q - \bar{q}\|_{L^{2}(\Omega)} \leq \eta$.  Here, $u$ solves \eqref{eq:weak_st_eq} and $\bar{u}$ solves \eqref{eq:weak_st_eq} with $q$ replaced by $\bar{q}$.


\subsection{First order optimality conditions}\label{sec:1st_opt_condition}
In this section we formulate first order necessary optimality conditions. We begin our analysis by introducing the set
\begin{equation*}
\mathcal{Q}:=\{ q\in L^{\infty}(\Omega): \exists c > 0 \text{ such that } q(x) > c > 0 \text{ for a.e.}~x \in \Omega\}
\end{equation*}
\cite[page 783]{MR2536007} and the control to state map $\mathcal{S}: \mathcal{Q} \to \tilde{H}^{s}(\Omega)$, which given a control $q$ associates to it the unique state $u=\mathcal{S}q$ that solves \eqref{eq:weak_st_eq}. 

The following result provides differentiability properties for the operator $\mathcal{S}$.

\begin{theorem}[differentiability properties of $\mathcal{S}$]
\label{thm:properties_C_to_S}
The control to state operator $\mathcal{S}: \mathcal{Q} \rightarrow \tilde{H}^{s}(\Omega)$ is of class $C^2$ with respect to the $L^{\infty}(\Omega)$--topology. In addition, if $w\in L^{\infty}(\Omega)$, then $z=\mathcal{S}'(q)w\in \tilde{H}^{s}(\Omega)$ corresponds to the unique solution to 
\begin{equation}\label{eq:first_der_S}
\mathcal{A}(z,v) + (qz,v)_{L^2(\Omega)} = -(wu,v)_{L^2(\Omega)} \quad \forall v\in \tilde{H}^{s}(\Omega),
\end{equation}
where $u = \mathcal{S}q$. Moreover, if $w_1,w_2\in L^{\infty}(\Omega)$, then $\mathfrak{z}=\mathcal{S}''(q)(w_1,w_2)\in \tilde{H}^{s}(\Omega)$ is the unique solution to 
\begin{equation}\label{eq:second_der_S}
\mathcal{A}(\mathfrak{z},v) + (q\mathfrak{z},v)_{L^2(\Omega)} = -(w_{2}z_{w_{1}},v)_{L^2(\Omega)} - (w_{1}z_{w_{2}},v)_{L^2(\Omega)} \quad \forall v\in \tilde{H}^{s}(\Omega),
\end{equation}
where $z_{w_{i}}=\mathcal{S}'(q)w_i$ and $i \in \{ 1,2 \}$.
\end{theorem}
\begin{proof} 
Let $q \in \mathcal{Q}$. To prove the first order Fr\'echet differentiability of $\mathcal{S}$ we proceed as in \cite[Theorem 4.17]{Troltzsch} and show the existence of a continuous linear operator $D: L^{\infty}(\Omega) \rightarrow \tilde{H}^{s}(\Omega)$ and a mapping $r(q,\cdot): L^{\infty}(\Omega) \rightarrow \tilde{H}^{s}(\Omega)$ such that for all $w \in L^{\infty}(\Omega)$ satisfying $q + w \in \mathcal{Q}$, we have $\mathcal{S}(q+w) - \mathcal{S}(q) = D(w) + r(q,w)$ and
\begin{equation}\label{eq:property_residual}
\|r(q,w)\|_{s}/\|w\|_{L^{\infty}(\Omega)} \to 0
\end{equation}
as $\|w\|_{L^{\infty}(\Omega)} \to 0$. We immediately note that the mapping $w \mapsto z$ defined by the problem \eqref{eq:first_der_S} is linear and continuous. On the other hand, we define $\tilde{u}:=\mathcal{S}(q+w)$ and $u:=\mathcal{S}(q)$. It is therefore sufficient to prove that $r= \tilde{u} - u - z$ satisfies \eqref{eq:property_residual}. To do so, we first note that $r$ solves uniquely the problem: Find $r \in \tilde{H}^{s}(\Omega)$ such that
\begin{equation*}
\mathcal{A}(r,v) + (q r,v)_{L^2(\Omega)} = -(w[\tilde{u} - u],v)_{L^2(\Omega)} \quad \forall v\in \tilde{H}^{s}(\Omega).
\end{equation*}
A basic stability bound for this problem yields the estimate
\begin{equation*}
\|r\|_{s} 
\lesssim 
\|\tilde{u} - u\|_{L^2(\Omega)}\|w\|_{L^{\infty}(\Omega)} 
\lesssim
\|\tilde{u}\|_{L^2(\Omega)}\|w\|_{L^{\infty}(\Omega)}^2 
\lesssim
\|f\|_{L^2(\Omega)}\|w\|_{L^{\infty}(\Omega)}^2, 
\end{equation*}
where we used a stability estimate for the problem that $\tilde{u} - u$ solves and $\|\tilde{u}\|_{s} \lesssim \|f\|_{L^2(\Omega)}$. Consequently, $r$ satisfies \eqref{eq:property_residual}. This proves that $\mathcal{S}$ is Fr\'echet differentiable.

The second order Fr\'echet differentiability of $\mathcal{S}$ follows similar considerations. The fact that $\mathfrak{z}$ solves the problem \eqref{eq:second_der_S} follows from the arguments in \cite[Theorem 4.24(ii)]{Troltzsch}. Note that the problem \eqref{eq:second_der_S} is well posed because $-w_{2}z_{w_{1}} - w_{1}z_{w_{2}}\in L^2(\Omega)$.
\end{proof}

In order to provide first order optimality conditions, we introduce the reduced cost functional $j: \mathcal{Q} \to \mathbb{R}$ by $j(q)=J(\mathcal{S}q,q)$ and present the following basic result: If $\bar{q}\in \mathbb{Q}_{ad}$ denotes a locally optimal control for \eqref{eq:weak_min_problem}--\eqref{eq:weak_st_eq}, then \cite[Lemma 4.18]{Troltzsch}
\begin{equation}
\label{eq:variational_inequality}
j'(\bar{q}) (q - \bar{q}) \geq 0 \qquad \forall q \in \mathbb{Q}_{ad}.
\end{equation}
In \eqref{eq:variational_inequality}, $j'(\bar{q})$ denotes the Gate\^aux derivative of $j$ at $\bar{q}$. To investigate \eqref{eq:variational_inequality}, we introduce the \emph{adjoint variable} $p \in \tilde{H}^{s}(\Omega)$ as the unique solution to the \emph{adjoint equation}
\begin{equation}\label{eq:adj_eq}
\mathcal{A}(v,p) + (qp,v)_{L^2(\Omega)} = (u - u_{\Omega},v)_{L^2(\Omega)} \quad \forall v\in \tilde{H}^{s}(\Omega).
\end{equation}
Here, $u = \mathcal{S}q$. The well-posedness of \eqref{eq:adj_eq} follows from the Lax-Milgram lemma. In particular, we have the \emph{stability bound}  $\| p \|_s \lesssim \| f \|_{L^2(\Omega)}+\| u_{\Omega} \|_{L^2(\Omega)}$.

We now have all the ingredients to present first order optimality conditions.

\begin{theorem}[first order necessary optimality conditions] Every locally optimal control $\bar{q}\in\mathbb{Q}_{ad}$ for problem \eqref{eq:weak_min_problem}--\eqref{eq:weak_st_eq} satisfies the variational inequality
\begin{equation}\label{eq:var_ineq_with_adj_state}
(\lambda\bar{q} - \bar{u}\bar{p}, q - \bar{q})_{L^2(\Omega)} \geq 0 \qquad \forall q\in \mathbb{Q}_{ad},
\end{equation}
where $\bar{p}\in \tilde{H}^{s}(\Omega)$ solves \eqref{eq:adj_eq} with $q$ and $u$ replaced by $\bar{q}$ and $\bar{u} = \mathcal{S}\bar{q}$, respectively.
\end{theorem}
\begin{proof}
We begin the proof with simple calculations that show that the variational inequality
\eqref{eq:variational_inequality} can be rewritten as follows:
\begin{equation}\label{eq:var_ineq_rewritten}
(\bar{u} - u_{\Omega},\mathcal{S}'(\bar{q})(q - \bar{q}))_{L^2(\Omega)} + \lambda(\bar{q}, q - \bar{q})_{L^2(\Omega)} \geq 0 \qquad \forall q \in \mathbb{Q}_{ad},
\end{equation}
where $\bar{u}=\mathcal{S}\bar{q}$. Since the second term on the left-hand side of \eqref{eq:var_ineq_rewritten} is already contained in the desired inequality \eqref{eq:var_ineq_with_adj_state}, we thus concentrate on the first term. Define $z:=\mathcal{S}'(\bar{q})(q - \bar{q})$ and observe that $z$ solves problem \eqref{eq:first_der_S} with $q$ replaced by $\bar{q}$ and $w = q - \bar{q}$. Setting $v=\bar{p}$ in this problem yields
\begin{equation}\label{eq:z_and_p_equation_II}
\mathcal{A}(z,\bar{p}) + (\bar{q}z,\bar{p})_{L^2(\Omega)} = -((q - \bar{q})\bar{u},\bar{p})_{L^2(\Omega)},
\end{equation}
where $\bar{u} = \mathcal{S} \bar{q}$. On the other hand, we set $v = z$ as a test function in \eqref{eq:adj_eq} to obtain
\begin{equation}\label{eq:z_and_p_equation_I}
\mathcal{A}(z,\bar{p}) + (\bar{q}\bar{p},z)_{L^2(\Omega)} = (\bar{u} - u_{\Omega},z)_{L^2(\Omega)}.
\end{equation}
The desired variational inequality \eqref{eq:var_ineq_with_adj_state} thus follows from \eqref{eq:var_ineq_rewritten}, \eqref{eq:z_and_p_equation_II}, and \eqref{eq:z_and_p_equation_I}.
\end{proof}

The following formula is essential for the derivation of regularity estimates: If $\bar{q}$ is a locally optimal control for the problem \eqref{eq:weak_min_problem}--\eqref{eq:weak_st_eq}, then \cite[section 4.6.1]{Troltzsch}
\begin{equation}\label{eq:projection_control} 
\bar{q}(x):=\Pi_{[a,b]}(\lambda^{-1}\bar{u}(x)\bar{p}(x)) \textrm{ a.e.}~x \in \Omega,
\end{equation}
where $\Pi_{[a,b]} : L^1(\Omega) \rightarrow  \mathbb{Q}_{ad}$ is defined by $\Pi_{[a,b]}(v) := \min\{ b, \max\{ v, a\} \}$ a.e.~in  $\Omega$. 

We conclude this section with a regularity result for an optimal control variable, which will be important for the analysis of the schemes in \S \ref{sec:fem_control_problem}. To present it, we define
\begin{equation}
\label{eq:Lambda}
\Lambda(f,u_{\Omega}):= 
 \|f\|_{L^2(\Omega)} 
 \left( 
 \|f\|_{L^r(\Omega)} + \|u_{\Omega}\|_{L^r(\Omega)}
 \right)
 +
 \| f \|_{L^r(\Omega)} \| u_{\Omega} \|_{L^2(\Omega)}.
\end{equation}

\begin{theorem}[regularity of $\bar{q}$]
\label{thm:regul_control}
Let $s\in(0,1)$ and $r > d/2s$. Let $\bar{q}$ be a locally optimal control for \eqref{eq:weak_min_problem}--\eqref{eq:weak_st_eq}. If $f,u_{\Omega}\in L^2(\Omega) \cap L^{r}(\Omega)$, then $\bar{q} \in H^{s + \kappa -\varepsilon}(\Omega)$ for all $0 < \varepsilon  < s$, where $\kappa = \frac{1}{2}$ for $\frac{1}{2} < s < 1$ and $\kappa = s - \varepsilon$ for $0 < s \leq \frac{1}{2}$. In addition,
\begin{equation}
\| \bar{q} \|_{H^{s + \kappa -\varepsilon}(\Omega)} \leq C \varepsilon^{-\nu} (1 + \Lambda(f,u_{\Omega})),
\label{eq:regularity_estimate_control}
\end{equation}
where $\nu$ is as in the statement of Theorem \ref{thm:sobolev_reg}.
\end{theorem}
\begin{proof}
Let $\bar u = \mathcal{S} \bar{q}$ and let $\bar p$ be the solution to \eqref{eq:adj_eq} with $u$ and $q$ replaced by $\bar u$ and $\bar q$, respectively. Since $f, u_{\Omega}\in L^2(\Omega)$, an application of Theorem \ref{thm:sobolev_reg} immediately shows that $\bar{u}$ and $\bar{p}$ belong to $H^{s + \kappa -\varepsilon}(\Omega)$, for all $0 < \varepsilon  < s$, where $\kappa$ is as in the statement of the theorem. On the other hand, since $f \in L^r(\Omega)$, for some $r>d/2s$, we can conclude from the results of Theorem \ref{thm:L_infty_reg} that $\bar{u} \in L^{\infty}(\Omega)$. Note that, in particular, $\bar{u} \in L^r(\Omega)$, for some $r>d/2s$. Consequently, $\bar{u} - u_{\Omega} \in L^r(\Omega)$. We can therefore again refer to Theorem \ref{thm:L_infty_reg} to conclude that $\bar{p} \in L^{\infty}(\Omega)$. Observe that we have obtained that $\bar{u}, \bar{p} \in H^{s + \kappa -\varepsilon}(\Omega) \cap L^{\infty}(\Omega)$. This, in light of the continuity of the product property stated in Lemma \ref{lemma:product_of_func}, allows us to conclude that $\bar{u}\bar{p}$ belongs to $H^{s + \kappa -\varepsilon}(\Omega)\cap L^{\infty}(\Omega)$ together with the bound \cite[estimate (25)]{MR1877265}
$
\| \bar{u}\bar{p}  \|_{H^{s + \kappa -\varepsilon}(\Omega)} 
\lesssim C\varepsilon^{-\nu}
\Lambda(f,u_{\Omega}).
$
The desired regularity property for $\bar{q}$ and the estimate \eqref{eq:regularity_estimate_control} follow from the projection formula \eqref{eq:projection_control}, the fact that $\max\{0,\mathfrak{r}\}=(\mathfrak{r} + |\mathfrak{r}|)/2$ for all $\mathfrak{r}\in\mathbb{R}$, and \cite[Theorem 1]{MR1173747}, which applies because $s + \kappa -\varepsilon = s + \frac{1}{2} -\varepsilon < \tfrac{3}{2}$ if $\frac{1}{2} < s < 1$ and $s + \kappa -\varepsilon = 2s -2\varepsilon \leq 1 -2\varepsilon < \tfrac{3}{2}$ if $0 < s \leq \frac{1}{2}$. This concludes the proof.
\end{proof} 


\subsection{Second order optimality conditions}\label{sec:2nd_order_opt_cond}
In this section we formulate necessary and sufficient second order optimality conditions. 

\subsubsection{Auxiliary results} 
We begin our studies with the following result.

\begin{theorem}[$j$ is of class $C^2$ and $j''$ is Lipschitz]
\label{thm:diff_properties_j}
The reduced cost functional $j: \mathbb{Q}_{ad}\!\subset\!\mathcal{Q} \rightarrow \mathbb{R}$ is of class $C^2$. Moreover, for every $q \in\! \mathbb{Q}_{ad}$ and $w \in \! L^{\infty}(\Omega)$, we have 
\begin{equation}\label{eq:charac_j2}
j''(q)w^2 = \lambda\|w\|_{L^2(\Omega)}^2 - 2(wz,p)_{L^2(\Omega)} + \|z\|_{L^2(\Omega)}^2,
\end{equation}
where $p$ solves \eqref{eq:adj_eq} and $z=\mathcal{S}'(q)w$. If, in addition,  we assume that $f,u_{\Omega} \in L^{r}(\Omega)$, with $r > d/2s$, then, for $q_{1},q_{2}\in\mathbb{Q}_{ad}$ and $w\in L^\infty(\Omega)$, we have
\begin{equation}\label{eq:estimate_of_j2}
|j''(q_1)w^2-j''(q_2)w^2|
\lesssim
\|w\|_{L^2(\Omega)}^2\|q_1 - q_2\|_{L^r(\Omega)}.
\end{equation}
\end{theorem}
\begin{proof}
The fact that $j$ is of class $C^2$ is a direct consequence of the differentiability properties of the control to state map $\mathcal{S}$ given in Theorem \ref{thm:properties_C_to_S}. It is therefore sufficient to derive the identity \eqref{eq:charac_j2} and the inequality \eqref{eq:estimate_of_j2}. To accomplish this task, we proceed in two steps.

\emph{Step 1}. We first obtain \eqref{eq:charac_j2}. We begin with simple calculations showing that for every $q\in\mathbb{Q}_{ad}$ and $w \in L^{\infty}(\Omega)$ we have
\begin{equation}\label{eq:charac_j2_prev}
j''(q)w^2
=
\lambda\|w\|_{L^2(\Omega)}^2 + (u - u_{\Omega},\mathfrak{z})_{L^2(\Omega)}+\|z\|_{L^2(\Omega)}^2,
\end{equation}
where $z,\mathfrak{z} \in \tilde{H}^{s}(\Omega)$ are as in the statement of Theorem \ref{thm:properties_C_to_S} with $w_1 = w_2 = w$ and $z_w = z$. To obtain the desired identity \eqref{eq:charac_j2}, we proceed as follows. We set $v =  \mathfrak{z}$ in \eqref{eq:adj_eq} and $v = p$ in \eqref{eq:second_der_S}. This results in the relation
$
(u - u_{\Omega},\mathfrak{z})_{L^2(\Omega)} = -2(wz,p)_{L^2(\Omega)}.
$
Substituting this identity into \eqref{eq:charac_j2_prev}, we obtain \eqref{eq:charac_j2}.

\emph{Step 2}. We now prove \eqref{eq:estimate_of_j2}. Let $q_1,q_2\in\mathbb{Q}_{ad}$ and $w \in L^{\infty}(\Omega)$. Define $z_{1} = \mathcal{S}'(q_1)w$ and $z_{2} = \mathcal{S}'(q_2)w$. Note that $z_{1}$ and $z_{2}$ solve \eqref{eq:first_der_S} with $u_{1} := \mathcal{S}q_{1}$ and  $u_{2} := \mathcal{S}q_{2}$, respectively. We now use the derived identity \eqref{eq:charac_j2} to obtain
\begin{multline}
\label{eq:identity_jq1_jq2}
j''(q_1)w^2-j''(q_2)w^2 = 
 2(w[z_{2} - z_{1}],p_{2})_{L^2(\Omega)} + 2(wz_{1},p_{2} - p_{1})_{L^2(\Omega)} \\
 + (z_{1} - z_{2}, z_{1} + z_{2})_{L^2(\Omega)}
=: \mathbf{I} + \mathbf{II} + \mathbf{III},
\end{multline}
where $p_{i} \in \tilde{H}^{s}(\Omega)$ is the solution of \eqref{eq:adj_eq} where $u$ and $q$ are replaced by $u_{i}$ and $q_{i}$, respectively $( i \in \{ 1,2\} )$. In the following, $\mathbf{I}$, $\mathbf{II}$, and $\mathbf{III}$ are estimated separately.

We start with the estimation of $\mathbf{I}$. In light of Theorem \ref{thm:L_infty_reg}, the fact that $f,u_{\Omega}\in L^{r}(\Omega)$, for $r > d/2s$, yields the following bounds
\begin{equation}\label{eq:estimate_p2_infty}
\|u_{2}\|_{L^{\infty}(\Omega)} \lesssim \|f\|_{L^{r}(\Omega)},
\qquad
\|p_{2}\|_{L^{\infty}(\Omega)}
\lesssim \|f\|_{L^{r}(\Omega)} + \|u_{\Omega}\|_{L^{r}(\Omega)}.
\end{equation}
Consequently, $\mathbf{I} \lesssim \|w\|_{L^2(\Omega)}\|z_{2} - z_{1}\|_{L^2(\Omega)}$, with a hidden constant that depends on $ \|f\|_{L^{r}(\Omega)}$ and $\|u_{\Omega}\|_{L^{r}(\Omega)}$. To bound $\|z_{2} - z_{1}\|_{L^2(\Omega)}$, we observe that $z_{2} - z_{1}$ solves
\begin{equation*}
\mathcal{A}(z_{2} - z_{1},v) + (q_{2}(z_{2} - z_{1}),v)_{L^2(\Omega)} = (z_{1}[q_{1} - q_{2}] - w[u_{2} - u_{1}],v)_{L^2(\Omega)}
\end{equation*}
for all $v\in \tilde{H}^{s}(\Omega)$. A stability estimate for this problem in conjunction with H\"older's inequality and Lemma \ref{lemma:embedding_result} allows us to obtain that
\begin{equation}
\label{eq:estimate_z2_z1}
\|z_{2} - z_{1}\|_{s} \lesssim \|z_{1}\|_{L^{\frac{2d}{d-2s}}(\Omega)}\|q_{1} - q_{2}\|_{L^r(\Omega)} + \|w\|_{L^2(\Omega)}\|u_{2} - u_{1}\|_{L^{\infty}(\Omega)}.
\end{equation}
An estimate for the term $\|z_{1}\|_{L^{2d/(d-2s)}(\Omega)}$ follows directly from the well-posedness of problem \eqref{eq:first_der_S} and the Sobolev embedding of Lemma \ref{lemma:embedding_result}. In fact, we have
\begin{equation}\label{eq:estimate_z1}
\|z_{1}\|_{L^{\frac{2d}{d-2s}}(\Omega)} \lesssim \|z_{1}\|_{s} \lesssim \|w\|_{L^2(\Omega)}\|u_{1}\|_{L^{\infty}(\Omega)} \lesssim \|w\|_{L^2(\Omega)}\|f\|_{L^{r}(\Omega)}.
\end{equation}
On the other hand, we notice that $u_{2} - u_{1}\in \tilde{H}^{s}(\Omega)$ solves the following problem:
\begin{equation*}
\mathcal{A}(u_{2} - u_{1},v) + (q_{2}(u_{2} - u_{1}),v)_{L^2(\Omega)} = (u_{1}[q_{1} - q_{2}],v)_{L^2(\Omega)} \quad \forall v\in \tilde{H}^{s}(\Omega).
\end{equation*}
Since $u_{1}(q_{1} - q_{2})\in L^{r}(\Omega)$, an application of the results of Theorem \ref{thm:L_infty_reg} reveals that $\|u_{2} - u_{1}\|_{L^{\infty}(\Omega)}\lesssim \|u_{1}\|_{L^{\infty}(\Omega)}\|q_{2} - q_{1}\|_{L^{r}(\Omega)}\lesssim \|f\|_{L^{r}(\Omega)}\|q_{2} - q_{1}\|_{L^{r}(\Omega)}$. Replacing the estimates obtained for $\|z_{1}\|_{L^{2d/(d-2s)}(\Omega)}$ and $\|u_{2} - u_{1}\|_{L^{\infty}(\Omega)}$ into \eqref{eq:estimate_z2_z1} yields $\|z_{2} - z_{1}\|_{s} \lesssim \|f\|_{L^{r}(\Omega)}\|w\|_{L^{2}(\Omega)}\|q_{1} - q_{2}\|_{L^r(\Omega)}$. From this, we can conclude that
\begin{equation}\label{eq:estimate_I}
\mathbf{I} \lesssim \|w\|_{L^2(\Omega)}\|z_{2} - z_{1}\|_{L^2(\Omega)} \lesssim \|w\|_{L^{2}(\Omega)}\|z_{2} - z_{1}\|_{s} \lesssim \|w\|_{L^{2}(\Omega)}^2\|q_{1} - q_{2}\|_{L^r(\Omega)},
\end{equation}
with a hidden constant that depends on $ \|f\|_{L^{r}(\Omega)}$ and $\|u_{\Omega}\|_{L^{r}(\Omega)}$.

We now control the term $\mathbf{II}$. To accomplish this task, we utilize H\"older's inequality and \eqref{eq:estimate_z1} to obtain $\mathbf{II} \lesssim \|w\|_{L^2(\Omega)}\|z_{1}\|_{L^2(\Omega)}\|p_{2} - p_{1}\|_{L^{\infty}(\Omega)} \lesssim \|w\|_{L^2(\Omega)}^2\|p_{2} - p_{1}\|_{L^{\infty}(\Omega)}$. We now bound the term $\|p_{2} - p_{1}\|_{L^{\infty}(\Omega)}$ based on the fact that $p_{2} - p_{1}$ solves
\begin{equation*}
\mathcal{A}(v,p_{2} - p_{1}) + (q_{2}(p_{2} - p_{1}),v)_{L^2(\Omega)} = ([u_{2} - u_{1}] + p_{1}[q_{1} - q_{2}],v)_{L^2(\Omega)} \quad \forall v\in \tilde{H}^{s}(\Omega).
\end{equation*}
Since the right-hand side of the previous weak formulation belongs to $L^{r}(\Omega)$, Theorem \ref{thm:L_infty_reg} immediately implies that $p_{2} - p_{1}\in \tilde H^s(\Omega) \cap L^{\infty}(\Omega)$ together with the estimate
\begin{equation*}
\|p_{2} - p_{1}\|_{L^{\infty}(\Omega)}
\lesssim
\|u_{2} - u_{1}\|_{L^{\infty}(\Omega)} + \|p_{1}\|_{L^{\infty}(\Omega)}\|q_{2} - q_{1}\|_{L^{r}(\Omega)}.
\end{equation*}
We thus obtain, in view of $\|u_{2} - u_{1}\|_{L^{\infty}(\Omega)} \lesssim \|q_{2} - q_{1}\|_{L^{r}(\Omega)}$ and an analogue of \eqref{eq:estimate_p2_infty} for $\|p_{1}\|_{L^{\infty}(\Omega)}$, that $\|p_{2} - p_{1}\|_{L^{\infty}(\Omega)} \lesssim \|q_{2} - q_{1}\|_{L^{r}(\Omega)}$. Consequently,
\begin{equation}\label{eq:estimate_II}
\mathbf{II} \lesssim \|w\|_{L^2(\Omega)}^2\|p_{2} - p_{1}\|_{L^{\infty}(\Omega)} \lesssim \|w\|_{L^2(\Omega)}^2
\|q_{2} - q_{1}\|_{L^{r}(\Omega)}.
\end{equation}

Finally, we control $\mathbf{III}$. To do this, we invoke the bound \eqref{eq:estimate_z1}, an analogue of \eqref{eq:estimate_z1} for $z_2$, and $\|z_{2} - z_{1}\|_{s} \lesssim \|w\|_{L^{2}(\Omega)}\|q_{1} - q_{2}\|_{L^r(\Omega)}$. These arguments yield
\begin{align}\label{eq:estimate_III}
\mathbf{III} 
&\lesssim \|z_{2} - z_{1}\|_{L^{2}(\Omega)}(\|z_{1}\|_{L^2(\Omega)} + \|z_{2}\|_{L^2(\Omega)})  \\
&\lesssim \|z_{2} - z_{1}\|_{s}\|w\|_{L^2(\Omega)}
\lesssim \|w\|_{L^2(\Omega)}^2\|q_{2} - q_{1}\|_{L^r(\Omega)}.
\nonumber
\end{align} 

We conclude \eqref{eq:estimate_of_j2} by replacing \eqref{eq:estimate_I}, \eqref{eq:estimate_II}, and \eqref{eq:estimate_III} into \eqref{eq:identity_jq1_jq2}.
\end{proof}


\subsubsection{Second order necessary optimality conditions} 
Let $(\bar{u},\bar{p},\bar{q}) \in \tilde{H}^{s}(\Omega) \times \tilde{H}^{s}(\Omega) \times \mathbb{Q}_{ad}$ satisfy the first order optimality conditions \eqref{eq:weak_st_eq}, \eqref{eq:adj_eq}, and \eqref{eq:var_ineq_with_adj_state}. Define $\bar{\mathfrak{d}} :=  \lambda\bar{q} - \bar{u}\bar{p}$. The variational inequality \eqref{eq:var_ineq_with_adj_state} immediately yields, for a.e.~$x\in \Omega$,
\begin{equation}
\label{eq:derivative_j}
\bar{\mathfrak{d}}(x) 
= 0  \text{ if } a < \bar{q}(x) < b, 
\qquad
\bar{\mathfrak{d}}(x) \geq  0  \text{ if }\bar{q}(x)=a, 
\qquad
\bar{\mathfrak{d}}(x) \leq  0 \text{ if } \bar{q}(x)=b.
\end{equation}

In order to formulate second order conditions, we introduce the following \emph{cone of critical directions} inspired by \cite[definition (2.7)]{MR3878305} and \cite[Section 6, page 20]{MR3586845}:
\begin{equation}\label{def:critical_cone}
C_{\bar{q}}:=\{w\in L^{2}(\Omega) \text{ satisfying } \eqref{eq:sign_cond} \text{ and } w(x) = 0 \text{ if } \bar{\mathfrak{d}}(x) \neq 0\},
\end{equation}
where condition \eqref{eq:sign_cond} reads as follows:
\begin{equation}
\label{eq:sign_cond}
w(x)
\geq 0 \text{ a.e.}~x\in\Omega \text{ if } \bar{q}(x)=a,
\qquad
w(x)
\leq 0 \text{ a.e.}~x\in\Omega \text{ if } \bar{q}(x)=b.
\end{equation}

In the following, we review the arguments developed in the proof of \cite[Theorem 23]{MR3586845} and present second order necessary optimality conditions.

\begin{theorem}[second order necessary optimality conditions]
\label{thm:nec_opt_cond}
If $\bar{q}\in \mathbb{Q}_{ad}$ denotes a locally optimal control for \eqref{eq:weak_min_problem}--\eqref{eq:weak_st_eq}, then $j''(\bar{q})w^2 \geq 0$ for all $w\in C_{\bar{q}}$.
\end{theorem}
\begin{proof}
Let $w\in C_{\bar{q}}$. Define, for every $k \in \mathbb{N}$ and for a.e.~$x\in\Omega$, the function
\begin{equation*}
w_{k}(x):=
\begin{cases}
\qquad \quad 0 \quad &\text{ if }\quad x: a < \bar{q}(x) < a + k^{-1}, 
\quad 
b-k^{-1} < \bar{q}(x) < b, 
\\
\Pi_{[-k,k]}(w(x)) &\text{ otherwise}. 
\end{cases}
\end{equation*}
Since $w\in C_{\bar{q}}$, it is immediate that $w_{k} \in C_{\bar{q}} \cap L^{\infty}(\Omega)$. On the other hand, we have that $w_{k}(x)\rightarrow w(x)$ as $k \uparrow \infty$ for a.e.~$x \in \Omega$ and that $|w_{k}(x)|\leq |w(x)|$ for a.e.~$x \in \Omega$. Consequently, $w_{k} \to w$ in $L^2(\Omega)$ as $k \uparrow \infty$. We now note that, for $\rho$ sufficiently small, or more precisely, for $\rho \in (0, k^{-2}]$, $\bar{q}+\rho w_{k}$ belongs to $\mathbb{Q}_{ad}$ for every $k \in \mathbb{N}$. Since $\bar{q}+\rho w_{k}$ is thus admissible, we rely on the fact that $\bar{q}$ is a local minimizer to arrive at the basic inequality $j(\bar{q}) \leq j(\bar{q}+\rho w_{k})$ when $\rho$ is sufficiently small. We now use the relation $j'(\bar{q})w_{k}=0$, which follows from the fact that $w_{k} \in C_{\bar{q}}$, and Taylor's theorem for $j$ at $\bar{q}$ to obtain that, for $\rho$ sufficiently small,
\begin{equation*}
0
\leq
j(\bar{q}+\rho w_{k})-j(\bar{q}) 
=
\rho j'(\bar{q})w_{k}+\tfrac{\rho^2}{2}j''(\bar{q}+\rho\theta_{k}w_{k})w_{k}^2
=
\tfrac{\rho^2}{2}j''(\bar{q}+\rho\theta_{k}w_{k})w_{k}^2,
\end{equation*}
with $\theta_{k} \in (0,1)$. We now let $\rho\downarrow 0$ to obtain, on the basis of the estimate \eqref{eq:estimate_of_j2},
\[
 |j''(\bar{q}+\rho\theta_{k}w_{k})w_{k}^2 - j''(\bar{q})w_{k}^2| \lesssim \| w_k\|^2_{L^2(\Omega)} \rho \theta_k\| w_k\|_{L^r(\Omega)} \rightarrow 0, \quad \rho \downarrow 0,
\]
which implies that $j''(\bar{q})w_{k}^2 \geq 0$ for every $k \in \mathbb{N}$. Let us now invoke the convergence property $w_{k} \to w$ in $L^2(\Omega)$ as $k \uparrow \infty$ and \eqref{eq:charac_j2} to conclude that $j''(\bar{q})w^2 \geq 0$; see the proof of Theorem \ref{thm:optimal_solution} below for further details. This concludes the proof.
\end{proof}


\subsubsection{Second order sufficient optimality conditions}
In this section, we follow the arguments elaborated in the proof of \cite[Theorem 2.3]{MR2902693} and \cite[Theorem 23]{MR3586845} and prove a sufficient second order optimality condition with a minimal gap with respect to the necessary condition derived in Theorem \ref{thm:nec_opt_cond}.

\begin{theorem}[second order sufficient optimality conditions]
\label{thm:optimal_solution}
Let $(\bar{u},\bar{p},\bar{q}) \in \tilde{H}^{s}(\Omega) \times\tilde{H}^{s}(\Omega) \times \mathbb{Q}_{ad}$ satisfy the first order optimality conditions \eqref{eq:weak_st_eq}, \eqref{eq:adj_eq}, and \eqref{eq:var_ineq_with_adj_state}. If $j''(\bar{q})w^2 > 0$ for all $w\in C_{\bar{q}}\setminus \{0\}$, then there exist $\delta > 0$ and $\sigma > 0$ such that
\begin{equation*}
\label{eq:optimal_minimum}
j(q)\geq j(\bar{q})+\tfrac{\delta}{2}\|q-\bar{q}\|_{L^2(\Omega)}^2\quad \forall q\in \mathbb{Q}_{ad}: \|q-\bar{q}\|_{L^{2}(\Omega)}\leq \sigma.
\end{equation*}
In particular, $\bar{q}$ is a locally optimal control in the sense of $L^{2}(\Omega)$.
\end{theorem}
\begin{proof}
We proceed by contradiction and assume that for every natural number $k$ there exists an element $q_{k}\in \mathbb{Q}_{ad}$ such that
\begin{equation}\label{eq:contradic_I}
\|\bar{q}-q_{k}\|_{L^{2}(\Omega)} < k^{-1},
\qquad 
j(q_{k}) < j(\bar{q}) + (2k)^{-1}\|\bar{q}-q_{k}\|_{L^2(\Omega)}^2.
\end{equation}
Let us introduce, for $k \in \mathbb{N}$, 
$
\rho_{k}:=\|\bar{q}-q_{k}\|_{L^2(\Omega)} 
$
and
$
w_{k}:=\rho_{k}^{-1}(q_{k}-\bar{q}).
$
Note that, for every $k \in \mathbb{N}$, $\| w_k\|_{L^2(\Omega)} = 1$. We can therefore assume, taking a subsequence if necessary, that $w_{k} \rightharpoonup w$ in $L^2(\Omega)$ as $k \uparrow \infty$. We now prove that the limit point $w$ belongs to the cone $C_{\bar{q}}$ and then that $w\equiv 0$. We proceed in three steps.

\emph{Step 1.} \emph{$w \in C_{\bar q}$}. We first note that the set of elements satisfying \eqref{eq:sign_cond} is closed and convex in $L^2(\Omega)$; it is thus weakly sequentially closed and therefore $w$ also satisfies  \eqref{eq:sign_cond}. To verify the remaining condition in \eqref{def:critical_cone}, we apply the mean value theorem and the estimate of the right-hand side in \eqref{eq:contradic_I} to obtain
\begin{equation}\label{eq:contradic_II}
 j'(\tilde{q}_{k})w_{k}  = \rho_k^{-1}(j(q_{k}) - j(\bar{q}))< \rho_{k}(2k )^{-1}\to 0,
\quad
k \uparrow \infty,
\end{equation}
where $\tilde{q}_{k}=\bar{q}+ \theta_{k}(q_{k} - \bar{q})$ and $\theta_k \in (0,1)$. Define $\tilde{u}_{k}:= \mathcal{S}\tilde{q}_{k}$ and $\tilde{p}_{k}$ as the unique solution to \eqref{eq:adj_eq} with $u$ and $q$ replaced by $\tilde{u}_{k}$ and $\tilde{q}_{k}$, respectively. Since $\tilde q_k \rightarrow \bar q$ in $L^2(\Omega)$ as $k \uparrow \infty$, a basic stability bound for the problem that $\bar{u} - \tilde{u}_k$ solves shows that  $\tilde{u}_{k} \to \bar{u}$ in $\tilde{H}^s(\Omega)$ as $k \uparrow \infty$. With this convergence property at hand, a similar argument yields $\tilde{p}_{k}\to \bar{p}$ in $\tilde{H}^s(\Omega)$ as $k \uparrow \infty$. In particular, $\tilde{p}_{k}\to \bar{p}$ in $L^2(\Omega)$ as $k \uparrow\infty$. We can thus conclude that $\lambda \tilde{q}_k - \tilde{u}_{k}\tilde{p}_k  =:\tilde{\mathfrak{d}}_k \rightarrow \bar{\mathfrak{d}} =  \lambda \bar q  - \bar{u}\bar{p}$ in $L^2(\Omega)$ as $k\uparrow\infty$, upon using that $\{\tilde{u}_k\}_{k \in \mathbb{N}}$ and $\{\tilde{p}_k\}_{k\in\mathbb{N}}$ are uniformly bounded in $L^{\infty}(\Omega)$. We now invoke the weak convergence $w_{k} \rightharpoonup w$ in $L^{2}(\Omega)$ as $k \uparrow\infty$ and \eqref{eq:contradic_II} to deduce that
\begin{equation*}
j'(\bar{q})w  = \int_{\Omega} \bar{\mathfrak{d}}(x) w(x) \mathrm{d}x = \lim_{k \uparrow \infty}
\int_{\Omega} \tilde{\mathfrak{d}}_k(x) w_k(x) \mathrm{d}x = \lim_{k \uparrow \infty}j'(\tilde{q}_{k})w_{k} \leq 0.
\end{equation*}
On the other hand, from \eqref{eq:var_ineq_with_adj_state} we have $j'(\bar{q})w  = \lim_{k\uparrow\infty}\int_{\Omega} \bar{\mathfrak{d}}(x) w_{k}(x) \mathrm{d}x \geq 0$. Consequently, $\int_{\Omega} \bar{\mathfrak{d}}(x)w(x) \mathrm{d}x = 0$. Since $w$ satisfies the condition \eqref{eq:sign_cond} and $\bar{\mathfrak{d}}$ satisfies the condition \eqref{eq:derivative_j}, we infer that $\int_{\Omega} |\bar{\mathfrak{d}}(x)w(x)| \mathrm{d}x = \int_{\Omega} \bar{\mathfrak{d}}(x)w(x) \mathrm{d}x = 0$. This proves that, a.e.~in $\Omega$, $\bar{\mathfrak{d}} \neq 0$ implies that $w = 0$. Consequently, $w\in C_{\bar{q}}$.

\emph{Step 2}. \emph{$w\equiv 0$.} With the help of Taylor's theorem, the basic inequality $j'(\bar{q})(q_{k}-\bar{q})\geq 0$, and the estimate of the right-hand side in \eqref{eq:contradic_I} we arrive at
\begin{equation*}
\tfrac{\rho_{k}^2}{2}j''(\hat{q}_{k})w_{k}^2
=
j(q_k) - j(\bar{q})-j'(\bar{q})(q_k - \bar{q})
\leq 
j({q}_{k})-j(\bar{q}) < \rho_{k}^2(2k)^{-1},
\quad
k \in \mathbb{N},
\end{equation*}
where $\hat{q}_{k}:=\bar{q}+\hat{\theta}_{k}(q_{k} - \bar{q})$ and $\hat{\theta}_{k} \in (0,1)$. Thus, $j''(\hat{q}_{k})w_{k}^2 < k^{-1} \rightarrow 0$ as $k\uparrow \infty$.

We now show that $j''(\bar{q})w^2 \leq \liminf_{k \uparrow \infty }j''(\hat{q}_{k})w_{k}^2$. As a first step, we invoke the characterization \eqref{eq:charac_j2} to write
\begin{equation*}
j''(\hat{q}_k)w_{k}^2
=
\lambda\|w_{k}\|_{L^2(\Omega)}^2  - 2(w_{k}\hat{z}_{k},\hat{p}_k)_{L^2(\Omega)}+\|\hat{z}_{k}\|_{L^2(\Omega)}^2. 
\end{equation*}
Here, $\hat{p}_k$ denotes the solution to \eqref{eq:adj_eq} with $q$ and $u$ replaced by $\hat{q}_k$ and $\hat{u}_k := \mathcal{S} \hat{q}_k$, respectively, and $\hat{z}_k$ solves \eqref{eq:first_der_S} with $q$, $w$, and $u$ replaced by $\hat{q}_k$, $w_k$, and $\hat{u}_k$, respectively. We note that $\{ \hat{u}_k \}_{k\in \mathbb{N}}$ and $\{ \hat{p}_k \}_{k\in \mathbb{N}}$ are uniformly bounded in $L^{\infty}(\Omega)$. In fact, for $k \in \mathbb{N}$, it holds that
$
 \| \hat{u}_k \|_{L^{\infty}(\Omega)} \lesssim \|f\|_{L^r(\Omega)}
$
and
$
 \| \hat{p}_k \|_{L^{\infty}(\Omega)} \lesssim \|f\|_{L^r(\Omega)} + \|u_{\Omega}\|_{L^r(\Omega)}.
$
With the available estimates, we can therefore derive
\[
\| \bar{u} - \hat{u}_k \|_s +  \| \bar{p} - \hat{p}_k \|_s
\lesssim
\left( \|f \|_{L^r(\Omega)} +  \|u_{\Omega} \|_{L^r(\Omega)}\right)
\| \bar{q} - \hat{q}_k \|_{L^2(\Omega)} \rightarrow 0,
\quad
k \uparrow \infty,
\]
upon using basic stability bounds for the problems that $\bar{u} - \hat{u}_k$ and $\bar{p} - \hat{p}_k$ solve. 

Let $z$ solve \eqref{eq:first_der_S}, where $q$ and $u$ are replaced by $\bar{q}$ and $\bar{u}$, respectively. We now investigate the convergence of $\{ \hat{z}_k \}_{k\in \mathbb{N}}$. First, we write the problem that $z - \hat{z}_k$ solves
\begin{multline*}
 A(z - \hat{z}_k,v) + (\bar{q}(z - \hat{z}_k),v)_{L^2(\Omega)}
 =-((w-w_k)\bar{u},v)_{L^2(\Omega)}
 \\
+ (w_k(\hat{u}_k-\bar{u}),v)_{L^2(\Omega)}
 +
 ((\hat{q}_k-\bar{q})\hat{z}_k,v)_{L^2(\Omega)}
 =: \mathrm{I}_k + \mathrm{II}_k + \mathrm{III}_k
 \quad \forall v \in \tilde{H}^s(\Omega).
\end{multline*}
Since $\bar{u} \in L^{\infty}(\Omega)$ and $w_{k} \rightharpoonup w$ in $L^{2}(\Omega)$ as $k \uparrow \infty$, it is immediate that $|\mathrm{I}_{k}| \rightarrow 0$. To analyze the convergence of $\mathrm{II}_k$, we observe that, for every $\tau \in (2,\infty)$, we have
\begin{equation}\label{eq:Lp_conv_uk}
 \| \bar{u} - \hat{u}_k \|^{\tau}_{L^{\tau}(\Omega)} \leq \| |\bar{u} - \hat{u}_k|^{\tau-2} \|_{L^{\infty}(\Omega)} \| \bar{u} - \hat{u}_k \|^2_{L^2(\Omega)}\lesssim \| f \|^{\tau-2}_{L^r(\Omega)}\| \bar{u} - \hat{u}_k \|^2_{L^2(\Omega)} \rightarrow 0
\end{equation}
as $k\uparrow \infty$, exploiting the fact that $\{ \hat{u}_k \}_{k\in \mathbb{N}}$ is uniformly bounded in $L^{\infty}(\Omega)$ and the bound of Theorem \ref{thm:L_infty_reg}. With this result at hand, we control the term $\mathrm{II}_k$ as follows:
\[
 |\mathrm{II}_k| \leq \| w_k\|_{L^2(\Omega)} \| \bar{u} - \hat{u}_k \|_{L^{\tau}(\Omega)} \| v \|_{L^{\mu}(\Omega)},
 \qquad
 2^{-1} + \tau^{-1} + {\mu}^{-1} = 1.
\]
Set $\mu = 2d/(d-2s)$ (see Lemma \ref{lemma:embedding_result}) and note that $\mu > 2$. Since $\hat{u}_k \rightarrow \bar{u}$ in $L^{\tau}(\Omega)$ for every $\tau \in (2,\infty)$, we can thus conclude that $|\mathrm{II}_k| \rightarrow 0$ as $k\uparrow \infty$. We now analyze the convergence of $\mathrm{III}_{k}$. First, we note that a stability bound for the problem that $\hat{z}_{k}$ solves, namely $\|\hat{z}_{k}\|_{s} \lesssim \|f\|_{L^{r}(\Omega)}$, for every $k\in\mathbb{N}$, yields
\[
\|\hat{z}_{k}\|_{L^{t}(\Omega)} \lesssim \|f\|_{L^{r}(\Omega)}, 
\qquad 
k \in \mathbb{N}, 
\qquad  
t \in \left[1, \frac{2d}{d - 2s}\right]. 
\] 
Second, as in \eqref{eq:Lp_conv_uk}, $\hat{q}_{k} \rightarrow \bar{q}$ in $L^{2}(\Omega)$ as $k \uparrow \infty$ combined with the fact that $\{\hat{q}_{k}\}$ is uniformly bounded in $L^{\infty}(\Omega)$ allows us to obtain that $\hat{q}_{k} \rightarrow \bar{q}$ in $L^{\tau}(\Omega)$ for every $\tau \in (2,\infty)$. A simple application of H\"older's inequality thus shows that
\[
|\textrm{III}_{k}| \leq \|\hat{z}_{k}\|_{L^{t}(\Omega)}\|\hat{q}_{k} - \bar{q}\|_{L^{\tau}(\Omega)}\|v\|_{L^{\mu}(\Omega)},
\qquad
t^{-1} + \tau^{-1} + {\mu}^{-1} = 1.
\]
Since $\hat{q}_k \rightarrow \bar{q}$ in $L^{\tau}(\Omega)$, for every $\tau < \infty$, we conclude that $|\mathrm{III}_k| \rightarrow 0$ as $k\uparrow \infty$. We have therefore proved that $|\mathrm{I}_k|, |\mathrm{II}_k|, |\mathrm{III}_k| \rightarrow 0$ as $k \uparrow \infty$. Consequently, $\hat{z}_k \rightharpoonup z$ in $\tilde{H}^s(\Omega)$. In view of Lemma \ref{lemma:embedding_result}, this results in
\begin{equation}
\hat{z}_k \rightarrow z \textrm{ in } L^{\sigma}(\Omega), \quad \sigma < 2d/(d-2s) \implies \| \hat{z}_k \|^2_{L^2(\Omega)} \rightarrow \|z\|_{L^2(\Omega)}^2, \quad k \uparrow \infty.
\label{eq:z_in_L2}
\end{equation}

On the other hand,
\begin{multline}
|(wz,\bar{p})_{L^2(\Omega)} - (w_{k}\hat{z}_{k} ,\hat{p}_k)_{L^2(\Omega)}| \leq |((w  - w_{k})z,\bar{p})_{L^2(\Omega)}| 
\\
+ |(w_k(z - \hat{z}_{k}),\bar{p})_{L^2(\Omega)}|
+ |(w_k\hat{z}_{k},\bar{p}-\hat{p}_k)_{L^2(\Omega)}|
\rightarrow 0, \quad k \uparrow \infty,
\label{eq:wzp}
\end{multline}
because $w_{k} \rightharpoonup w$ in $L^2(\Omega)$ as $k \uparrow \infty$, 
$z\bar{p} \in L^{2}(\Omega)$, $\| w_k \|_{L^2(\Omega)} = 1$, $\| z - \hat{z}_k \|_{L^2(\Omega)} \rightarrow 0$ as $k \uparrow \infty$, $\bar{p} \in L^{\infty}(\Omega)$, and $\| \bar{p}-\hat{p}_k \|_{L^{\tau}(\Omega)} \rightarrow 0$, as $k \uparrow \infty$, for every $\tau \in (2,\infty)$. 

Finally, we invoke \eqref{eq:z_in_L2}, \eqref{eq:wzp}, and the fact that the square of the $L^2(\Omega)$-norm is weakly lower semicontinuous in $L^2(\Omega)$ to deduce that $j''(\bar{q})w^2 \leq \liminf_{k} j''(\hat{q}_k)w_{k}^2$.

As a result, since $j''(\bar{q})w^2 \leq \liminf_{k} j''(\hat{q}_k)w_{k}^2$, $j''(\hat{q}_k)w_{k}^2 < k^{-1} \rightarrow 0$ as $k \uparrow \infty$, and $w \in C_{\bar{q}}$, the optimality condition $j''(\bar{q})w^2 >0$, for all $w \in C_{\bar{q}} \setminus \{ 0 \}$, yields $w\equiv 0$.

\emph{Step 3.} \emph{Contradiction}. We finally arrive at the contradiction. Since $w\equiv 0$, it is immediate that $\hat{z}_{k} \rightharpoonup 0$ in $\tilde{H}^s(\Omega)$ as $k \uparrow \infty$. Hence, from the equalities
\begin{equation*}
\lambda = \lambda\|w_{k}\|_{L^2(\Omega)}^2 = j''(\hat{q}_{k})w_{k}^2 + 2(w_{k}\hat{z}_{k},\hat{p}_k)_{L^2(\Omega)} - \|\hat{z}_{k}\|_{L^2(\Omega)}^2,
\end{equation*}
and the fact that  $\liminf_{k}j''(\hat{q}_{k})w_{k}^2 \leq 0$, we conclude that $\lambda \leq 0$. This is a contradiction and concludes the proof.
\end{proof}

The following result establishes an equivalent representation of the second order optimality condition introduced in Theorem \ref{thm:optimal_solution}. This equivalence is important for deriving error estimates for the methods proposed in our work. To present it, we introduce $C_{\bar{q}}^\tau:=\{w\in L^{2}(\Omega) \textnormal{ satisfying \eqref{eq:sign_cond} and } w(x)=0 \textnormal{ if } |\bar{\mathfrak{d}}(x)|>\tau\}$.

\begin{theorem}[equivalent optimality conditions]\label{thm:equivalent_opt_cond}
Let $(\bar{u},\bar{p},\bar{q}) \in \tilde{H}^{s}(\Omega) \times \tilde{H}^{s}(\Omega) \times \mathbb{Q}_{ad}$ satisfy the first order optimality conditions \eqref{eq:weak_st_eq}, \eqref{eq:adj_eq}, and \eqref{eq:var_ineq_with_adj_state}. Then, the following statements are equivalent:
\begin{equation}
\label{eq:second_order_2_2}
j''(\bar{q})w^2 > 0 \, \forall w \in C_{\bar{q}}\setminus \{0\}
\Longleftrightarrow
\exists \mu, \tau >0: 
\,\,
j''(\bar{q})w^2 \geq \mu \|w\|_{L^2(\Omega)}^2 
\,\, 
\forall w \in C_{\bar{q}}^\tau.
\end{equation}
\end{theorem}
\begin{proof}
The proof of the equivalence \eqref{eq:second_order_2_2} follows from a combination of the arguments elaborated in the proofs of \cite[Theorem 25]{MR3586845}, Theorem \ref{thm:nec_opt_cond}, and Theorem \ref{thm:optimal_solution}. For the sake of brevity, we omit details.
\end{proof}


\section{Discretization schemes for the control problem}\label{sec:fem_control_problem}

We propose two different finite element discretization schemes to approximate solutions to the optimal control problem \eqref{eq:weak_min_problem}--\eqref{eq:weak_st_eq}: a fully discrete one, where the admissible control set is discretized with piecewise constant functions, and a semidiscrete scheme, based on the so-called variational discretization approach, where the control set is not discretized.

\subsection{A fully discrete scheme}
\label{sec:fully_discrete}
The solution technique reads as follows: Find $\min J(u_{h},q_{h})$ subject to the \emph{discrete state equation} 
\begin{equation}\label{eq:discrete_st_eq}
\mathcal{A}(u_{h},v_{h}) + (q_{h}u_{h},v_{h})_{L^2(\Omega)} = (f,v_{h})_{L^2(\Omega)} 
\quad \forall v_h \in \mathbb{V}_h,
\end{equation}
and the \emph{control constraints} $q_{h} \in \mathbb{Q}_{ad,h}$. Here, $\mathbb{Q}_{ad,h}:=\mathbb{Q}_{h}\cap \mathbb{Q}_{ad}$, where $\mathbb{Q}_{h}=\{ q_h \in L^\infty(\Omega): q_{h}|_T\in \mathbb{P}_0(T) \ \forall T\in \T_{h}\}$. We recall that $\mathbb{V}_h$, defined in \eqref{def:piecewise_linear_set}, corresponds to the space of standard continuous and piecewise linear functions that vanish on $\partial\Omega$.

The existence of a discrete solution follows from the compactness of $\mathbb{Q}_{ad,h}$ and the continuity of $J$. To formulate first order optimality conditions, we introduce the discrete control to state map $\mathcal{S}_{h}:  \mathbb{Q}_{ad,h} \ni q_h \mapsto u_h \in \mathbb{V}_{h}$, where $u_h$ is  the solution to \eqref{eq:discrete_st_eq}, and the reduced cost functional $j_{h}(q_{h}):=J(\mathcal{S}_{h}q_{h},q_{h})$. With these ingredients, the first order conditions read as follows: If $\bar{q}_{h} \in \mathbb{Q}_{ad,h}$ is a local solution, then
\begin{equation}
\label{eq:discrete_var_ineq}
j_{h}^{\prime}(\bar{q}_{h})(q_{h}-\bar{q}_{h}) = (\lambda\bar{q}_{h} - \bar{u}_{h}\bar{p}_{h}, q_{h}-\bar{q}_{h})_{L^2(\Omega)}  \geq  0 \quad \forall q_{h} \in \mathbb{Q}_{ad,h}.
\end{equation}
Here, $\bar{p}_{h} \in \mathbb{V}_{h}$ corresponds to the \emph{discrete adjoint state}, which solves
\begin{equation}\label{eq:discrete_adjoint_equation}
\mathcal{A}(v_{h},\bar{p}_{h}) + (\bar{q}_{h}\bar{p}_{h},v_{h})_{L^2(\Omega)} = (\bar{u}_{h} - u_{\Omega},v_{h})_{L^2(\Omega)} 
\quad \forall v_h \in \mathbb{V}_h.
\end{equation}


\subsection{A semidiscrete scheme}
\label{sec:variational_discretization}

In this section, we propose a semidiscrete scheme based on the so-called variational discretization approach \cite{MR2122182}. The scheme, \EO{in which only the state space is discretized} (the control space is not discretized), reads as follows: Find $\min J(u_h,\mathsf{q})$ subject to the \emph{discrete state equation}
\begin{equation*}
\mathcal{A}(u_{h},v_{h}) + (\mathsf{q}u_{h},v_{h})_{L^2(\Omega)} = (f,v_{h})_{L^2(\Omega)} 
\quad \forall v_h \in \mathbb{V}_h,
\end{equation*}
and the \emph{control constraints} $\mathsf{q}\in \mathbb{Q}_{ad}$. The existence of a discrete solution and first order optimality conditions for the semidiscrete scheme follow standard arguments. In particular, if $\bar{\mathsf{q}} \in \mathbb{Q}_{ad}$ denotes a local solution, then
\begin{equation}
\label{eq:discrete_var_ineq_variational}
j_{h}^{\prime}(\bar{\mathsf{q}})(q-\bar{\mathsf{q}}) = (\lambda\bar{\mathsf{q}} - \bar{u}_{h}\bar{p}_{h},q - \bar{\mathsf{q}})_{L^2(\Omega)}  \geq  0 \quad \forall q \in \mathbb{Q}_{ad},
\end{equation}
where $\bar{p}_{h} \in \mathbb{V}_{h}$ solves problem \eqref{eq:discrete_adjoint_equation}. We immediately note that, in view of the variational inequality \eqref{eq:discrete_var_ineq_variational}, the following projection formula holds \cite[Section 4.6]{Troltzsch}:
\begin{equation*}
\bar{\mathsf{q}}(x):=\Pi_{[a,b]}(\lambda^{-1}\bar{u}_h(x)\bar{p}_h(x)) \textrm{ a.e. } x \in \Omega.
\end{equation*}
Since $\bar{\mathsf{q}}$ implicitly depends on $h$, \EO{we will use the notation $\bar{\mathsf{q}}_h$ in the following.}


\subsection{Convergence of discretizations}
\label{sec:conver_of_disc}
In this section, we analyze the convergence properties of the fully and semidiscrete schemes. \EO{To this end}, we begin our investigation by providing some auxiliary convergence properties and error bounds for suitable finite element discretizations \EO{associated with} the state and adjoint equations.

\subsubsection{Auxiliary error bounds}
Let us first provide a convergence property and bounds for the errors $\| u -u_h \|_{s}$ and $\|u - u_{h}\|_{L^2(\Omega)}$.

\begin{theorem}[convergence properties]
\label{thm:error_estimates_state_aux}
Let $s\in(0,1)$ and let $f \in H^{-s}(\Omega)$. Let $u$ and $u_{h}$ be the unique solutions to problems \eqref{eq:weak_st_eq} and  \eqref{eq:discrete_st_eq}, respectively. Then,
\begin{equation}
 q_h \rightharpoonup q ~\textrm{in}~L^{\frac{d}{2s}}(\Omega),
 \quad
 h \rightarrow 0
 \implies
 u_h \rightarrow u ~\textrm{in}~\tilde{H}^{s}(\Omega), 
 \quad
 h \rightarrow 0.
 \label{eq:convergence_property_state}
\end{equation}
Let $r > d/2s$. If, in addition, $f\in L^2(\Omega) \cap L^r(\Omega)$, then we have the error bounds
\begin{align}
\|u - u_{h}\|_{s} &\lesssim h^{\gamma}|\log h|^{\varphi}\| f\|_{L^2(\Omega)} + \|q - q_{h}\|_{L^{2}(\Omega)}, \quad \gamma = \min\{s,\tfrac{1}{2}\},
\label{eq:global_estimate_state_s}
\\
\|u - u_{h}\|_{L^2(\Omega)} &\lesssim h^{2\gamma}|\log h|^{2\varphi}\| f\|_{L^2(\Omega)} + \|q - q_{h}\|_{L^{2}(\Omega)},
\label{eq:global_estimate_state_L2}
\end{align}
where $\varphi = \nu$ if $s\neq \frac{1}{2}$, $\varphi = 1 +\nu$ if $s=\frac{1}{2}$, and $\nu \geq \frac{1}{2}$ is the constant in Theorem \ref{thm:sobolev_reg}.
\end{theorem}   
\begin{proof}
Let us first derive the convergence property \eqref{eq:convergence_property_state}. We begin with a simple application of a triangle inequality to write 
$
 \| u - u_h \|_{s} 
 \leq 
 \| u - u_h(q) \|_{s}
 +
 \| u_h(q) - u_h \|_{s}.
$
Here, $u_h(q)$ denotes the solution to \eqref{eq:discrete_st_eq} with $q_h$ replaced by $q$. The control of $\| u - u_h(q) \|_{s}$ follows from \EO{a density} argument as the one developed in the proof of \cite[Theorem 3.2.3]{CiarletBook} (see also \cite[Corollary 1.109]{Guermond-Ern}): $\| u - u_h(q) \|_{s} \rightarrow 0$ as $h \rightarrow 0$. To bound $\| u_h(q) - u_h \|_{s}$, we invoke the discrete problem that $u_h(q) - u_h$ solves and utilize the fact that $q_h \in \mathbb{Q}_{ad,h} = \mathbb{Q}_{ad} \cap \mathbb{Q}_h$ to obtain
\[ 
\| u_h(q) - u_h \|_{s} \leq \| u_h(q)(q_h-q) \|_{H^{-s}(\Omega)}.
\]
The weak convergence $q_h \rightharpoonup q$ in $L^{\frac{d}{2s}}(\Omega)$
and the strong one $u_h(q)v \rightarrow uv$ in $L^{\frac{d}{d-2s}}(\Omega)$ as $h \rightarrow 0$, which \EO{is valid} for every $v \in \tilde{H}^{s}(\Omega)$, allow us to conclude.

\EO{We proceed similarly} to derive \eqref{eq:global_estimate_state_s}: 
$
\|u - u_{h}\|_{s} 
\leq 
\|u - u(q_h)\|_{s} 
+ 
\|u(q_{h})- u_{h}\|_{s}.
$
Here, $u(q_{h})$ denotes the solution to \eqref{eq:weak_st_eq} with $q$ replaced by $q_h$. Since $u_h$, the solution to \eqref{eq:discrete_st_eq}, corresponds to the finite element approximation of $u(q_{h})$ within the discrete setting of section \ref{sec:fem}, we can invoke the error bound of Theorem \ref{thm:error_estimates_frac_Lap} to deduce
\begin{equation}\label{eq:estimate_u-u_h}
\|u(q_{h}) - u_{h}\|_{s} \lesssim h^{\gamma}|\log h|^{\varphi}\| f\|_{L^2(\Omega)}.
\end{equation}
On the other hand, let us observe that $u - u(q_{h}) \in \tilde{H}^{s}(\Omega)$ uniquely solves
\begin{equation}
\label{eq:problem_u-hat_u}
\mathcal{A}(u - u(q_{h}),v) + (q(u - u(q_{h})),v)_{L^2(\Omega)} = (u(q_{h})(q_{h} - q),v)_{L^2(\Omega)}  \quad \forall v\in \tilde{H}^{s}(\Omega).
\end{equation}
Since $f \in L^r(\Omega)$ and $q_h \in \mathbb{Q}_{ad,h}$, Theorem \ref{thm:L_infty_reg} guarantees $\| u(q_{h}) \|_{L^{\infty}(\Omega)} \lesssim \| f\|_{L^r(\Omega)}$, which is a uniform bound with respect to discretization. As a consequence, a basic stability bound for problem \eqref{eq:problem_u-hat_u} yields $\|u - u(q_{h})\|_{s} \lesssim \|u(q_{h})\|_{L^{\infty}(\Omega)}\|q - q_{h}\|_{L^{2}(\Omega)} \lesssim \| f\|_{L^r(\Omega)} \|q - q_{h}\|_{L^{2}(\Omega)}$. This bound and estimate \eqref{eq:estimate_u-u_h} yield the desired bound \eqref{eq:global_estimate_state_s}. The proof of \eqref{eq:global_estimate_state_L2} follows similar arguments.
\end{proof}

\EO{To present the following result, we} introduce the variables $\mathfrak{p}$ and $p_h$ as follows:
 \begin{align}
 \label{eq:aux_adj_problem1*}
 \mathfrak{p} \in \tilde H^s(\Omega):
 \quad
 \mathcal{A}(v,\mathfrak{p}) + (q_{h}\mathfrak{p},v)_{L^2(\Omega)} = (u_{h} - u_{\Omega},v)_{L^2(\Omega)} 
 \quad \forall v \in \tilde{H}^{s}(\Omega).
\\
 p_h \in \mathbb{V}_h:
 \quad
\mathcal{A}(v_{h},p_{h}) + (q_{h}p_{h},v_{h})_{L^2(\Omega)} = (u_{h} - u_{\Omega},v_{h})_{L^2(\Omega)} 
\quad \forall v_h \in \mathbb{V}_h.
\label{eq:aux_adj_problem2}
\end{align}

In what follows, we present error bounds for $p-p_h$.

\begin{theorem}[convergence properties]
\label{thm:error_estimates_adj_aux}
Let $s\in(0,1)$ and $r > d/2s$. Let $p$ and $p_h$ be the unique solutions to problems \eqref{eq:adj_eq} and \eqref{eq:aux_adj_problem2}, respectively. If $f, u_{\Omega} \in L^2(\Omega) \cap L^r(\Omega)$ and $\{ u_h \}_{h>0}$ is uniformly bounded in $L^r(\Omega)$, then 
\begin{align}
\label{eq:global_estimate_adj_s}
\|p - p_{h}\|_{s} 
&
\lesssim 
h^{\gamma}|\log h|^{\varphi} + \|q - q_{h}\|_{L^{2}(\Omega)}, \quad \gamma = \min\{s,\tfrac{1}{2}\},
\\
\label{eq:global_estimate_adj_L2}
\|p - p_{h}\|_{L^2(\Omega)} 
& 
\lesssim h^{2\gamma}|\log h|^{2\varphi} + \|q - q_{h}\|_{L^{2}(\Omega)}, 
\end{align}
where $\varphi = \nu$ if $s\neq \frac{1}{2}$, $\varphi = 1 +\nu$ if $s=\frac{1}{2}$, and $\nu \geq \frac{1}{2}$ is the constant in Theorem \ref{thm:sobolev_reg}.
\end{theorem}   
\begin{proof}
We invoke $\mathfrak{p}$ and write
$
\|p - p_h\|_{s} \leq \|p - \mathfrak{p}\|_{s} + \|\mathfrak{p} - p_{h}\|_{s}.
$
To control $ \|\mathfrak{p}-p_h\|_{s}$ we apply the error estimate \eqref{eq:error_in_norm_s} of Theorem \ref{thm:error_estimates_frac_Lap}:
\begin{equation}\label{eq:estimate_p-p_h}
\|\mathfrak{p} - p_{h}\|_{s} 
\lesssim h^{\gamma}|\log h|^{\varphi}\left( \| f \|_{H^{-s}(\Omega)} + \|u_{\Omega} \|_{L^2(\Omega)}\right),
\end{equation}
upon using $\| u_h \|_s \lesssim \| f \|_{H^{-s}(\Omega)}$, which is uniform with respect to discretization. To control $\| p - \mathfrak{p} \|_s$ we observe that $p - \mathfrak{p}$ solves
\begin{equation*}
\mathcal{A}(v,p - \mathfrak{p}) + (q(p - \mathfrak{p}),v)_{L^2(\Omega)} = (\mathfrak{p}(q_{h} - q) + (u - u_{h}),v)_{L^2(\Omega)} \quad \forall v\in\tilde{H}^{s}(\Omega).
\end{equation*}
Since, by assumption, $\{ u_h \}_{h>0} $ is uniformly bounded in $L^{r}(\Omega)$ ($r>d/2s$), we obtain 
\begin{align}\label{eq:p-mathfrakp}
\| p - \mathfrak{p} \|_{s} &\lesssim \| \mathfrak{p} \|_{L^{\frac{d}{s}}(\Omega)}\|q - q_{h}\|_{L^{2}(\Omega)} + \|u - u_{h}\|_{L^2(\Omega)}
\\
&\lesssim \|q - q_{h}\|_{L^{2}(\Omega)}
+
h^{2\gamma}|\log h|^{2\varphi}\| f\|_{L^2(\Omega)},
\nonumber
\end{align}
upon using \eqref{eq:global_estimate_state_L2} and the uniform bound $\| \mathfrak{p} \|_{L^{\infty}(\Omega)} \lesssim \| u_h \|_{L^r(\Omega)} + \| u_{\Omega} \|_{L^r(\Omega)}$. The bound \eqref{eq:p-mathfrakp} \EO{in conjunction with} \eqref{eq:estimate_p-p_h} implies \eqref{eq:global_estimate_adj_s}. The proof of \eqref{eq:global_estimate_adj_L2} follows similar arguments. \EO{For the sake of brevity, we omit details.}
\end{proof}


\subsubsection{Convergence of discretizations: the fully discrete scheme} 
\label{sec:convergence_fullydiscrete}
We begin this section with a convergence result that essentially guarantees that a sequence of discrete global solutions
$\{\bar{q}_{h}\}_{h>0}$ contains subsequences that converge to global solutions of the optimal control problem \eqref{eq:weak_min_problem}--\eqref{eq:weak_st_eq} as $h\rightarrow 0$.

\begin{theorem}[convergence of global solutions]
\label{thm:convergence_discrete_sol}
Let $s\in(0,1)$, $r > d/2s$, and $f,u_{\Omega}\in L^2(\Omega)\cap L^r(\Omega)$. Let $h>0$ and let $\bar{q}_h\in\mathbb{Q}_{ad,h}$ be a global solution of the fully discrete optimal control problem. Then, there exist nonrelabeled subsequences of $\{\bar{q}_{h}\}_{h>0}$ such that $\bar{q}_h \mathrel{\ensurestackMath{\stackon[1pt]{\rightharpoonup}{\scriptstyle\ast}}} \bar{q}$ in the weak$^\star$ topology of $L^\infty(\Omega)$ as $h \rightarrow 0$, \EO{where} $\bar{q}$ \EO{is} a global solution of the optimal control problem \eqref{eq:weak_min_problem}--\eqref{eq:weak_st_eq}. Furthermore, we have
\begin{equation}\label{eq:discrete_cont_convergence}
\lim_{h \rightarrow 0}\|\bar{q}-\bar{q}_{h}\|_{L^2(\Omega)}  = 0, \qquad
\lim_{h \rightarrow 0}j_{h}(\bar{q}_{h}) = j(\bar{q}).
\end{equation}
\end{theorem}
\begin{proof}
Since $\{\bar{q}_{h}\}_{h>0} \subset \mathbb{Q}_{ad,h}$ is uniformly bounded in $L^\infty(\Omega)$, there exists a nonrelabeled subsequence such that $\bar{q}_{h} \mathrel{\ensurestackMath{\stackon[1pt]{\rightharpoonup}{\scriptstyle\ast}}} \bar{q}$ in $L^\infty(\Omega)$ as $h \rightarrow 0$. In what follows, we prove that $\bar{q}\in \mathbb{Q}_{ad}$ is a global solution to \eqref{eq:weak_min_problem}--\eqref{eq:weak_st_eq} and that \eqref{eq:discrete_cont_convergence} holds.

Let $\mathfrak{q} \in \mathbb{Q}_{ad}$ be a global solution to \eqref{eq:weak_min_problem}--\eqref{eq:weak_st_eq}. Let $\mathcal{P}_h: L^2(\Omega)\to \mathbb{Q}_{h}$ be the orthogonal projection operator. Define $\mathfrak{q}_h:= \mathcal{P}_h(\mathfrak{q})$ and note that $\mathfrak{q}_h \in \mathbb{Q}_{ad,h}$. \EO{Applying} the regularity result from Theorem \ref{thm:regul_control} and a standard error estimate for $\mathcal{P}_h$, \EO{we} obtain $\|\mathfrak{q}-\mathfrak{q}_h\|_{L^2(\Omega)} \to 0$ as $h \rightarrow 0$. Therefore, the global optimality of $\mathfrak{q}$, Theorem \ref{thm:error_estimates_state_aux}, the global optimality of $\bar{q}_{h}$, and $\mathfrak{q}_{h} \rightarrow \mathfrak{q}$ in $L^2(\Omega)$ allow us to obtain that
\[
j(\mathfrak{q}) 
\leq 
j(\bar{q}) 
\leq 
\liminf_{h \downarrow 0} j_{h}(\bar{q}_{h}) 
\leq 
\limsup_{h \downarrow 0} j_{h}(\bar{q}_{h}) 
\leq 
\limsup_{h \downarrow 0} j_{h}(\mathfrak{q}_h) =  j(\mathfrak{q}).
\]
This shows that $\bar{q}$ is a global solution of \eqref{eq:weak_min_problem}--\eqref{eq:weak_st_eq} and that \EO{$j_{h}(\bar{q}_{h}) \rightarrow j(\bar{q})$ as $h \rightarrow 0$.} To prove the strong convergence of $\{ \bar{q}_h \}_{h>0}$ to $\bar q$ in $L^2(\Omega)$, we use the convergence result of Theorem \ref{thm:error_estimates_state_aux} to deduce that $\bar{u}_{h} \to \bar{u}$ in $L^{\mathfrak{t}}(\Omega)$ for every $\mathfrak{t} \leq 2d/(d-2s)$. With this convergence result, we use that $j_{h}(\bar{q}_{h}) \rightarrow j(\bar{q})$ to obtain $\| \bar{q}_h \|_{L^2(\Omega)} \rightarrow \| \bar{q} \|_{L^2(\Omega)}$ \EO{as $h \rightarrow 0$}. The fact that $\bar{q}_h \mathrel{\ensurestackMath{\stackon[1pt]{\rightharpoonup}{\scriptstyle\ast}}} \bar{q}$ in $L^{\infty}(\Omega)$ as $h \rightarrow 0$ allows us to conclude.
\end{proof}

Our second convergence result is as follows: strict local solutions of \eqref{eq:weak_min_problem}--\eqref{eq:weak_st_eq} can be approximated by local solutions of the fully discrete optimal control problems.

\begin{theorem}[convergence of local solutions]
\label{thm:convergence_discrete_sol_local_fully}
Let the assumptions of Theorem \ref{thm:convergence_discrete_sol} hold. Let $\bar{q}\in\mathbb{Q}_{ad}$ be a strict local minimum of \eqref{eq:weak_min_problem}--\eqref{eq:weak_st_eq}. Then, there exists a sequence of local minima \EO{$\{\bar{q}_h\}_{h<h_{\Box}}$} of the fully discrete scheme such that \eqref{eq:discrete_cont_convergence} holds.
\end{theorem}
\begin{proof}
Since $\bar{q}$ is a strict local minimum of \eqref{eq:weak_min_problem}--\eqref{eq:weak_st_eq}, there exists $\varepsilon > 0$ such that the problem: Find $\min\{ j(q): q \in \mathbb{Q}_{ad}\cap B_{\varepsilon}(\bar{q})\}$ admits as a unique solution $\bar{q}$. Here, $B_{\varepsilon}(\bar{q}):=\{ q \in L^2(\Omega): \|\bar{q} - q\|_{L^2(\Omega)}\leq \varepsilon\}$. Let us now introduce, for $h>0$, the discrete problem: Find
$
\min\{j_{h}(q_h): q_h\in\mathbb{Q}_{ad,h}\cap B_{\varepsilon}(\bar{q})\}.
$
To conclude that this problem admits at least \EO{one} solution, we need to verify that the set \EO{in which} the minimum is sought is nonempty; note that such a set is compact. To do this, we note that there exists $h_{\varepsilon}>0$ such that, for every $h \leq h_{\varepsilon}$, $\mathcal{P}_h(\bar{q}) \in \mathbb{Q}_{ad,h} \cap B_{\varepsilon}(\bar{q})$.

Let $h \in (0,h_{\varepsilon}]$ and let $\bar{q}_h$ be a global solution of the introduced discrete problem. The arguments \EO{presented} in the proof \EO{of} Theorem \ref{thm:convergence_discrete_sol} allow us to conclude the existence of a subsequence of $\{\bar{q}_h\}_{h\leq h_\varepsilon}$ \EO{that converges strongly} in $L^2(\Omega)$ to a solution of $\min\{j(q): q \in\mathbb{Q}_{ad}\cap B_{\varepsilon}(\bar{q})\}$. Since this problem admits a unique solution $\bar{q}$, we must have $\bar{q}_h \rightarrow \bar{q}$ in $L^2(\Omega)$ as $h \rightarrow 0$. In particular, this guarantees that the constraint $\bar{q}_{h} \in B_{\varepsilon}(\bar{q})$ is not active for $h$ sufficiently small. Consequently, $\bar{q}_h$ solves the fully discrete scheme and the convergence properties stated in \eqref{eq:discrete_cont_convergence} hold.
\end{proof}


\subsubsection{Convergence of discretizations: the semidiscrete scheme}
\label{sec:convergence_semidiscrete} 
We present the following results. Let $s\in(0,1)$, $r > d/2s$, and $f,u_{\Omega}\in L^2(\Omega)\cap L^r(\Omega)$.

\begin{itemize}
 \item Let $h>0$ and let $\bar{\mathsf{q}}_h \in \mathbb{Q}_{ad}$ be a global solution of the semidiscrete scheme. Then, there exist nonrelabeled subsequences of $\{\bar{\mathsf{q}}_h\}_{h>0}$ such that $\bar{\mathsf{q}}_h \mathrel{\ensurestackMath{\stackon[1pt]{\rightharpoonup}{\scriptstyle\ast}}} \bar{q}$ in $L^{\infty}(\Omega)$ as $h \rightarrow 0$ and \eqref{eq:discrete_cont_convergence} holds; $\bar{q}$ is a global solution to \eqref{eq:weak_min_problem}--\eqref{eq:weak_st_eq}.
 \item Let $\bar{q} \in\mathbb{Q}_{ad}$ be a strict local minimum of \eqref{eq:weak_min_problem}--\eqref{eq:weak_st_eq}. Then, there exists a sequence of local minima $\{\bar{\mathsf{q}}_h\}_{h>0}$ of the semidiscrete scheme satisfying \eqref{eq:discrete_cont_convergence}.
\end{itemize}
The proof of these results follows very similar arguments to \EO{those} developed in the proof of Theorems \ref{thm:convergence_discrete_sol} and \ref{thm:convergence_discrete_sol_local_fully}. \EO{For the sake of brevity, we will omit further details.}


\section{Error estimates}\label{sec:error_estimates}

In this section, we derive error estimates for the fully and semidiscrete schemes introduced in sections \ref{sec:fully_discrete} and \ref{sec:variational_discretization}, respectively.


\subsection{Error estimates for the fully discrete scheme}

Let $\{\bar{q}_h\}_{h>0} \subset \mathbb{Q}_{ad,h}$ be a sequence of local minima of the fully discrete optimal control problems such that $\bar{q}_h \to \bar{q}$ in $L^2(\Omega)$ as $h \rightarrow 0$;  $\bar{q}\in\mathbb{Q}_{ad}$ denotes a local solution of \eqref{eq:weak_min_problem}--\eqref{eq:weak_st_eq} (see Theorems \ref{thm:convergence_discrete_sol} and  \ref{thm:convergence_discrete_sol_local_fully}). The main goal of this section is to derive the error bound
\begin{equation}\label{eq:control_error_estimate_fully}
\|\bar{q}-\bar{q}_h\|_{L^2(\Omega)}\lesssim h^{2\gamma}|\log h|^{2\varphi} \quad \forall h < h_{\ddagger},
\qquad
\gamma = \min\{s,\tfrac{1}{2}\}, 
\qquad 
h_{\ddagger}>0. 
\end{equation}
Here, $\varphi = \nu$ if $s\neq \frac{1}{2}$, $\varphi = 1 +\nu$ if $s=\frac{1}{2}$, and $\nu \geq \frac{1}{2}$ is the constant in Theorem \ref{thm:sobolev_reg}.

In the \EO{following analysis we \emph{assume}} that discrete solutions $u_{h}$ to problem \eqref{eq:discrete_st_eq} are uniformly bounded in $L^{d/s}(\Omega)$, i.e,
\begin{equation}\label{eq:assumption_uniform_L-infty_bound}
\exists C > 0:
\quad
\|u_{h}\|_{L^{d/s}(\Omega)} \leq C
\quad \forall h>0.
\end{equation}

The following result \EO{is helpful for deriving} the error bound \eqref{eq:control_error_estimate_fully}.

\begin{lemma}[auxiliary error estimate]\label{lemma:aux_error_estimate}
Let $s\in(0,1)$, $r > d/2s$, and $f,u_{\Omega} \in L^2(\Omega) \cap L^r(\Omega)$. Let us assume that \eqref{eq:assumption_uniform_L-infty_bound} holds and that  $\bar{q}\in\mathbb{Q}_{ad}$ \EO{satisfies} the second order optimality conditions \eqref{eq:second_order_2_2}. If \eqref{eq:control_error_estimate_fully} is false, then there exists $h_{\dagger} > 0$ such that
\begin{equation}\label{eq:aux_estimate}
\textgoth{C}
\|\bar{q}-\bar{q}_h\|_{L^2(\Omega)}^2\leq [j'(\bar{q}_h)-j'(\bar{q})](\bar{q}_h-\bar{q}) \quad \forall h < h_{\dagger},
\quad
\textgoth{C} = 2^{-1} \min\{\mu,\lambda\}.
\end{equation}
Here, $\lambda$ is the control cost and $\mu$ denotes the constant appearing in \eqref{eq:second_order_2_2}.
\end{lemma}
\begin{proof}
We follow \cite[Section 7]{MR2272157} and proceed by contradiction. \EO{Since} by assumption \eqref{eq:control_error_estimate_fully} is false, there exists a subsequence $\{h_k\}_{k \in \mathbb{N}} \subset \mathbb{R}^{+}$ such that 
\begin{equation}\label{eq:estimate_contradiction}
\lim_{ h_{k}\rightarrow 0}\|\bar{q}-\bar{q}_{h_k}\|_{L^2(\Omega)} \EO{= 0}, 
\qquad
\lim_{h_{k} \rightarrow 0}\frac{\|\bar{q}-\bar{q}_{h_{k}}\|_{L^2(\Omega)}}{ h_{k}^{2\gamma}|\log h_{k}|^{2\varphi}}=+\infty.
\end{equation}
\EO{To} simplify the exposition of the material, we \EO{omit} the subindex $k$ \EO{in the following} and denote $\bar{q}_{h_k} = \bar{q}_{h}$. Note that $h \rightarrow 0$ as $k \uparrow \infty$.

Define $w_h:= (\bar{q}_h-\bar{q})/\|\bar{q}_h-\bar{q}\|_{L^2(\Omega)}$. Since $\{ w_h \}_{h>0}$ is uniformly bounded in $L^2(\Omega)$, \EO{possibly} up to a subsequence, we can assume that 
$
w_{h}  \rightharpoonup w
$
in $L^{2}(\Omega)$ as $h \rightarrow 0$. In what follows, we prove that $w\in C_{\bar{q}}$. Recall that $C_{\bar{q}}$ is defined in \eqref{def:critical_cone}. Since, for each $h>0$, $\bar{q}_h \in \mathbb{Q}_{ad,h}\subset \mathbb{Q}_{ad}$, $w_h$ satisfies the sign conditions \eqref{eq:sign_cond}. Consequently, $w$ \EO{also} satisfies \eqref{eq:sign_cond}. To show that $\bar{\mathfrak{d}}(x) \neq 0$ implies that $w(x) = 0$ for a.e.~$x\in\Omega$, we introduce $\bar{\mathfrak{d}}_h:= \lambda\bar{q}_h - \bar{u}_{h}\bar{p}_{h}$ and recall that $\bar{\mathfrak{d}}= \lambda\bar{q}- \bar{u}\bar{p}$. Let us now invoke $\|\bar{q}-\bar{q}_{h}\|_{L^2(\Omega)}\to 0$ as $h\rightarrow 0$, \eqref{eq:assumption_uniform_L-infty_bound}, and estimates \eqref{eq:global_estimate_adj_s} and \eqref{eq:global_estimate_state_L2} to obtain
\begin{multline*}
\| \bar{\mathfrak{d}} - \bar{\mathfrak{d}}_h\|_{L^2(\Omega)} 
\leq \lambda\|\bar{q} - \bar{q}_{h}\|_{L^2(\Omega)}  + \|\bar{u}_{h}\|_{L^{\frac{d}{s}}(\Omega)} \| \bar{p} - \bar{p}_{h} \|_{L^{\frac{2d}{d-2s}}(\Omega)} 
\\
+ \|\bar{p}\|_{L^{\infty}(\Omega)} \| \bar{u} - \bar{u}_{h} \|_{L^2(\Omega)} 
 \lesssim \|\bar{q}-\bar{q}_{h}\|_{L^2(\Omega)} + h^{\gamma}|\log h|^{\varphi}  \to 0, \qquad h\rightarrow 0.
\end{multline*}
Hence, we have that $\bar{\mathfrak{d}}_{h} \to \bar{\mathfrak{d}}$ in $L^2(\Omega)$ as $h\rightarrow 0$. \EO{From this follows}
\begin{align*}
\int_\Omega \bar{\mathfrak{d}}(x)w(x) \mathrm{d}x
&=\lim_{h \rightarrow 0}\int_\Omega\bar{\mathfrak{d}}_h(x)w_h(x) \mathrm{d}x
\\
&=\lim_{h \rightarrow 0}\frac{1}{\|\bar{q}_h-\bar{q}\|_{L^2(\Omega)}}
\left(\int_\Omega\bar{\mathfrak{d}}_h (\mathcal{P}_h(\bar{q})-\bar{q})\mathrm{d}x 
+ \int_\Omega\bar{\mathfrak{d}}_h(\bar{q}_h - \mathcal{P}_h(\bar{q})) \mathrm{d}x\right),
\end{align*}
where $\mathcal{P}_h$ denotes the $L^2$-orthogonal projection operator into piecewise constant functions over $\T_h$. Since $\mathcal{P}_h(\bar{q})\in \mathbb{Q}_{ad,h}$, the discrete variational inequality \eqref{eq:discrete_var_ineq} implies that $ (\bar{\mathfrak{d}}_h,\bar{q}_h - \mathcal{P}_h(\bar{q}))_{L^2(\Omega)}  \leq 0$. The uniform boundedness of $\{ \|\bar{\mathfrak{d}}_{h}\|_{L^2(\Omega)} \}_{h < h_{\bowtie}}$, which follows from $\|\bar{\mathfrak{d}}_{h}\|_{L^2(\Omega)}\leq \|\bar{\mathfrak{d}}_{h}-\bar{\mathfrak{d}}\|_{L^2(\Omega)}+\|\bar{\mathfrak{d}}\|_{L^2(\Omega)} \lesssim 1$ for $h < h_{\bowtie}$, thus implies that
\begin{equation*}
\int_\Omega \bar{\mathfrak{d}}(x) w(x) \mathrm{d}x
\leq 
\lim_{h \rightarrow 0}\frac{1}{\|\bar{q}_h-\bar{q}\|_{L^2(\Omega)}}
\int_\Omega\bar{\mathfrak{d}}_h (\mathcal{P}_h(\bar{q})-\bar{q})\mathrm{d}x 
\lesssim
\lim_{h \rightarrow 0}
\frac{\|\mathcal{P}_h(\bar{q})-\bar{q}\|_{L^2(\Omega)}}{\|\bar{q}_h-\bar{q}\|_{L^2(\Omega)}}
=0.
\end{equation*}
To obtain the last equality, we have used \eqref{eq:estimate_contradiction} and the regularity results for $\bar{q}$ provided in Theorem \ref{thm:regul_control} in conjunction with standard error estimates for $\mathcal{P}_h$, which yield
\begin{equation}
 \|\bar{q} - \mathcal{P}_h(\bar{q}) \|_{L^2(\Omega)} \lesssim \EO{h^{2\gamma}|\log h|^{\nu} (1 + \Lambda(f,u_{\Omega})), \quad
\gamma = \min\{s,\tfrac{1}{2}\}.}
  \label{eq:error_estimate_Ph}
\end{equation}
Here, $\Lambda(f,u_{\Omega})$ is defined in \eqref{eq:Lambda}. Now, since $w$ satisfies the sign conditions \eqref{eq:sign_cond}, we have $\bar{\mathfrak{d}}(x)w(x) \geq 0$ \EO{for a.e.~$x \in \Omega$}. Consequently, $\int_{\Omega} |\bar{\mathfrak{d}}(x)w(x)| \mathrm{d}x \leq 0$. As a result, if $\bar{\mathfrak{d}}(x)\neq 0$, then $w(x) = 0$ for a.e.~$x\in\Omega$. We have thus obtained that $w \in C_{\bar{q}}$.

\EO{Now that we have proven that} $w \in C_{\bar{q}}$, let us derive the auxiliary error estimate \eqref{eq:aux_estimate} by using \eqref{eq:second_order_2_2}. To begin, we apply the mean value theorem to obtain 
\begin{equation}\label{eq:difference_of_j}
[j'(\bar{q}_h)-j'(\bar{q})](\bar{q}_h-\bar{q})= j''(\hat{q}_h)(\bar{q}_h-\bar{q})^2, \qquad \hat{q}_h = \bar{q} + \theta_h (\bar{q}_h - \bar{q}),
\end{equation}
where $\theta_h \in (0,1)$. Let $u(\hat{q}_h)$ be unique solution to \eqref{eq:weak_st_eq} with $q$ replaced by $\hat{q}_{h}$ and let $p(\hat{q}_h)$ be the unique solution to \eqref{eq:adj_eq} with $u$ and $q$ replaced by $u(\hat{q}_h)$ and $\hat{q}_h$, respectively. Since $\hat{q}_h(x) \geq a$ for a.e.~$x \in \Omega$ and $\bar u \in L^{\infty}(\Omega)$, a basic stability bound for the problem that $\bar{u} - u(\hat{q}_h)$ solves combined with the fact that $\bar{q}_h \rightarrow \bar{q}$ in $L^2(\Omega)$ as $h \rightarrow 0$ reveal that $u(\hat{q}_h) \rightarrow \bar{u}$ in $\tilde{H}^{s}(\Omega)$. Similarly, $p(\hat{q}_h) \rightarrow \bar{p}$ in $\tilde{H}^{s}(\Omega)$ as $h \rightarrow 0$. We now define $z(w_h)$ as the unique solution to \eqref{eq:first_der_S} with $q$, $u$ and $w$ replaced by $\hat{q}_{h}$, $u(\hat{q}_h)$, and $w_h$, respectively. Proceeding as in the Step 2 of the proof of Theorem \ref{thm:optimal_solution}, we can show that $w_{h} \rightharpoonup w$ in $L^2(\Omega)$ as $h \rightarrow 0$ guarantees that $z(w_h) \rightharpoonup z$ in $\tilde{H}^{s}(\Omega)$. With all these convergence properties at hand, we \EO{can refer to} the characterization \eqref{eq:charac_j2}, the definition of $w_h$, and the second order condition \eqref{eq:second_order_2_2} to obtain
\begin{multline*}
\lim_{h \rightarrow 0}j''(\hat{q}_h)w_h^2 
 = 
\lim_{h \rightarrow 0}
\left[
 \lambda\|w_{h}\|_{L^2(\Omega)}^2 - 2(w_{h}z(w_{h}),p(\hat{q}_{h}))_{L^2(\Omega)} + \|z(w_{h})\|_{L^2(\Omega)}^2
 \right]
 \\
= \lambda
- 2(w z,\bar{p})_{L^2(\Omega)} + \|z\|_{L^2(\Omega)}^2
= 
\lambda + j''(\bar{q})w^2 - \lambda\|w\|_{L^2(\Omega)}^2
\geq \lambda +(\mu-\lambda) \|w\|_{L^2(\Omega)}^2.
\end{multline*}
Therefore, since $\|w\|_{L^2(\Omega)}\leq 1$, we obtain
$
\lim_{h \rightarrow 0}j''(\hat{q}_h)w_{h}^2 \geq \min\{\mu,\lambda\}>0.
$
We can therefore obtain the existence of $h_{\dagger} > 0$ such that
$
j''(\hat{q}_h)w_{h}^2 \geq 
\min\{\mu,\lambda\}/2
$
for every
$
h< h_{\dagger}.
$
\EO{Given the definition of $w_h$ and \eqref{eq:difference_of_j}, this allows the conclusion to be drawn.}
\end{proof}

We now derive the error bound \eqref{eq:control_error_estimate_fully}. \EO{For this purpose}, we strengthen the assumption \eqref{eq:assumption_uniform_L-infty_bound} as follows: solutions to \eqref{eq:discrete_st_eq} are uniformly bounded in $L^{\infty}(\Omega)$, i.e,
\begin{equation}\label{eq:assumption_uniform_L-infty_bound_2}
\exists C > 0:
\quad
\|u_{h}\|_{L^{\infty}(\Omega)} \leq C
\quad \forall h>0.
\end{equation}
\DQ{
Unfortunately, we are not aware of a proof of \eqref{eq:assumption_uniform_L-infty_bound_2} in a general setting. A basic proof can be given for the case where $d=2$ and $s > 0.5+\varepsilon$, with $\varepsilon > 0$ as in Theorem \ref{thm:sobolev_reg}. This proof is based on inverse estimates, error estimates for the Lagrange interpolation operator, and the $L^2(\Omega)$ error estimate \eqref{eq:error_in_norm_L2}. Finally, we would like to note that for a fixed $h$ the function $u_h$ is a globally continuous piecewise linear function that is bounded in $L^{\infty}(\Omega)$ and that $\{u_h\}_{h>0}$ converges strongly in $H^s(\mathbb{R}^n)$ as $h \rightarrow 0$ to a function $u$ belonging to $L^{\infty}(\Omega)$; see Theorem \ref{thm:L_infty_reg}, \S \ref{sec:convergence_fullydiscrete} and \S \ref{sec:convergence_semidiscrete}, and Theorem \ref{thm:error_estimates_state_aux}. }

\begin{theorem}[error estimate]\label{thm:error_estimate_fully}
Let $s\in(0,1)$, $r > d/2s$, and $f,u_{\Omega}\in L^{2}(\Omega) \cap L^r(\Omega)$. Let us assume that \eqref{eq:assumption_uniform_L-infty_bound_2} holds and that $\bar{q}\in\mathbb{Q}_{ad}$ satisfies the second order optimality conditions \eqref{eq:second_order_2_2}. Then, there exists $h_{\ddagger} > 0$ such that 
\begin{equation}
\label{eq:error_estimates_control_final_fd}
\|\bar{q}-\bar{q}_h\|_{L^2(\Omega)}\lesssim h^{2\gamma} |\log h|^{2\varphi}
\qquad
\forall h < h_{\ddagger},
\qquad
\gamma = \min\{ s, \tfrac{1}{2} \}.
\end{equation}
Here, $\varphi = \nu$ if $s\neq \frac{1}{2}$, $\varphi = 1 +\nu$ if $s=\frac{1}{2}$, and $\nu \geq \frac{1}{2}$ is the constant in Theorem \ref{thm:sobolev_reg}.
\end{theorem}
\begin{proof}
We proceed by contradiction and assume that \eqref{eq:error_estimates_control_final_fd} does not hold. We can \EO{therefore refer to} the result of Lemma \ref{lemma:aux_error_estimate} to conclude that \eqref{eq:aux_estimate} holds for every $h < h_{\dagger}$. Let us now set $q=\bar{q}_{h}$ in \eqref{eq:variational_inequality} and $q_{h} = \mathcal{P}_{h}(\bar{q})$ in \eqref{eq:discrete_var_ineq} to obtain $- j'(\bar{q})(\bar{q}_{h} - \bar{q}) \leq 0$ and $j_{h}^{\prime}(\bar{q}_{h})(\mathcal{P}_{h}(\bar{q}) - \bar{q}_{h}) \geq 0$, respectively. Given these two inequalities we invoke the auxiliary error estimate \eqref{eq:aux_estimate} to obtain
\begin{equation}
\label{eq:FD_ct_error_estimate_0}
\begin{aligned}
\|\bar{q} - \bar{q}_{h}\|_{L^2(\Omega)}^2 
& \lesssim j'(\bar{q}_{h})(\bar{q}_{h} - \bar{q}) + j_{h}^{\prime}(\bar{q}_{h})(\mathcal{P}_{h}(\bar{q}) - \bar{q}_{h})
\\
& = j'(\bar{q}_{h})(\mathcal{P}_h(\bar{q}) - \bar{q}) + [j'(\bar{q}_{h}) - j_{h}'(\bar{q}_{h})](\bar{q}_{h} - \mathcal{P}_{h}(\bar{q}))=:\mathbf{I}_h + \mathbf{II}_h.
 \end{aligned}
\end{equation}

Let us first bound the term $\mathbf{I}_h$. For this purpose, we introduce the following auxiliary variables: We define $\hat{u}$ and $\hat{p}$ to be the solutions to the problems
\begin{align}
\label{eq:fd_aux_hat_u}
\mathcal{A}(\hat{u},v) + (\bar{q}_{h}\hat{u},v)_{L^{2}(\Omega)} & = (f,v)_{L^2(\Omega)} 
\quad \forall v \in \tilde{H}^{s}(\Omega),
\\
\label{eq:fd_aux_hat_p}
\mathcal{A}(\hat{p},v) + (\bar{q}_{h}\hat{p},v)_{L^{2}(\Omega)} & = (\hat{u} - u_{\Omega},v)_{L^2(\Omega)} 
\quad \forall v \in \tilde{H}^{s}(\Omega).
\end{align}
With these variables at hand, we control the term $\mathbf{I}_h$ as follows:
\begin{align*}
\mathbf{I}_h & = (\lambda\bar{q}_{h} - \hat{u}\hat{p}, \mathcal{P}_{h}(\bar{q}) - \bar{q})_{L^{2}(\Omega)} = (\mathcal{P}_{h}(\hat{u}\hat{p}) - \hat{u}\hat{p}, \mathcal{P}_{h}(\bar{q}) - \bar{q})_{L^{2}(\Omega)} \\
& \leq \|\hat{u}\hat{p} - \mathcal{P}_{h}(\hat{u}\hat{p})\|_{L^{2}(\Omega)}
\|\mathcal{P}_{h}(\bar{q}) - \bar{q} \|_{L^{2}(\Omega)},
\end{align*}
upon using the basic property $(\mathcal{P}_{h}(\bar{q}) - \bar{q},q_h) = 0$ \EO{for all $q_h \in \mathbb{Q}_{h}$.} The term $\|\mathcal{P}_{h}(\bar{q}) - \bar{q} \|_{L^{2}(\Omega)}$ was previously controlled in the proof of Lemma \ref{lemma:aux_error_estimate}; see 
\eqref{eq:error_estimate_Ph}. In what follows we control $\|\hat{u}\hat{p} - \mathcal{P}_{h}(\hat{u}\hat{p})\|_{L^{2}(\Omega)}$. First, since $f$ and $\hat{u} - u_{\Omega}$ belong to $L^2(\Omega) \cap L^r(\Omega)$, we deduce from the regularity results of Theorems \ref{thm:sobolev_reg} and \ref{thm:L_infty_reg} in conjunction with the continuity of the product property of Lemma \ref{lemma:product_of_func} that
\begin{equation*}
\hat{u}\hat{p} \in H^{s+\kappa-\varepsilon}(\Omega)\cap L^{\infty}(\Omega), 
\quad 
\|\hat{u}\hat{p}\|_{H^{s + \kappa - \varepsilon}(\Omega)} \lesssim  \varepsilon^{-\nu}
\EO{\Lambda(f,u_{\Omega})}
\quad
\forall 0 < \varepsilon < s.
\end{equation*}
Here, $\kappa$ and $\nu$ are as in the statement of Theorem \ref{thm:sobolev_reg} and $\Lambda(f,u_{\Omega})$ is defined in \eqref{eq:Lambda}. In view of this regularity result, a standard approximation property for $\mathcal{P}_{h}$ shows that $\mathfrak{e}:= \hat{u}\hat{p} - \mathcal{P}_{h}(\hat{u}\hat{p})$ satisfies \EO{$
\| \mathfrak{e}\|_{L^{2}(\Omega)}  \lesssim h^{2\gamma}|\log h|^{\nu} \Lambda(f,u_{\Omega})$, where $\gamma = \min\{s,\tfrac{1}{2}\}.$}
We can therefore deduce that
\[
\EO{\mathbf{I}_h \lesssim h^{4\gamma}|\log h|^{2\nu} \Lambda(f,u_{\Omega})(1 + \Lambda(f,u_{\Omega})), \quad \gamma = \min\{s,\tfrac{1}{2}\}.}
\]

We now bound the term $\mathbf{II}_{h}$. To accomplish this task, we first notice that $\mathbf{II}_{h} = ( \hat{u}\hat{p} - \bar{u}_{h}\bar{p}_{h}, \bar{q}_{h} - \mathcal{P}_{h}(\bar{q}))_{L^{2}(\Omega)}
=
(\hat{u}\hat{p} - \bar{u}_{h}\bar{p}_{h},\mathcal{P}_{h}(\bar{q}_h - \bar{q}))_{L^{2}(\Omega)}
$. Secondly, by adding and subtracting $\hat{p} \bar{u}_h$ and using the assumption \eqref{eq:assumption_uniform_L-infty_bound_2} and basic inequalities, we arrive at
\begin{align*}
\label{eq:II_estimate_1}
\mathbf{II}_{h} 
& \leq \left( \|\hat{p}\|_{L^{\infty}(\Omega)}\|\hat{u}-\bar{u}_{h}\|_{L^{2}(\Omega)} 
+ 
\|\bar{u}_{h}\|_{L^{\infty}(\Omega)}\| \hat{p} - \bar{p}_{h}\|_{L^{2}(\Omega)}
\right)
\|\bar{q} - \bar{q}_{h}\|_{L^{2}(\Omega)}
\\
& 
\leq
\textswab{C}
\left(
\| \hat{u} - \bar{u}_{h} \|^2_{L^{2}(\Omega)} + 
\| \hat{p} - \bar{p}_{h} \|^2_{L^{2}(\Omega)} 
\right)
+
\tfrac{1}{2}
\|\bar{q} - \bar{q}_{h}\|^2_{L^{2}(\Omega)}, \qquad \textswab{C} > 0.
\end{align*}
A bound for $\| \hat{u} - \bar{u}_{h} \|_{L^{2}(\Omega)}$ follows from the fact that $\bar{u}_{h}$ corresponds to the finite element approximation of $\hat{u}$ \EO{in the framework} of section \ref{sec:fem}. In fact, as an application of Theorem \ref{thm:error_estimates_frac_Lap}, we have the bound
$
\|\hat{u} - \bar{u}_{h}\|_{L^{2}(\Omega)} \lesssim h^{2\gamma}|\log h|^{2\varphi}.
$
We now define  
\begin{equation}
\tilde{p} \in \tilde{H}^s(\Omega):
\quad
 \mathcal{A}(\tilde{p},v) + (\bar{q}_{h}\tilde{p},v)_{L^2(\Omega)} = (\bar{u}_{h} - u_{\Omega},v)_{L^2(\Omega)}
 \quad
 \forall v \in \tilde{H}^{s}(\Omega)
 \label{eq:tilde_p}
\end{equation}
and proceed as in the proof of Theorem \ref{thm:error_estimates_adj_aux} to obtain 
$
\|\hat{p} - \bar{p}_{h}\|_{L^{2}(\Omega)} 
\leq 
\|\hat{p} - \tilde{p} \|_{L^{2}(\Omega)} 
+
\|\tilde{p} - \bar{p}_{h}\|_{L^{2}(\Omega)} 
\lesssim 
 \| \hat{u} - \EO{\bar{u}_h} \|_{L^2(\Omega)}
 +
 h^{2\gamma}|\log h|^{2\varphi}.
$
Consequently, we can thus arrive at
\begin{equation*}
\mathbf{II}_{h} \lesssim  h^{4\gamma}|\log h|^{4\varphi} + \tfrac{1}{2}\|\bar{q} - \bar{q}_{h}\|_{L^{2}(\Omega)}^{2}.
\end{equation*}

Finally, we replace the bounds obtained for $\mathbf{I}_{h}$ and $\mathbf{II}_{h}$ into 
\eqref{eq:FD_ct_error_estimate_0} to obtain \eqref{eq:error_estimates_control_final_fd}. This is a contradiction and concludes the proof.
\end{proof}

\begin{remark}[nearly optimality]
The error bound \eqref{eq:error_estimates_control_final_fd} behaves like $\mathcal{O}(h|\log h|^{2\varphi})$ when $s \geq \frac{1}{2}$. This error bound is nearly-optimal in terms of approximation; nearly due to the presence of the log-term. When $s < \frac{1}{2}$, the derived error bound behaves like $\mathcal{O}(h^{2s}|\log h|^{2\varphi})$. Since the best convergence rate we can expect with piecewise constant approximation is $\mathcal{O}(h)$ and $2s<1$, the derived error bound is suboptimal; the suboptimality is determined by the regularity properties derived in Theorem \ref{thm:regul_control}. 
\end{remark}

\begin{remark}[\DQ{improvement of the error bound \eqref{eq:error_estimates_control_final_fd}}]
\DQ{If $s \leq 1/2$, the regularity results of Theorem \ref{thm:regul_control} guarantee that $\bar{q} \in H^{2s-2\epsilon}(\Omega)$ for all $0 < \epsilon < s$. This regularity result can be improved for $s \in [1/4,1/2)$ to $\bar{q} \in H^{s + 1/2 - \epsilon}(\Omega)$ for all $ 0 < \epsilon < s + 1/2$ under the assumptions that $\Omega$ is a bounded Lipschitz domain satisfying an exterior ball condition and the data $f$ and $u_{\Omega}$ belong to $C^{1/2 - s}(\bar{\Omega})$. We also have the following regularity improvement of the state and adjoint state variables: $\bar{u}, \bar{p} \in H^{s + 1/2 - \epsilon}(\Omega)$. We leave to the reader the details of how these regularity results can lead to improved error estimates on quasi-uniform meshes. Finally, we note that, as in \cite{MR4599045}, the use of appropriate graded meshes \cite{MR3893441} can also lead to improved convergence rates.}
%
\end{remark}

\begin{corollary}[error estimates]
\label{cor:error_estimates_st_ad_fully}
Let the assumptions of Theorem \ref{thm:error_estimate_fully} hold. Then, there exists $h_{\ddagger}>0$ such that, for all $h < h_{\ddagger}$,
\begin{align}
\| \bar{u} - \bar{u}_h \|_s \lesssim h^{\gamma}|\log h|^{\varphi},
\quad
\| \bar{p} - \bar{p}_h \|_s \lesssim h^{\gamma}|\log h|^{\varphi},
\quad
\gamma = \min\{s,\tfrac{1}{2}\}.
\label{eq:error_estimates_state_final_fd}
\\
\| \bar{u} - \bar{u}_h \|_{L^2(\Omega)} \lesssim h^{2\gamma}|\log h|^{2\varphi},
\qquad
\| \bar{p} - \bar{p}_h \|_{L^2(\Omega)} \lesssim h^{2\gamma}|\log h|^{2\varphi}.
\label{eq:error_estimates_adjoint_final_fd}
\end{align}
Here, $\varphi = \nu$ if $s\neq \frac{1}{2}$, $\varphi = 1 +\nu$ if $s=\frac{1}{2}$, and $\nu \geq \frac{1}{2}$ is the constant in Theorem \ref{thm:sobolev_reg}.
\end{corollary}
\begin{proof}
The proof of the estimates in \eqref{eq:error_estimates_state_final_fd} and \eqref{eq:error_estimates_adjoint_final_fd} follows directly from the combination of the estimates in Theorems \ref{thm:error_estimates_state_aux}, \ref{thm:error_estimates_adj_aux}, and \ref{thm:error_estimate_fully}.
\end{proof}


\subsection{Error estimates for the semidiscrete scheme}

Let $\{\bar{\mathsf{q}}_h\}_{h>0} \subset \mathbb{Q}_{ad}$ be a sequence of local minima of the semidiscrete optimal control problems such that $\bar{\mathsf{q}}_h \to \bar{q}$ in $L^2(\Omega)$ as $h \rightarrow 0$; $\bar{q} \in \mathbb{Q}_{ad}$ denotes a local solution to \eqref{eq:weak_min_problem}--\eqref{eq:weak_st_eq} (see section \ref{sec:convergence_semidiscrete}). The main goal of this section is to obtain an estimate for $\|\bar{q}-\bar{\mathsf{q}}_h\|_{L^2(\Omega)}$.

The following result, which is instrumental, is analogous to Lemma \ref{lemma:aux_error_estimate}.

\begin{lemma}[auxiliary error estimate]
\label{lemma:aux_error_estimate_semi}
Let $s\in(0,1)$, $r > d/2s$, and $f,u_{\Omega}\in L^2(\Omega) \cap L^r(\Omega)$. Assume that \eqref{eq:assumption_uniform_L-infty_bound} holds and that $\bar{q}\in\mathbb{Q}_{ad}$ satisfies the second order optimality conditions \eqref{eq:second_order_2_2}.  \EO{Then, there exists $h_{\dagger} > 0$ such that}
\begin{equation}\label{eq:aux_estimate_semi}
\textgoth{C}
\|\bar{q}-\bar{\mathsf{q}}_h\|_{L^2(\Omega)}^2\leq [j'(\bar{\mathsf{q}}_h)-j'(\bar{q})](\bar{\mathsf{q}}_h-\bar{q}) \quad \forall h < h_{\dagger},
\quad
\textgoth{C} = 2^{-1} \min\{\mu,\lambda\}.
\end{equation}
Here, $\lambda$ is the control cost and $\mu$ denotes the constant appearing in \eqref{eq:second_order_2_2}.
\end{lemma}
\begin{proof}
\EO{Define} $w_h:= (\bar{\mathsf{q}}_h-\bar{q})/\|\bar{\mathsf{q}}_h-\bar{q}\|_{L^2(\Omega)}$. We assume that (up to a subsequence if necessary) $w_{h} \rightharpoonup w$ in $L^{2}(\Omega)$ as $h \rightarrow 0$. The arguments elaborated in the proof of Lemma \ref{lemma:aux_error_estimate} guarantee that $w$ satisfies \eqref{eq:sign_cond} and that $\bar{\mathfrak{d}}_{h} \rightarrow \bar{\mathfrak{d}}$ in $L^2(\Omega)$ as $h \rightarrow 0$. \EO{Here,} $\bar{\mathfrak{d}} = \lambda\bar{\mathsf{q}} - \bar{u}\bar{p}$ and $\bar{\mathfrak{d}}_h = \lambda\bar{\mathsf{q}}_{h} - \bar{u}_{h}\bar{p}_{h}$. Invoke \eqref{eq:discrete_var_ineq_variational} with $q=\bar{q}$ to obtain
\begin{equation*}
\int_\Omega \bar{\mathfrak{d}}(x)w(x) \mathrm{d}x
=
\lim_{h\to 0}\frac{1}{\|\bar{\mathsf{q}}_h-\bar{q}\|_{L^2(\Omega)}}
\int_\Omega\bar{\mathfrak{d}}_h(x) (\bar{\mathsf{q}}_h(x)-\bar{q}(x))\mathrm{d}x \leq 0.
\end{equation*}
Since $\int_{\Omega} |\bar{\mathfrak{d}}(x)w(x)| \mathrm{d}x = \int_{\Omega} \bar{\mathfrak{d}}(x)w(x) \mathrm{d}x\leq 0$, we conclude that if $\bar{\mathfrak{d}}(x)\neq 0$, then $w(x) = 0$ for a.e. $x\in\Omega$. Consequently, $w \in C_{\bar{q}}$.

The rest of the proof follows the same arguments used in the proof of Lemma \ref{lemma:aux_error_estimate}. For the sake of brevity we omit these details.
\end{proof}

We are now in a position to derive an estimate for the error \EO{that arises} when approximating a locally optimal control variable $\bar{q}$.

\begin{theorem}[error estimate]\label{thm:error_estimate_semi}
Let $s\in(0,1)$, $r > d/2s$, and $f,u_{\Omega}\in L^2(\Omega) \cap L^r(\Omega)$. Let us assume that \eqref{eq:assumption_uniform_L-infty_bound_2} holds and that $\bar{q}\in\mathbb{Q}_{ad}$ satisfies the second order optimality conditions \eqref{eq:second_order_2_2}. Then, there exists $h_{\ddagger} > 0$ such that
\begin{equation}\label{eq:control_error_estimate_semi_extended}
\|\bar{q}-\bar{\mathsf{q}}_h\|_{L^2(\Omega)}\lesssim h^{2\gamma}|\log h|^{2\varphi}
\qquad \forall h < h_{\ddagger}, \qquad \gamma = \min\{s,\tfrac{1}{2}\}.
\end{equation}
Here $\varphi = \nu$ if $s \neq \tfrac{1}{2}$, $\varphi = 1 + \nu$ if $s = \tfrac{1}{2}$, and $\nu \geq \tfrac{1}{2}$ is the constant in Theorem \ref{thm:sobolev_reg}.
\end{theorem}
\begin{proof}
\EO{Let us set $q=\bar{\mathsf{q}}_{h}$ in \eqref{eq:variational_inequality} and $q = \bar{q}$ in \eqref{eq:discrete_var_ineq_variational} to obtain $- j'(\bar{q})(\bar{\mathsf{q}}_{h} - \bar{q}) \leq 0$ and $j_{h}^{\prime}(\bar{\mathsf{q}}_{h})(\bar{q} - \bar{\mathsf{q}}_{h}) \geq 0$, respectively. Using these estimates in the auxiliary error estimate \eqref{eq:aux_estimate_semi} we obtain}
\begin{multline*}\label{eq:error_j-jh_norm}
\|\bar{q} -\bar{\mathsf{q}}_{h}\|_{L^2(\Omega)}^2  
\lesssim 
[j'(\bar{\mathsf{q}}_{h}) - j_{h}^{\prime}(\bar{\mathsf{q}}_{h})](\bar{\mathsf{q}}_{h} - \bar{q}) 
= 
(\bar{u}_{h}\bar{p}_{h} - \hat{u}\hat{p},\bar{\mathsf{q}}_{h} - \bar{q})_{L^2(\Omega)} 
\\
=  (\hat{p}(\bar{u}_{h} - \hat{u}),\bar{\mathsf{q}}_{h} - \bar{q})_{L^2(\Omega)} + (\bar{u}_{h}(\bar{p}_{h} - \hat{p}),\bar{\mathsf{q}}_{h} - \bar{q})_{L^2(\Omega)} = \mathfrak{J}_h + \mathfrak{K}_h.
\end{multline*}
Here, $\hat{u}$ and $\hat{p}$ in $\tilde{H}^{s}(\Omega)$ denote auxiliary variables that are defined as the solutions to \eqref{eq:fd_aux_hat_u} and \eqref{eq:fd_aux_hat_p}, respectively, but \EO{where} $\bar{q}_h$ \EO{is} replaced by $\bar{\mathsf{q}}_{h}$.

Let us estimate $\mathfrak{J}_h$. Since $f, u_{\Omega}\in L^{r}(\Omega)$, by applying the Theorem \ref{thm:L_infty_reg}, we deduce that $\hat{u},\hat{p}\in L^{\infty}(\Omega)$. \EO{The error estimate \eqref{eq:error_in_norm_L2} therefore allows the conclusion that}
\begin{equation}\label{eq:estimate_I_var}
\mathfrak{J}_h \leq  \|\hat{p}\|_{L^{\infty}(\Omega)}\|  \hat{u} - \bar{u}_{h} \|_{L^2(\Omega)}\| \bar{q} - \bar{\mathsf{q}}_{h}\|_{L^2(\Omega)} \lesssim h^{2\gamma}|\log h|^{2\varphi}\| \bar{q} - \bar{\mathsf{q}}_{h}\|_{L^2(\Omega)}.
\end{equation}
Here, $\gamma$ and $\varphi$ are as in the statement of the theorem. We now focus on the estimation of $\mathfrak{K}_h$. To do this, we introduce the auxiliary variable $\tilde{p}$ as in \eqref{eq:tilde_p}, \EO{but replacing} $\bar{q}_h$ with $\bar{\mathsf{q}}_h$. With this variable at hand, we invoke \eqref{eq:assumption_uniform_L-infty_bound_2} to obtain
\[
\mathfrak{K}_h 
\leq  
\|\bar{u}_{h}\|_{L^{\infty}(\Omega)}
\left( 
\| \hat{p} - \tilde{p} \|_{L^2(\Omega)} 
+
\| \tilde{p} - \bar{p}_{h} \|_{L^2(\Omega)}
\right)
\| \bar{q} - \bar{\mathsf{q}}_{h}\|_{L^2(\Omega)}.
\]
A stability estimate for the problem that $\hat{p} - \tilde{p}$ solves yields $\| \hat{p} - \tilde{p} \|_s \lesssim \| \hat{u} - \bar{u}_{h} \|_{L^2(\Omega)}$. The control of $\| \tilde{p} - \bar{p}_{h} \|_{L^2(\Omega)}$ follows from noticing that $\bar{p}_{h}$ corresponds to the finite element approximation of $\tilde{p}$ \EO{in the framework} of \S \ref{sec:fem}: $\| \tilde{p} - \bar{p}_{h} \|_{L^2(\Omega)} \lesssim h^{2\gamma}|\log h|^{2\varphi}$. A collection of these estimates allows us to conclude that $\mathfrak{K}_h \lesssim h^{2\gamma}|\log h|^{2\varphi}\|\bar{q} - \bar{\mathsf{q}}_{h}\|_{L^2(\Omega)}$.

The bound derived for $\mathfrak{K}_h$ together with that in \eqref{eq:estimate_I_var} for $\mathfrak{J}_h$ \EO{yields the final estimate} $\|\bar{q} -\bar{\mathsf{q}}_{h}\|_{L^2(\Omega)} \lesssim h^{2\gamma}|\log h|^{2\varphi}$. \EO{This concludes the proof.}
\end{proof}

We conclude this section with the following error estimates.

\begin{corollary}[error estimates]\label{cor:error_estimates_st_ad_semi}
Let the assumptions of Theorem \ref{thm:error_estimate_semi} hold. Then, there exists $h_{\ddagger}>0$ such that, for all $h < \EO{h_{\ddagger}}$,
\begin{align*}
\| \bar{u} - \bar{u}_h \|_s \lesssim h^{\gamma}|\log h|^{\varphi},
\quad
\| \bar{p} - \bar{p}_h \|_s \lesssim h^{\gamma}|\log h|^{\varphi},
\quad
\gamma = \min\{s,\tfrac{1}{2}\}.
\\
\| \bar{u} - \bar{u}_h \|_{L^2(\Omega)} \lesssim h^{2\gamma}|\log h|^{2\varphi},
\qquad
\| \bar{p} - \bar{p}_h \|_{L^2(\Omega)} \lesssim h^{2\gamma}|\log h|^{2\varphi}.
\end{align*}
Here, $\varphi = \nu$ if $s\neq \frac{1}{2}$, $\varphi = 1 +\nu$ if $s=\frac{1}{2}$, and $\nu \geq \frac{1}{2}$ is the constant in Theorem \ref{thm:sobolev_reg}.
\end{corollary}


\section{Numerical examples}\label{sec:numerical_exp}

We present three numerical experiments that illustrate the performance of the fully discrete and semidiscrete methods presented in sections \ref{sec:fully_discrete} and \ref{sec:variational_discretization}, respectively, when used to approximate a solution to the control problem \eqref{eq:weak_min_problem}--\eqref{eq:weak_st_eq}. The experiments were performed with a code implemented in MATLAB, and the schemes were solved with a semi--smooth Newton method.

The setting of the experiments is as follows: we let $s \in \{0.1, 0.2,\ldots, 0.9\}$, $\lambda = 1$, $d = 2$, and $\Omega = B(0,1)$; $B(0,1)$ denotes the unit disc. The exact optimal state and the exact optimal adjoint state are given by
\begin{equation}
\bar{u}(x) = \bar{p}(x) = ( 2^{2s}\Gamma^{2}\left(1 + s\right))^{-1}(1 - |x|^{2})^{s}_{+}, 
\qquad t_{+} = \max\{0,t\}.
\label{eq:u_and_p}
\end{equation} 

\subsection{Example 1}
In this numerical experiment we go beyond the theory presented and consider $a = 0$ and $b = 0.5$.


\begin{figure}[!ht]
\centering
\psfrag{s=01}{{\normalsize $s = 0.1$}}
\psfrag{s=02}{{\normalsize $s = 0.2$}}
\psfrag{s=03}{{\normalsize $s = 0.3$}}
\psfrag{s=04}{{\normalsize $s = 0.4$}}
\psfrag{s=05}{{\normalsize $s = 0.5$}}
\psfrag{s=06}{{\normalsize $s = 0.6$}}
\psfrag{s=07}{{\normalsize $s = 0.7$}}
\psfrag{s=08}{{\normalsize $s = 0.8$}}
\psfrag{s=09}{{\normalsize $s = 0.9$}}
\psfrag{h}{{\normalsize $h$}}
\psfrag{o(h05)}{{\normalsize $h^{0.5}$}}
\psfrag{o(h06)}{{\normalsize $h^{0.6}$}}
\psfrag{o(h07)}{{\normalsize $h^{0.7}$}}
\psfrag{o(h08)}{{\normalsize $h^{0.8}$}}
\psfrag{o(h09)}{{\normalsize $h^{0.9}$}}
\psfrag{o(h10)}{{\normalsize $h^{1.0}$}}
\begin{minipage}[c]{0.545\textwidth}\centering
{\large \hspace{0.5cm} $\| \bar{u} - \bar{u}_{h} \|_{s}$}~\\ 
\psfrag{state seminorm error}{}
\includegraphics[trim={0 0 0 0},clip,width=6.70cm,height=4.6cm,scale=0.35]{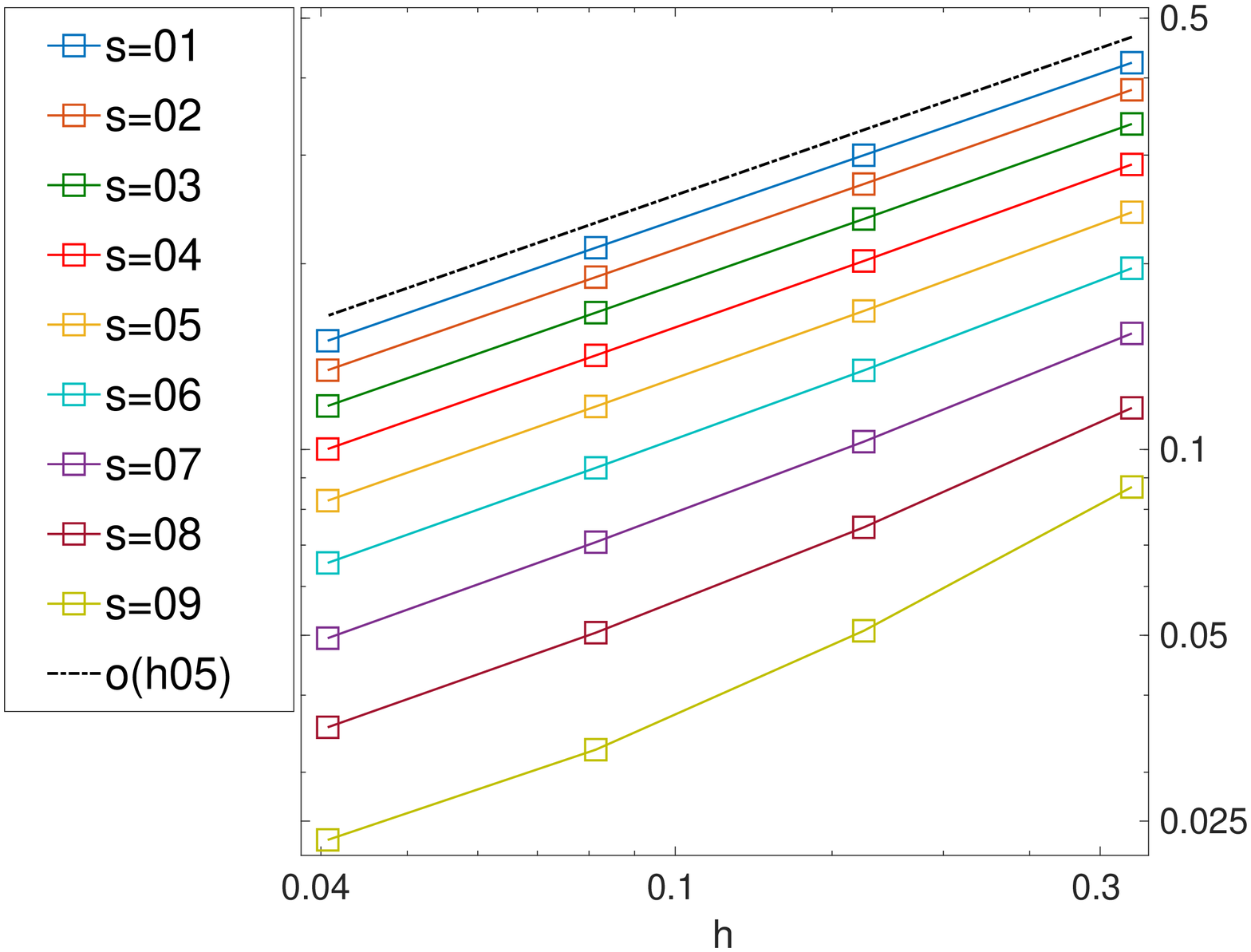} \\
\hspace{0.78cm} \tiny{(A.1)}~\\~\\
\psfrag{state seminorm error}{
}
\hspace{1.55cm}\includegraphics[trim={0 0 0 0},clip,width=5.30cm,height=4.6cm,scale=0.35]{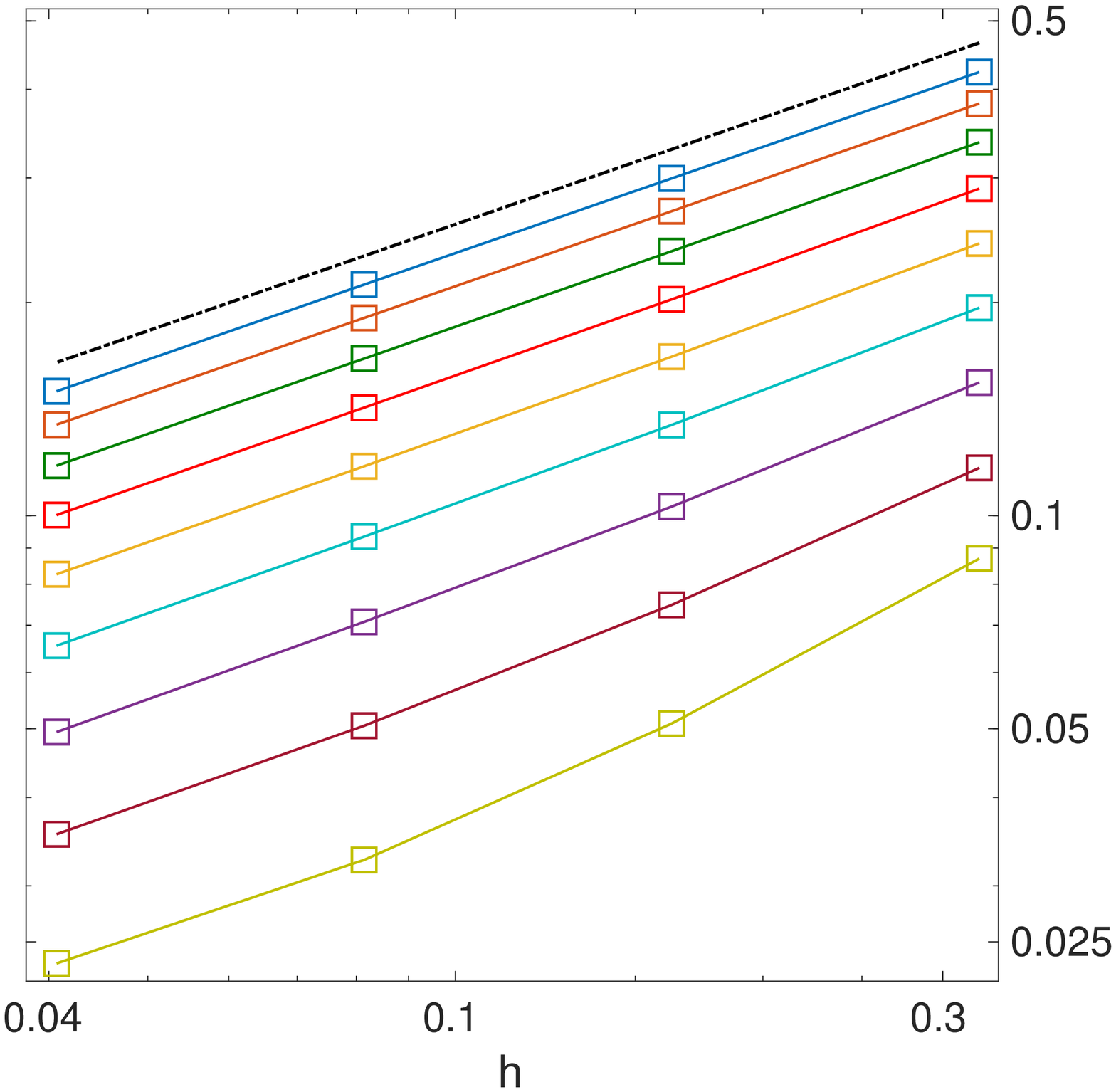} \\
\hspace{0.97cm}\tiny{(B.1)}~\\~\\
\end{minipage}
\begin{minipage}[c]{0.445\textwidth}\centering
{\large \hspace{-0.7cm} $\| \bar{p} - \bar{p}_{h} \|_{s}$}~\\ 
\psfrag{adjoint seminorm error}{}
\includegraphics[trim={0 0 0 0},clip,width=5.30cm,height=4.6cm,scale=0.35]{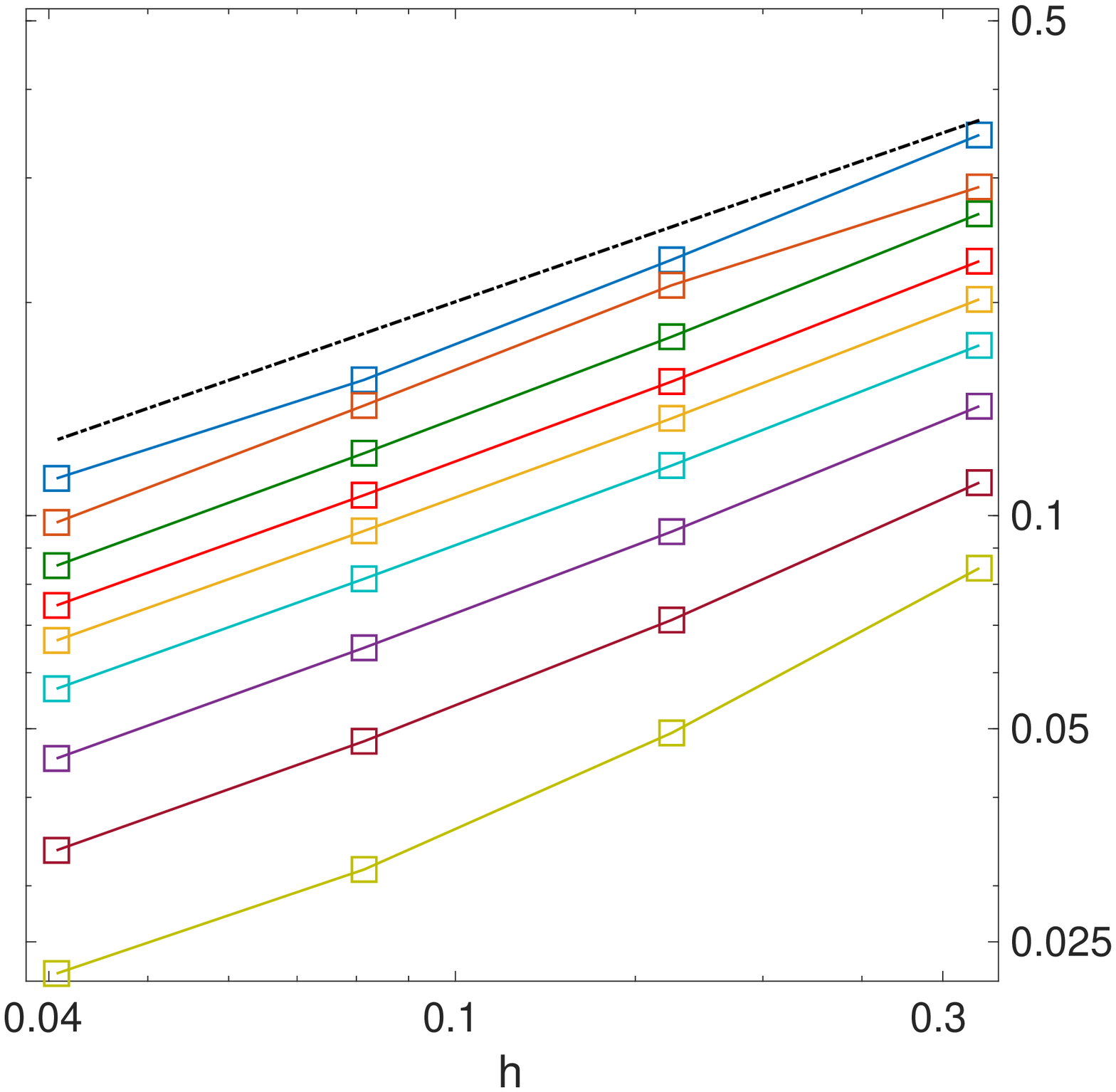}\\
\hspace{-0.55cm}\tiny{(A.2)}~\\~\\
\psfrag{adjoint seminorm error}{
}
\includegraphics[trim={0 0 0 0},clip,width=5.30cm,height=4.6cm,scale=0.35]{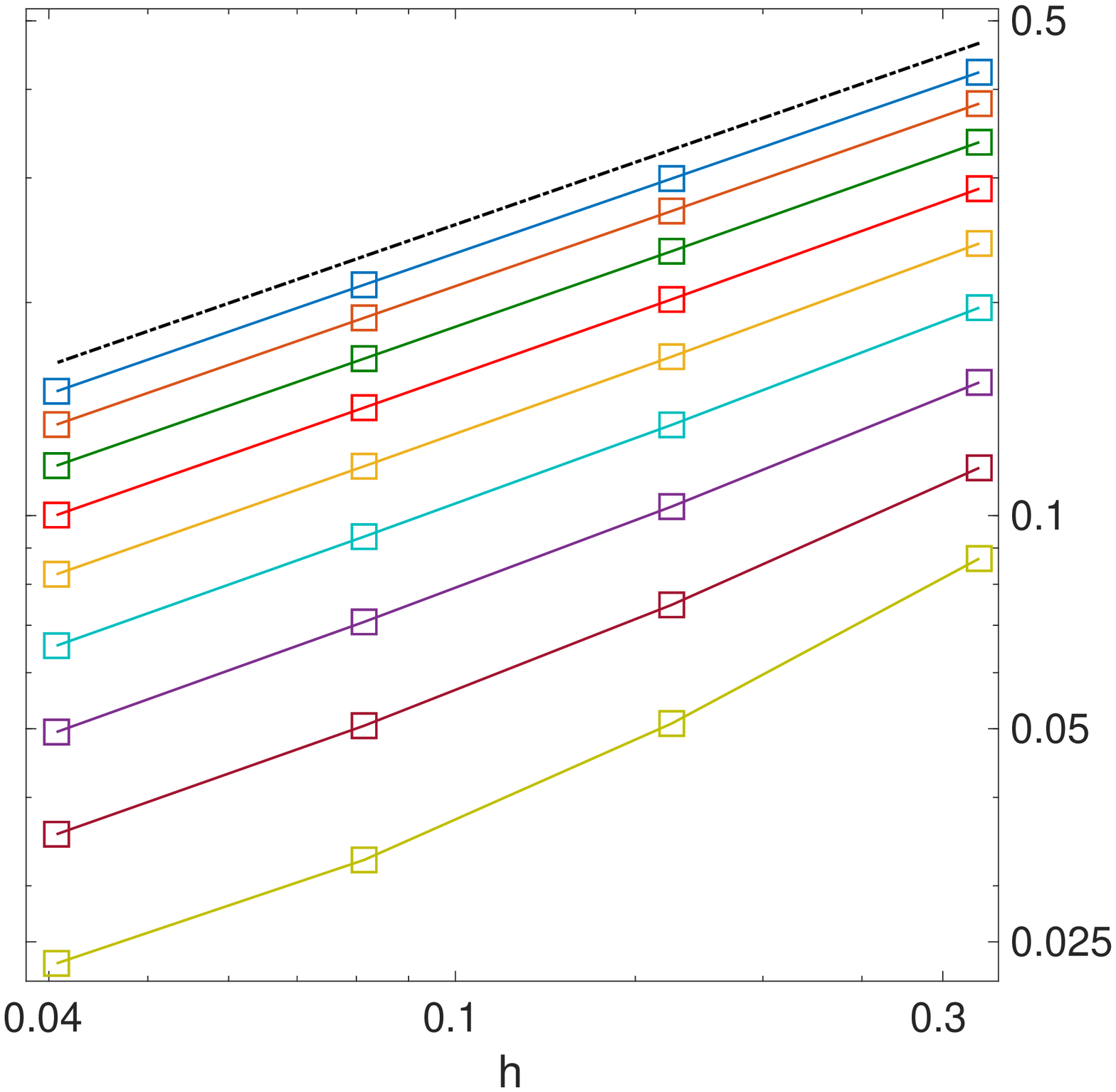}\\
\hspace{-0.55cm}\tiny{(B.2)}~\\~\\
\end{minipage}
\caption{Experimental rates of convergence for $\| \bar{u} - \bar{u}_h\|_s$ and $\| \bar{p} - \bar{p}_h \|_s$ considering the fully discrete (A.1)--(A.2) and semidiscrete schemes (B.1)--(B.2) for $s \in \{0.1,0.2,...,0.9\}$.}
\label{fig:ex-1.1}
\end{figure}


\begin{figure}[!ht]
\centering
\psfrag{s=01}{{\normalsize $s = 0.1$}}
\psfrag{s=02}{{\normalsize $s = 0.2$}}
\psfrag{s=03}{{\normalsize $s = 0.3$}}
\psfrag{s=04}{{\normalsize $s = 0.4$}}
\psfrag{s=05}{{\normalsize $s = 0.5$}}
\psfrag{s=06}{{\normalsize $s = 0.6$}}
\psfrag{s=07}{{\normalsize $s = 0.7$}}
\psfrag{s=08}{{\normalsize $s = 0.8$}}
\psfrag{s=09}{{\normalsize $s = 0.9$}}
\psfrag{h}{{\normalsize $h$}}
\psfrag{o(h05)}{{\normalsize $h^{0.5}$}}
\psfrag{o(h06)}{{\normalsize $h^{0.6}$}}
\psfrag{o(h07)}{{\normalsize $h^{0.7}$}}
\psfrag{o(h08)}{{\normalsize $h^{0.8}$}}
\psfrag{o(h09)}{{\normalsize $h^{0.9}$}}
\psfrag{o(h10)}{{\normalsize $h^{1.0}$}}
\begin{minipage}[c]{0.393\textwidth}\centering
{\large \hspace{0.7cm} $\| \bar{u} - \bar{u}_{h} \|_{L^{2}(\Omega)}$}~\\ 
\psfrag{state norm error}{}
\includegraphics[trim={0 0 0 0},clip,width=5.0cm,height=3.9cm,scale=0.35]{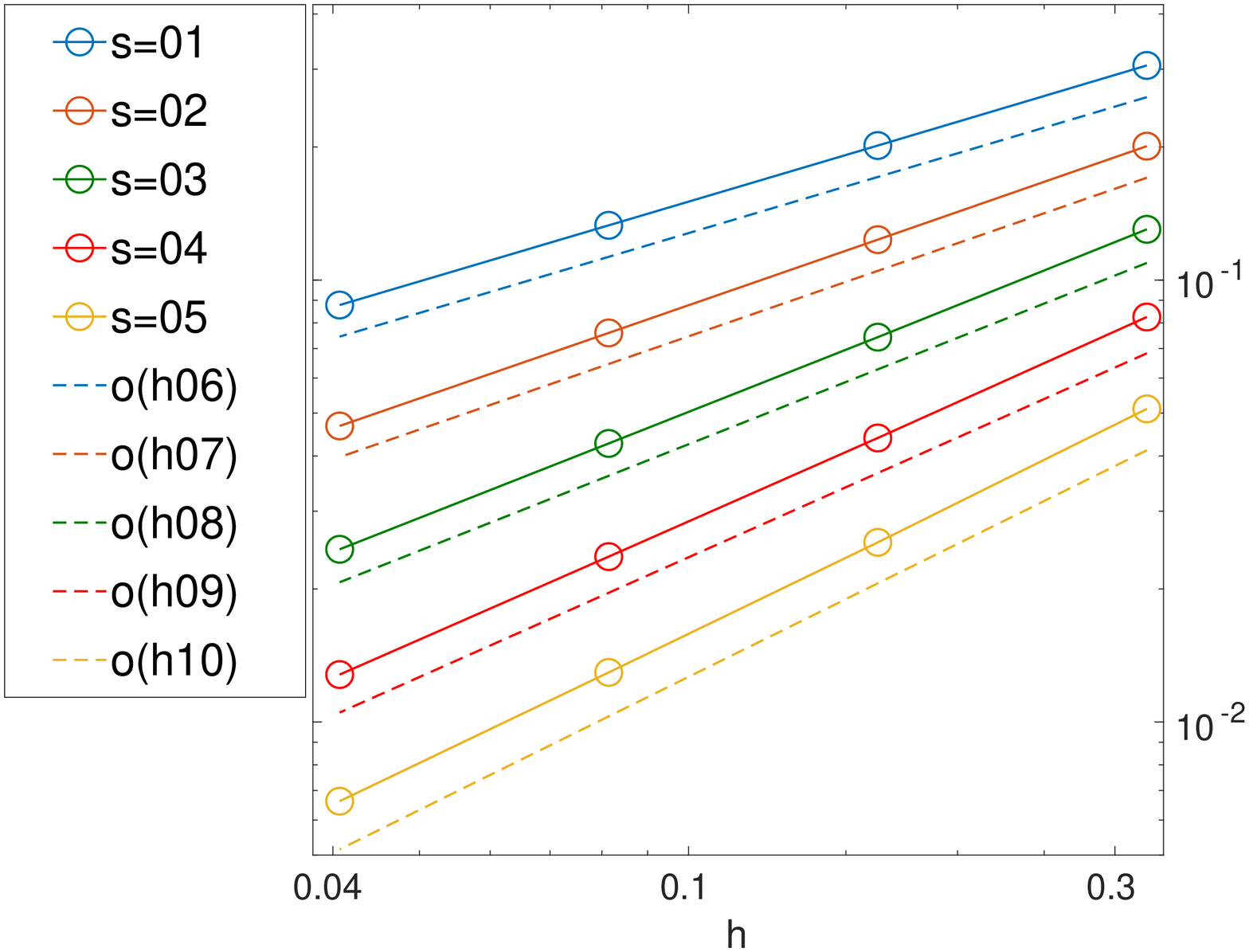} \\
\hspace{0.85cm}\tiny{(C.1)}~\\~\\ 
\psfrag{state norm error}{}
\hspace{1.15cm}\includegraphics[trim={0 0 0 0},clip,width=3.9cm,height=3.9cm,scale=0.35]{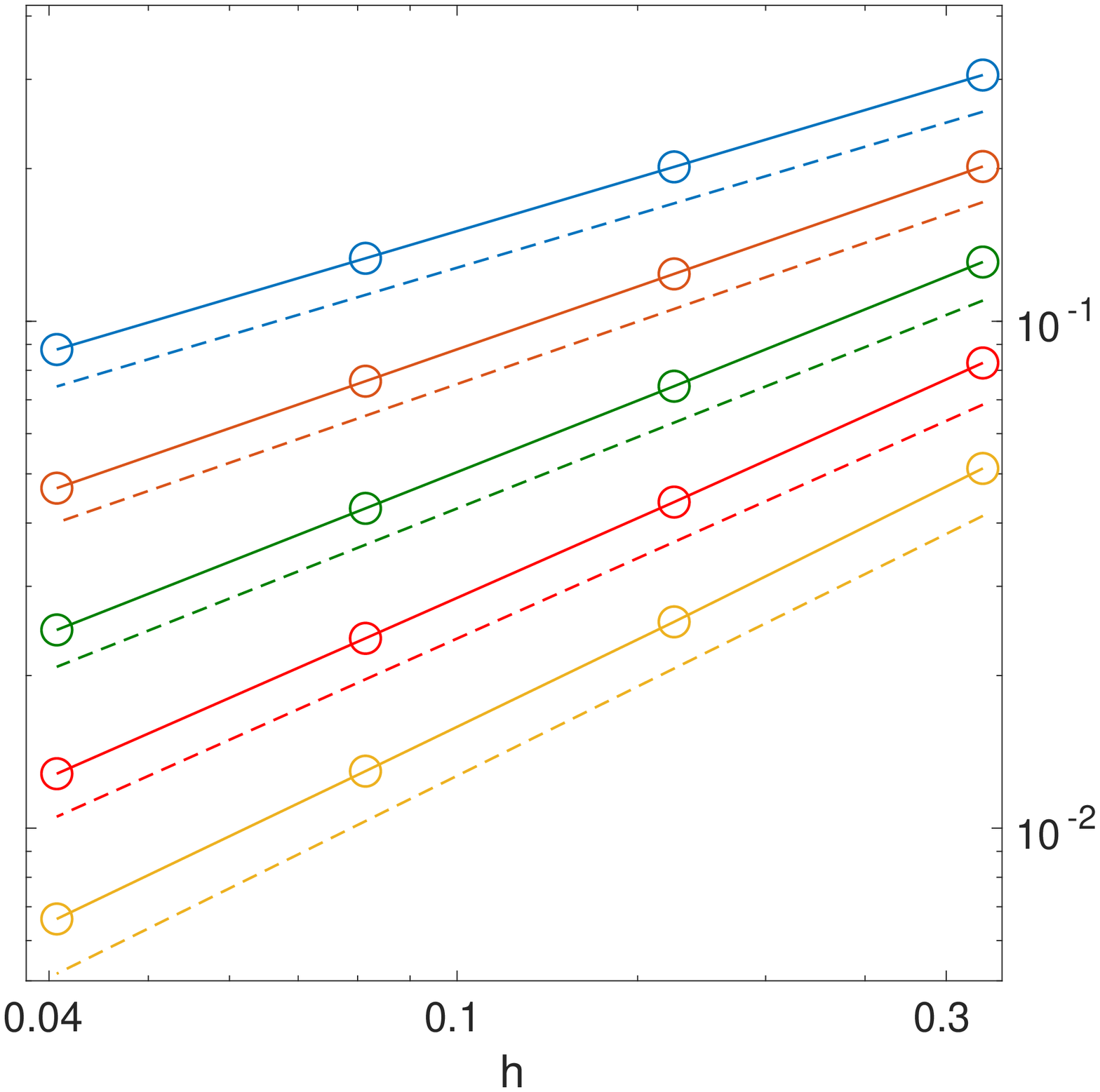} \\
\hspace{0.85cm}\tiny{(D.1)}~\\~\\
\end{minipage}
\begin{minipage}[c]{0.303\textwidth}\centering
{\large \hspace{-0.3cm} $\| \bar{p} - \bar{p}_{h} \|_{L^{2}(\Omega)}$}~\\ 
\psfrag{adjoint norm error}{}
\includegraphics[trim={0 0 0 0},clip,width=3.9cm,height=3.9cm,scale=0.35]{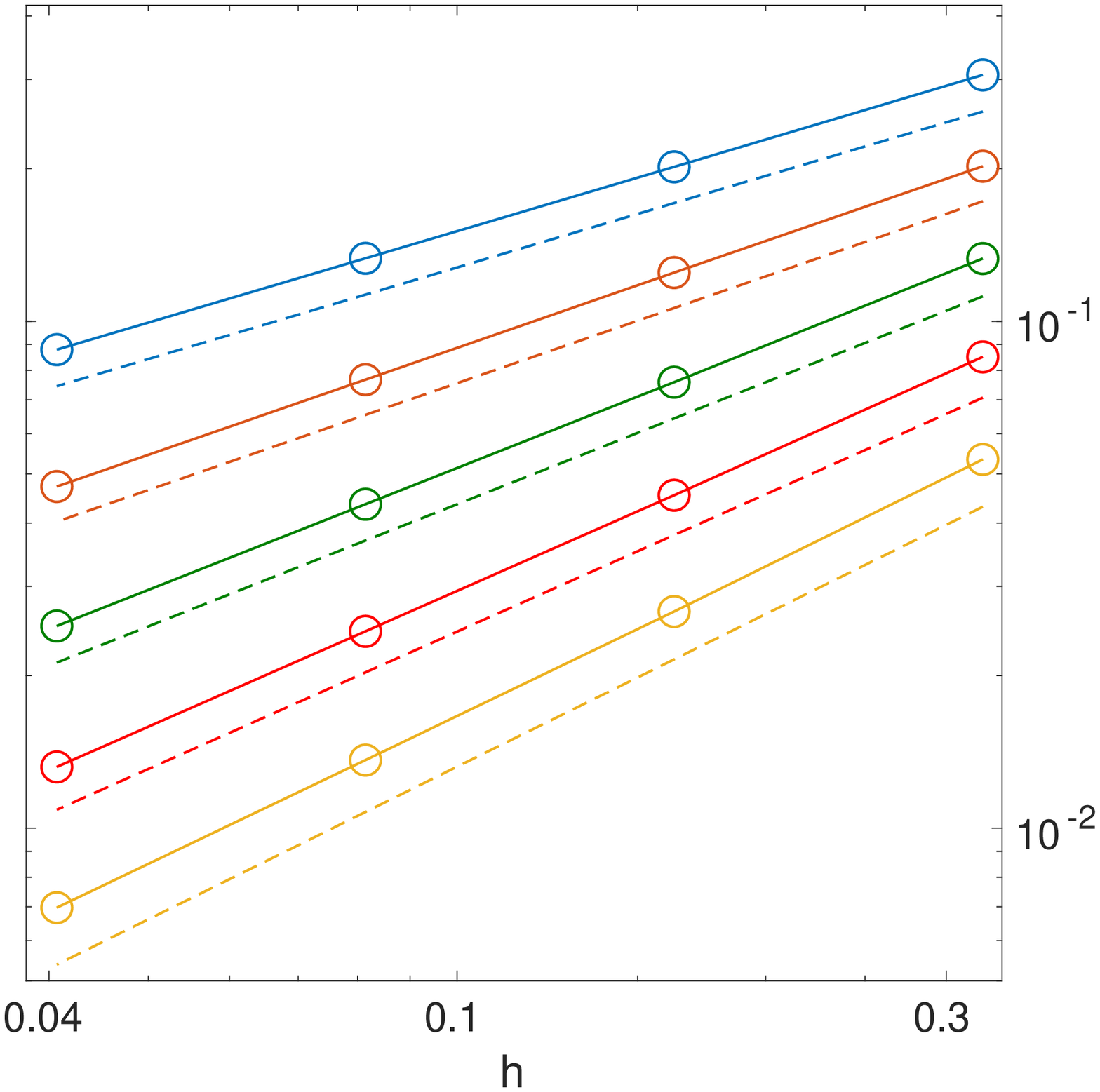} \\
\hspace{-0.3cm}\tiny{(C.2)}~\\~\\ 
\psfrag{adjoint norm error}{}
\includegraphics[trim={0 0 0 0},clip,width=3.9cm,height=3.9cm,scale=0.35]{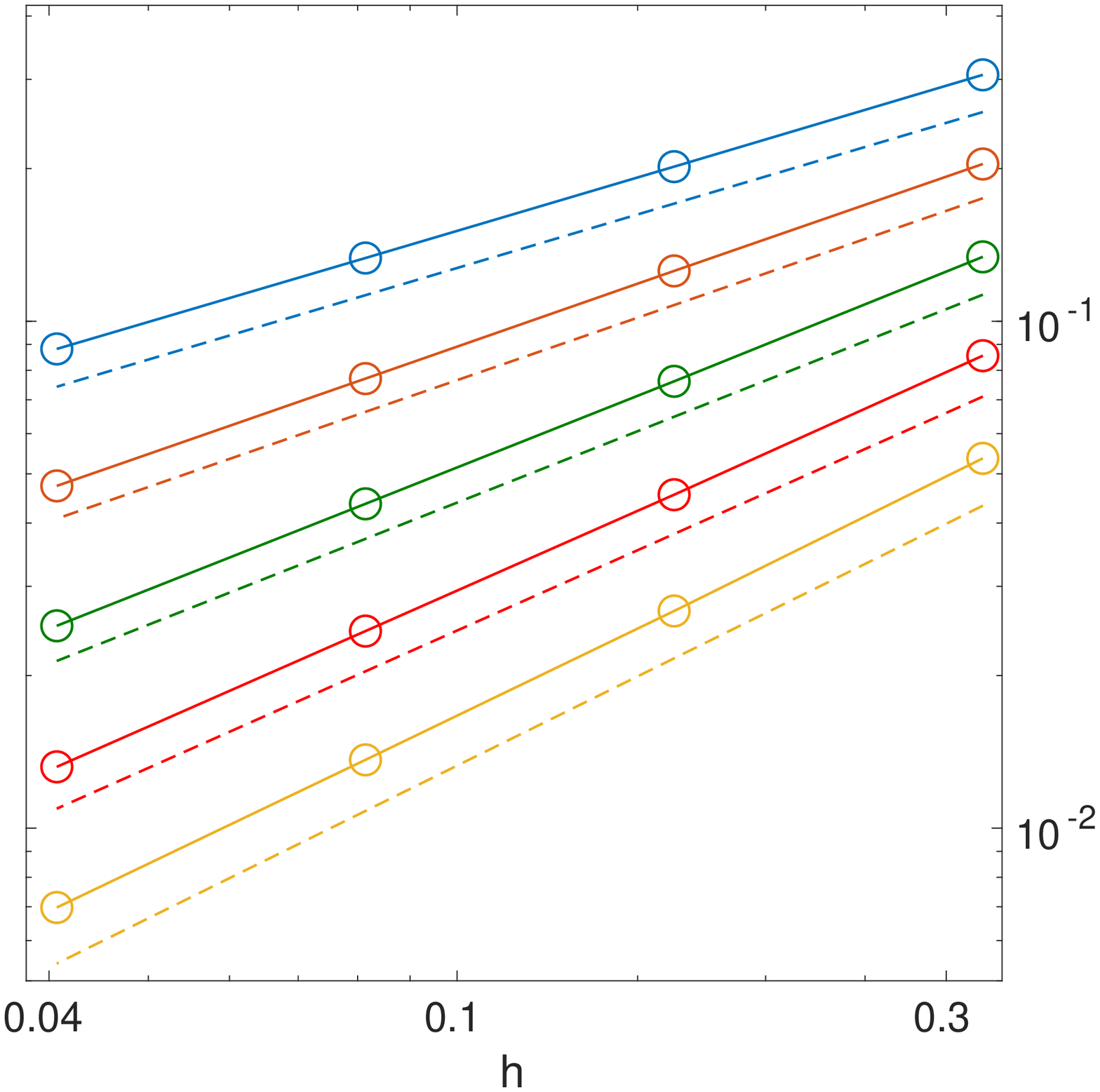} \\
\hspace{-0.3cm}\tiny{(D.2)}~\\~\\
\end{minipage}
\begin{minipage}[c]{0.303\textwidth}\centering
{\large \hspace{-0.3cm} $\| \bar{q} - \bar{q}_{h} \|_{L^{2}(\Omega)}$}~\\ 
\psfrag{control norm error}{
}
\includegraphics[trim={0 0 0 0},clip,width=3.9cm,height=3.9cm,scale=0.35]{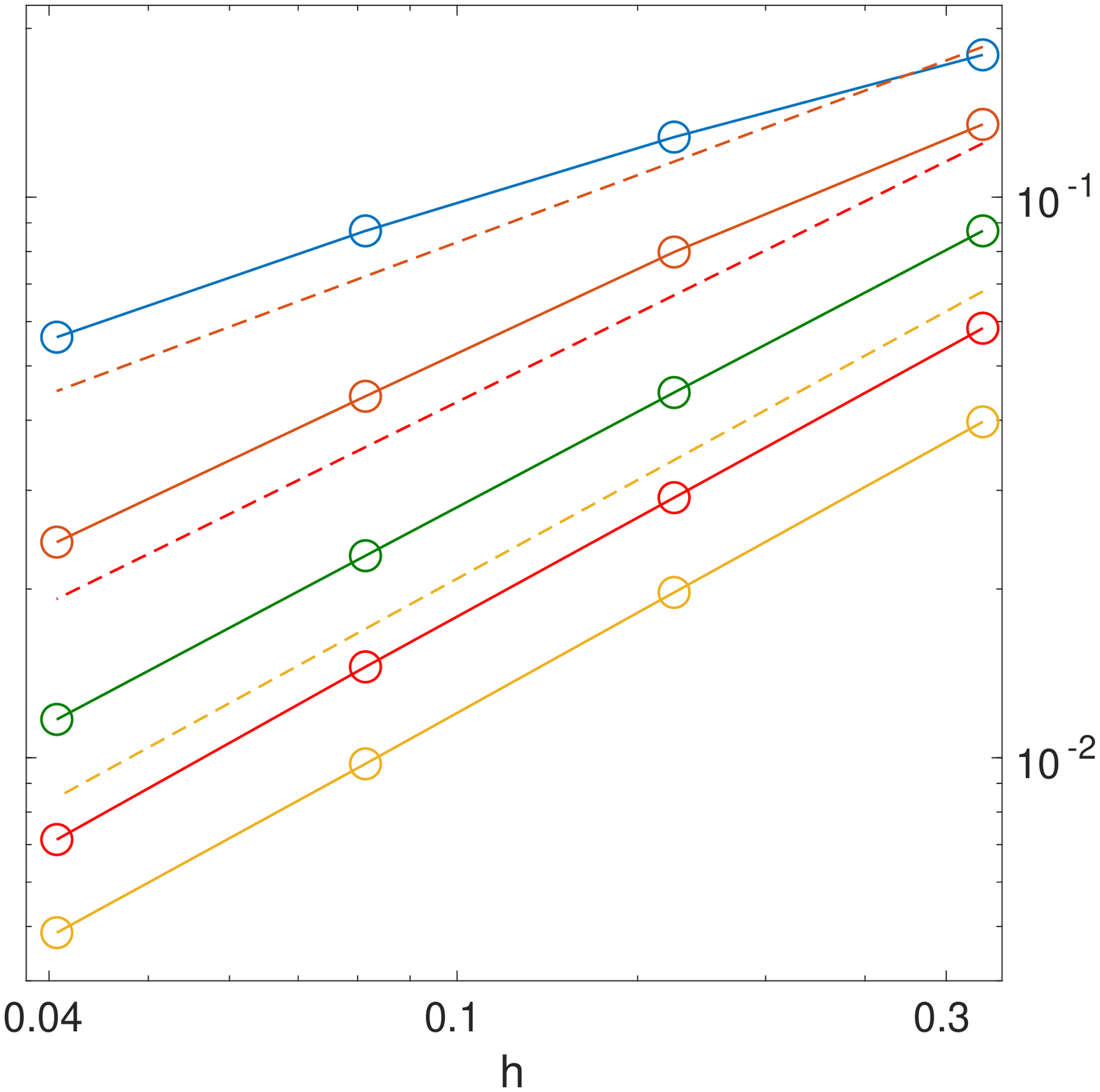}\\
\hspace{-0.3cm}\tiny{(C.3)}~\\~\\ 
\psfrag{control norm error}{
}
\includegraphics[trim={0 0 0 0},clip,width=3.9cm,height=3.9cm,scale=0.35]{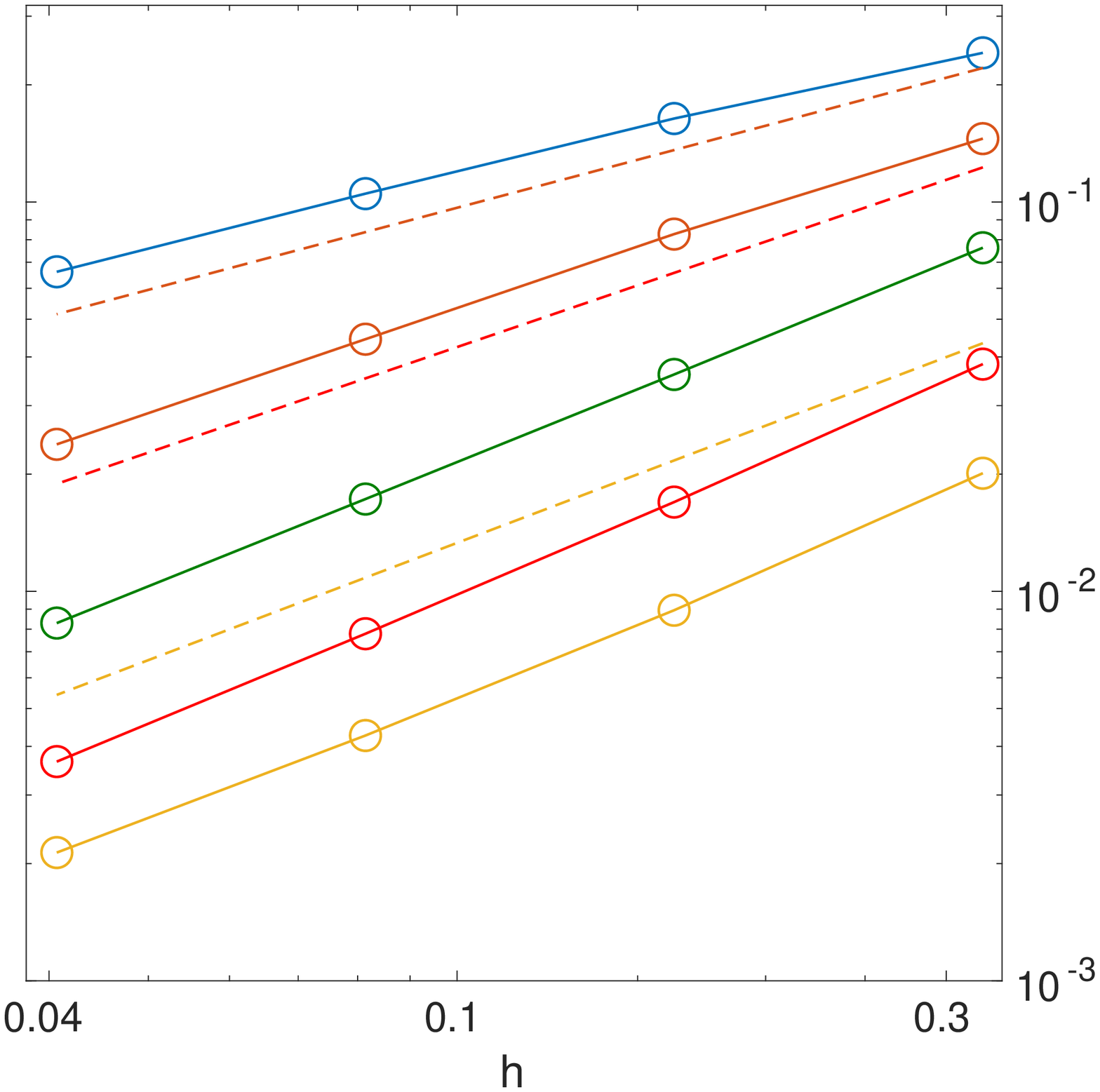}\\
\hspace{-0.3cm}\tiny{(D.3)}~\\~\\
\end{minipage}
\caption{Experimental rates of convergence for $\| \bar{u} - \bar{u}_h \|_{L^2(\Omega)}$, $\| \bar{p} - \bar{p}_h \|_{L^2(\Omega)}$, and $\| \bar{q} - \bar{q}_h \|_{L^2(\Omega)}$ considering the fully discrete (C.1)--(C.3) and semidiscrete schemes (D.1)--(D.3) for $s \in \{0.1,0.2,...,0.5\}$.}
\label{fig:ex-1.2}
\end{figure}


\begin{figure}[!ht]
\centering
\psfrag{s=01}{{\normalsize $s = 0.1$}}
\psfrag{s=02}{{\normalsize $s = 0.2$}}
\psfrag{s=03}{{\normalsize $s = 0.3$}}
\psfrag{s=04}{{\normalsize $s = 0.4$}}
\psfrag{s=05}{{\normalsize $s = 0.5$}}
\psfrag{s=06}{{\normalsize $s = 0.6$}}
\psfrag{s=07}{{\normalsize $s = 0.7$}}
\psfrag{s=08}{{\normalsize $s = 0.8$}}
\psfrag{s=09}{{\normalsize $s = 0.9$}}
\psfrag{h}{{\normalsize $h$}}
\psfrag{o(h10)}{{\normalsize $h^{1.0}$}}
\psfrag{o(h11)}{{\normalsize $h^{1.1}$}}
\psfrag{o(h12)}{{\normalsize $h^{1.2}$}}
\psfrag{o(h13)}{{\normalsize $h^{1.3}$}}
\psfrag{o(h14)}{{\normalsize $h^{1.4}$}}
\begin{minipage}[c]{0.393\textwidth}\centering
{\large \hspace{0.7cm} $\| \bar{u} - \bar{u}_{h} \|_{L^{2}(\Omega)}$}~\\ 
\psfrag{state norm error}{}
\includegraphics[trim={0 0 0 0},clip,width=5.0cm,height=3.9cm,scale=0.35]{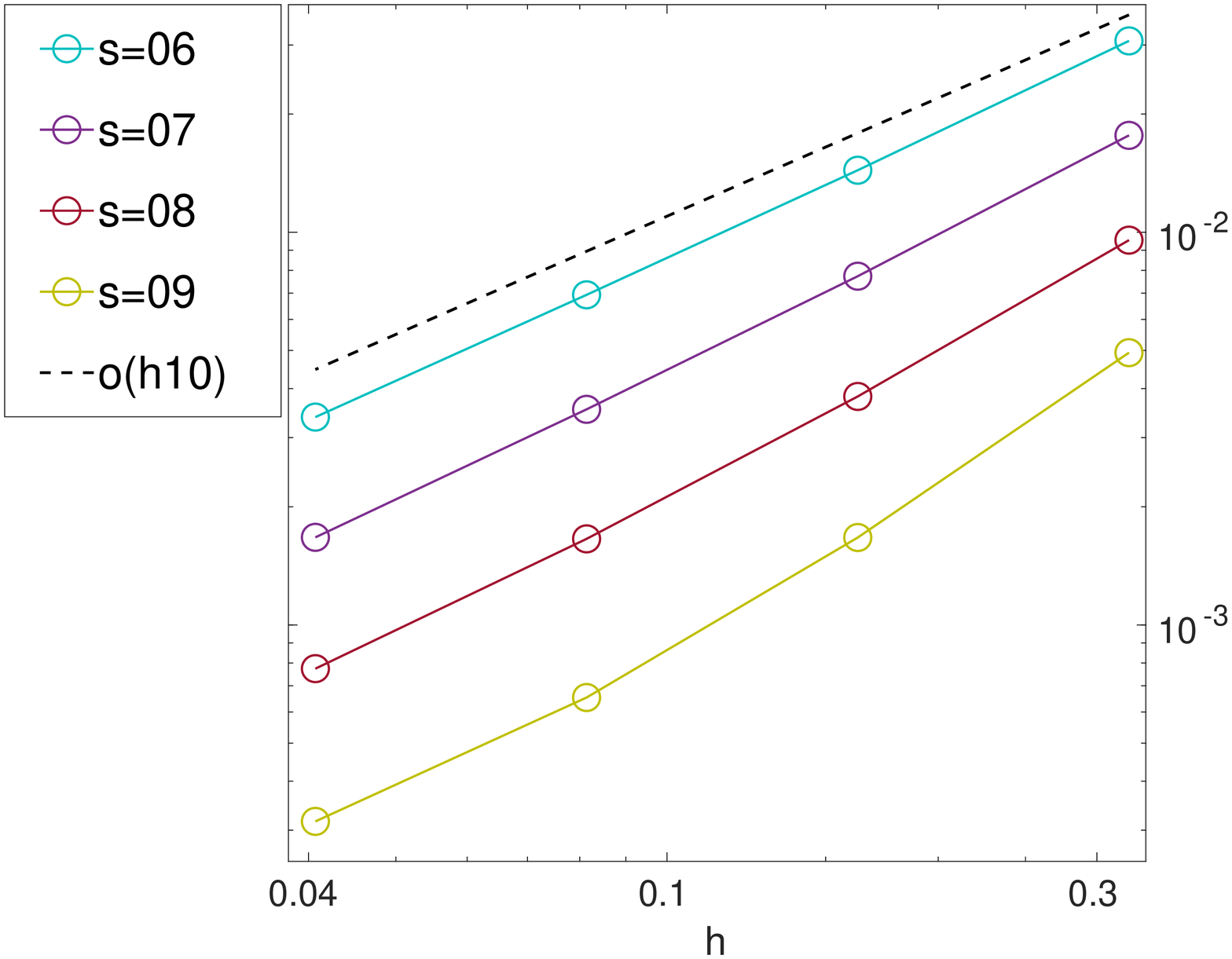} \\
\hspace{0.75cm}\tiny{(E.1)}~\\~\\ 
\psfrag{state norm error}{
}
\hspace{1.1cm}\includegraphics[trim={0 0 0 0},clip,width=3.95cm,height=3.9cm,scale=0.35]{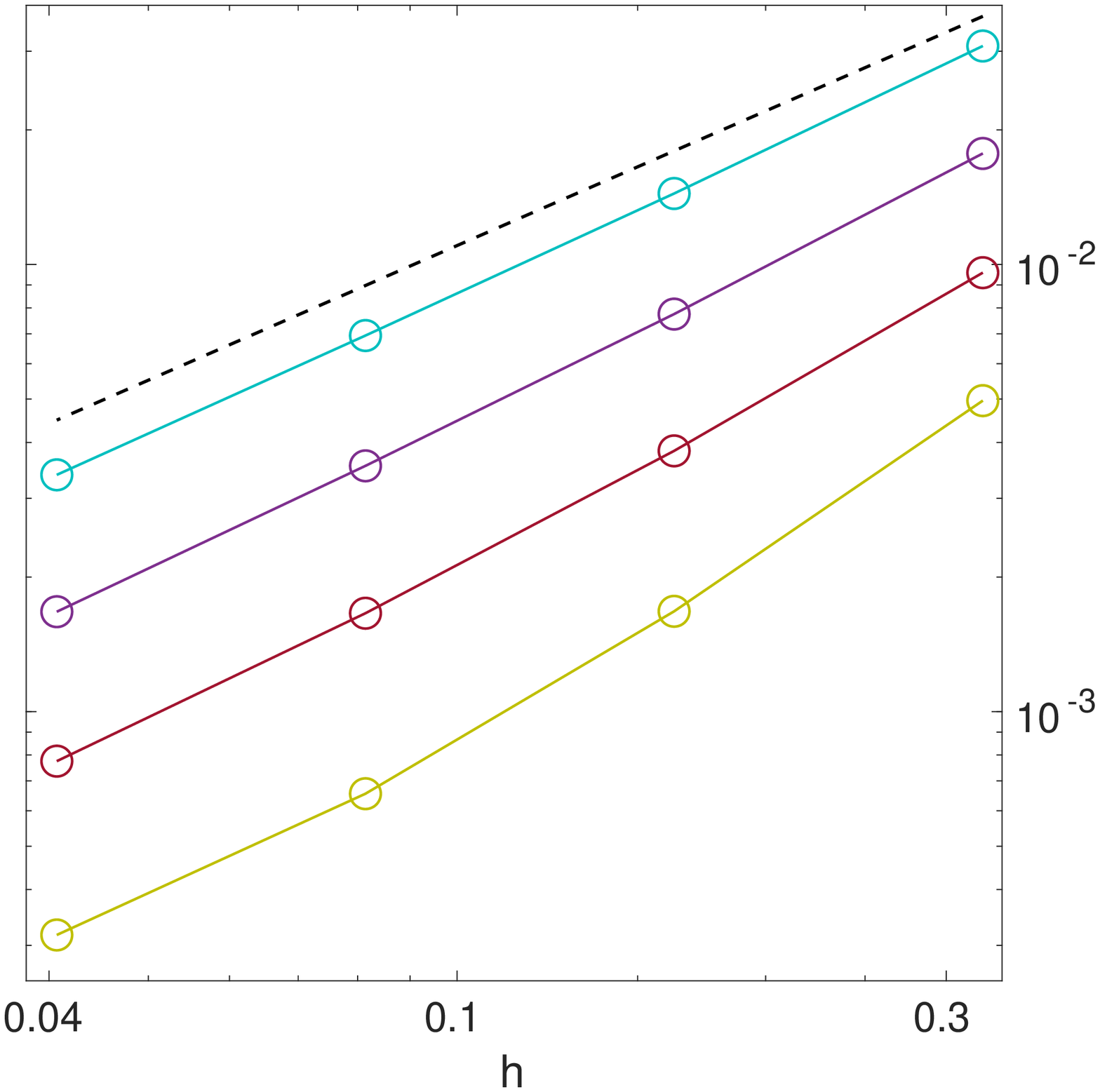} \\
\hspace{0.75cm}\tiny{(F.1)}~\\~\\
\end{minipage}
\begin{minipage}[c]{0.303\textwidth}\centering
{\large \hspace{-0.3cm} $\| \bar{p} - \bar{p}_{h} \|_{L^{2}(\Omega)}$}~\\ 
\psfrag{adjoint norm error}{}
\includegraphics[trim={0 0 0 0},clip,width=3.95cm,height=3.9cm,scale=0.35]{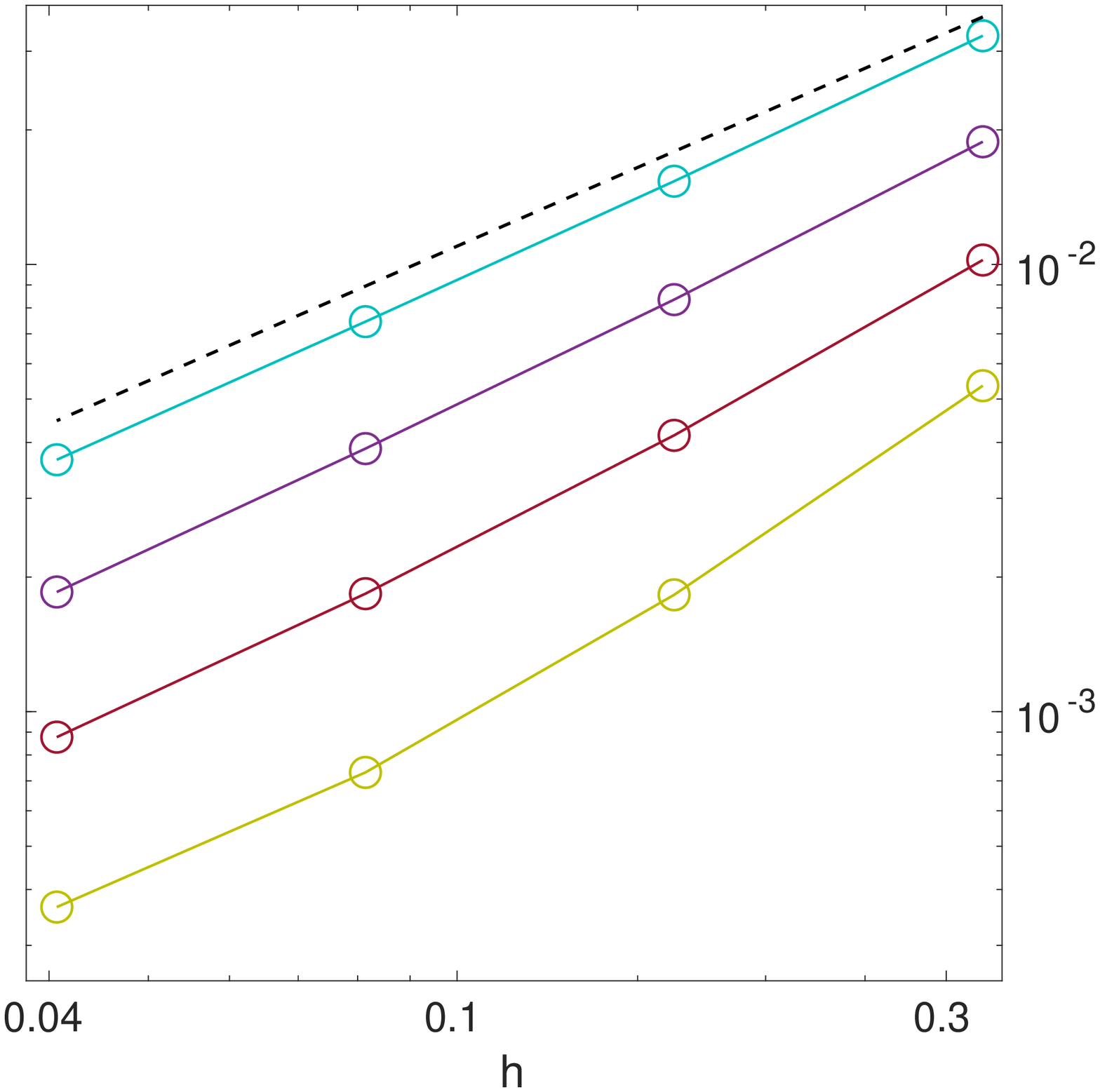} \\
\hspace{-0.25cm}\tiny{(E.2)}~\\~\\ 
\psfrag{adjoint norm error}{
}
\includegraphics[trim={0 0 0 0},clip,width=3.95cm,height=3.9cm,scale=0.35]{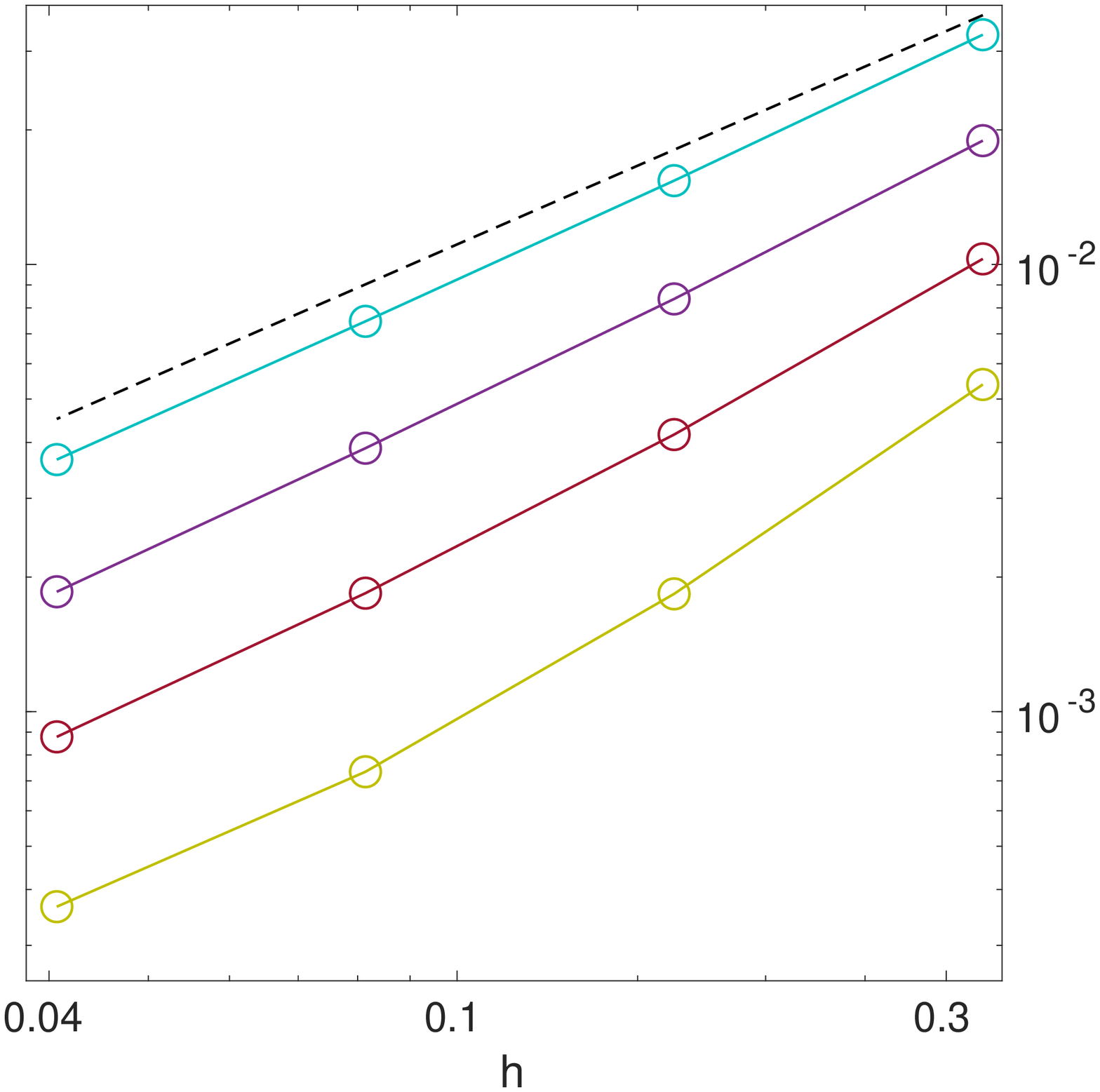} \\
\hspace{-0.25cm}\tiny{(F.2)}~\\~\\
\end{minipage}
\begin{minipage}[c]{0.303\textwidth}\centering
{\large \hspace{-0.3cm} $\| \bar{q} - \bar{q}_{h} \|_{L^{2}(\Omega)}$}~\\
\psfrag{control norm error}{}
\includegraphics[trim={0 0 0 0},clip,width=3.95cm,height=3.9cm,scale=0.35]{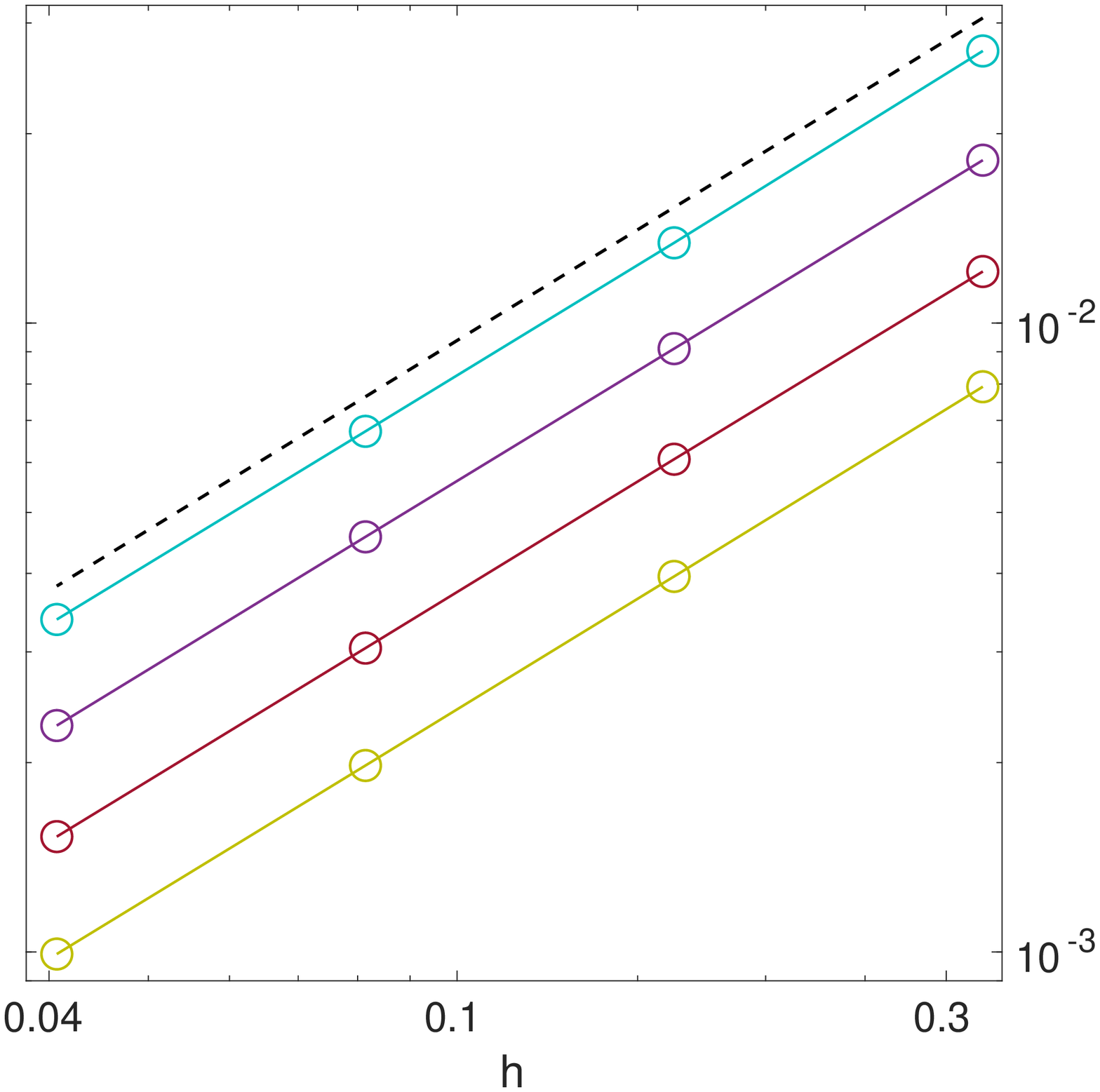}\\
\hspace{-0.25cm}\tiny{(E.3)}~\\~\\ 
\psfrag{control norm error}{
}
\includegraphics[trim={0 0 0 0},clip,width=3.95cm,height=3.9cm,scale=0.35]{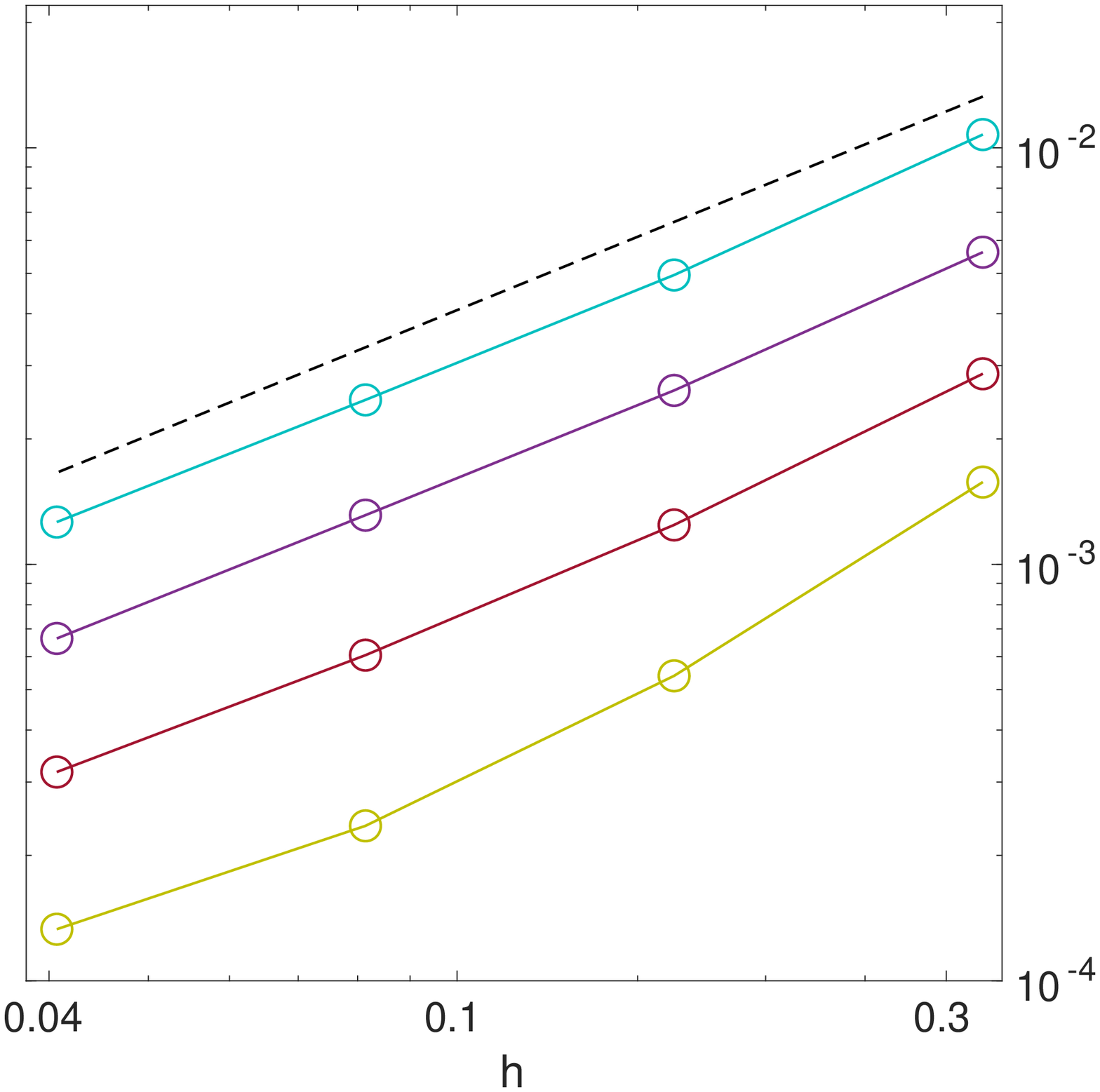}\\
\hspace{-0.25cm}\tiny{(F.3)}~\\~\\
\end{minipage}
\caption{Experimental rates of convergence for $\| \bar{u} - \bar{u}_h \|_{L^2(\Omega)}$, $\| \bar{p} - \bar{p}_h \|_{L^2(\Omega)}$, and $\| \bar{q} - \bar{q}_h \|_{L^2(\Omega)}$ considering the fully (E.1)--(E.3) and semidiscrete scheme (F.1)--(F.3) for $s \in \{0.6,0.7,0.8,0.9\}$.}
\label{fig:ex-1.3}
\end{figure}

Figures \ref{fig:ex-1.1}--\ref{fig:ex-1.3} show the results obtained for both the fully discrete and the semidiscrete schemes. In Figure \ref{fig:ex-1.1}, we show for $s \in \{0.1, 0.2,\ldots, 0.9\}$ the experimental convergence rates for $\| \bar{u} - \bar{u}_h \|_s$ and $\| \bar{p} - \bar{p}_h \|_s$. We observe that the rates predicted in the Corollaries \ref{cor:error_estimates_st_ad_fully} and \ref{cor:error_estimates_st_ad_semi} are achieved when $s \geq 0.5$. However, when $s < 0.5$ the experimental convergence rates exceed those derived in these corollaries.
We present experimental convergence rates for $\| \bar{u} - \bar{u}_h \|_{L^2(\Omega)}$, $\| \bar{p} - \bar{p}_h \|_{L^2(\Omega)}$, and $\| \bar{q} - \bar{q}_h \|_{L^2(\Omega)}$ for $s \in \{0.1,\ldots,0.5\}$  and $s \in \{0.6,\ldots,0.9\}$ in Figures \ref{fig:ex-1.2} and \ref{fig:ex-1.3}, respectively. It can be observed that the experimental convergence rates for all involved approximation errors are consistent with the error bounds obtained in section \ref{sec:error_estimates} when $s\geq 0.5$. However, when $s < 0.5$, the reported experimental convergence rates exceed those predicted in our manuscript.  These cases are discussed in Remarks \ref{rem:higher_rates} and \ref{rem:higher_rates2}.

\begin{remark}[convergence rates: state and adjoint variables]\label{rem:higher_rates}
Figures \ref{fig:ex-1.1} and \ref{fig:ex-1.2} show that the experimental convergence rates for the approximation errors associated to the state and adjoint variables, when $s < 0.5$, exceed the rates predicted in Corollaries \ref{cor:error_estimates_st_ad_fully} and \ref{cor:error_estimates_st_ad_semi} but are in agreement with respect to the \emph{maximal regularity}
\begin{equation}
 H^{s + \frac{1}{2} -\epsilon}(\Omega), \qquad 0 < \epsilon < s + \tfrac{1}{2}.
\label{eq:maximal_regularity}
\end{equation}
The error bounds that we derive in Corollaries \ref{cor:error_estimates_st_ad_fully} and \ref{cor:error_estimates_st_ad_semi} are based on the regularity estimates of Theorem \ref{thm:regul_control}, which in turn are inspired by the results in Theorem \ref{thm:sobolev_reg} (\cite[Theorem 2.1]{MR4283703}). If $s \in (0,0.5]$, $\mathfrak{u} \in H^{2s - 2\epsilon}(\Omega)$ for every $\epsilon \in (0,s)$, which is weaker than \eqref{eq:maximal_regularity}. As explained in \cite[page 1921]{MR4283703}, one expects the solutions to be smoother than just $H^{2s}(\Omega)$ if the forcing term $\mathfrak{f} \in H^r(\Omega)$, for some $r>0$; however, such a result of higher regularity cannot be derived from \cite[Theorem 2.1]{MR4283703}. Nevertheless, it is important to emphasize that the estimates in \cite[Theorem 2.1]{MR4283703} hold under the assumption that $\partial \Omega$ is \emph{merely} Lipschitz. Finally, we note that the functions $\bar{u}$ and $\bar{p}$ defined in \eqref{eq:u_and_p} satisfy \eqref{eq:maximal_regularity}.
\end{remark}

\begin{remark}[convergence rates: control variable]\label{rem:higher_rates2}
Figure \ref{fig:ex-1.2} (subfigures (C.3) and (D.3)) shows that the experimental convergence rates obtained for the control variable are higher than those predicted by \eqref{eq:error_estimates_control_final_fd} and \eqref{eq:control_error_estimate_semi_extended} when $s<0.5$. To explain this, we further investigate the regularity properties of $\bar{q}$ in our particular setting. Note that
$
\bar{u}(x)\bar{p}(x) = \mathfrak{C}(1 - |x|^{2})^{\sigma}_{+}, 
$
where $\sigma = 2s$ and $\mathfrak{C}^{-1} = (2^{2\sigma}\Gamma^{4}(1 + s))$. $\bar{u}\bar{p}$ can thus be regarded as the solution to \eqref{def:state_eq} with $s$ replaced by $\sigma$, $q \equiv 0$, and $f \equiv \Gamma^{-2}(1 + s)$. Consequently, the maximal regularity property \eqref{eq:maximal_regularity} in this case reads $\bar{u}\bar{p} \in H^{\iota}(\Omega)$, where $\iota = \sigma + 0.5 - \epsilon$ and $\epsilon \in (0,\sigma + 0.5)$. In view of the projection formula \eqref{eq:projection_control}, we apply \cite[Theorem 1]{MR1173747} to obtain $\bar{q} \in H^{\iota}(\Omega)$; observe that $\iota < 1.5$. As a result, in terms of regularity and approximation degree, we would expect, for the fully discrete scheme, $\| \bar{q} - \bar{q}_h \|_{L^2(\Omega)} \lesssim h^{\omega}$, where $\omega = \min\{2s+0.5-\epsilon, 1\}$. This is the behaviour observed in Subfigure (C.3). A similar conclusion holds for the semidiscrete scheme. Since $\bar{\mathsf{q}}_h$ is implicitly discretized with piecewise linear functions, we would expect $\| \bar{q} - \bar{\mathsf{q}}_h \|_{L^2(\Omega)} \lesssim h^{\varpi}$, where $\varpi = \min\{2s+0.5-\epsilon, 2\}$. However, we observe $\mathcal{O}(h^{\omega})$. An important observation in favor of this is the fact that it has been experimentally observed that $\| \bar{u} - \bar{u}_h \|_{L^2(\Omega)}$ and $\| \bar{p} - \bar{p}_h \|_{L^2(\Omega)}$ do not exceed $\mathcal{O}(h)$; see \cite[\S 6.1]{MR4283703}.
\end{remark}

\subsection{Example 2} 
We consider $a = 0.001 \|\bar{u}\bar{p}\|_{L^{\infty}(\Omega)}$ and $b = 1.5$.


\begin{figure}[!ht]
\centering
\psfrag{s=01}{{\normalsize $s = 0.1$}}
\psfrag{s=02}{{\normalsize $s = 0.2$}}
\psfrag{s=03}{{\normalsize $s = 0.3$}}
\psfrag{s=04}{{\normalsize $s = 0.4$}}
\psfrag{s=05}{{\normalsize $s = 0.5$}}
\psfrag{s=06}{{\normalsize $s = 0.6$}}
\psfrag{s=07}{{\normalsize $s = 0.7$}}
\psfrag{s=08}{{\normalsize $s = 0.8$}}
\psfrag{s=09}{{\normalsize $s = 0.9$}}
\psfrag{h}{{\normalsize $h$}}
\psfrag{o(h05)}{{\normalsize $h^{0.5}$}}
\psfrag{o(h06)}{{\normalsize $h^{0.6}$}}
\psfrag{o(h07)}{{\normalsize $h^{0.7}$}}
\psfrag{o(h08)}{{\normalsize $h^{0.8}$}}
\psfrag{o(h09)}{{\normalsize $h^{0.9}$}}
\psfrag{o(h10)}{{\normalsize $h^{1.0}$}}
{\large \hspace{0.7cm} $\|\bar{q} - \bar{q}_{h}\|_{L^{2}(\Omega)}$ for $a = 0.001\|\bar{u}\bar{p}\|_{L^{\infty}(\Omega)}$}
\begin{minipage}[c]{0.545\textwidth}\centering
\psfrag{control norm error}{}
\includegraphics[trim={0 0 0 0},clip,width=6.30cm,height=4.6cm,scale=0.4]{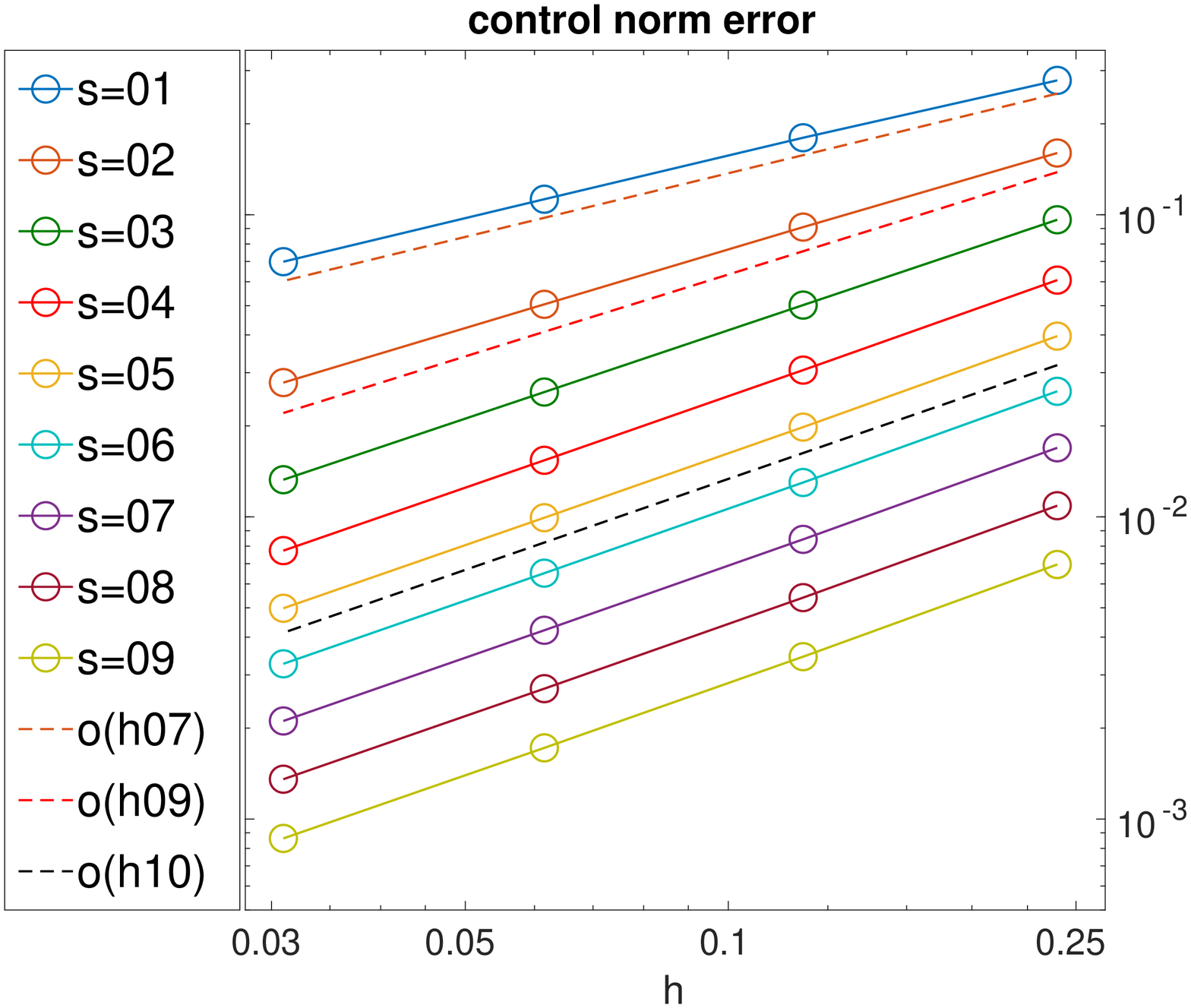} \\
\hspace{0.75cm} \tiny{(G.1)}
\end{minipage}
\begin{minipage}[c]{0.445\textwidth}\centering
\psfrag{control norm error}{}
\includegraphics[trim={0 0 0 0},clip,width=5.10cm,height=4.6cm,scale=0.4]{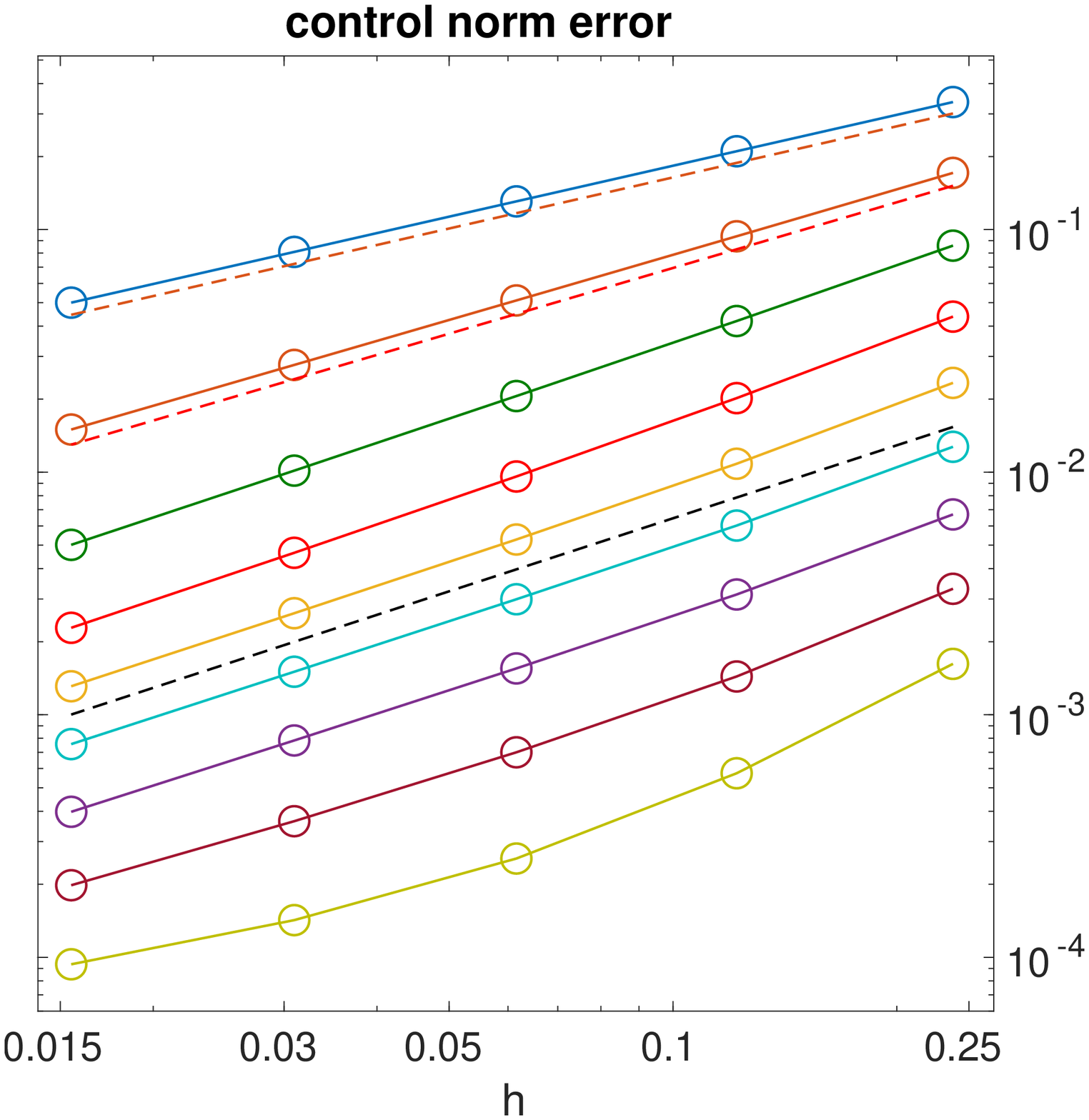}\\
\hspace{-0.4cm}\tiny{(G.2)}
\end{minipage}
\caption{Experimental rates of convergence for $\| \bar{q} - \bar{q}_h\|_{L^{2}(\Omega)}$ considering the fully discrete (G.1) and semidiscrete schemes (G.2) for $a = 0.001\|\bar{u}\bar{p}\|_{L^{\infty}(\Omega)}$ and $s \in \{0.1,0.2,..., 0.9\}$.}
\label{fig:ex-2.4}
\end{figure}

Figure \ref{fig:ex-2.4} shows for $s \in \{0.1,0.2,\ldots,0.9\}$ the experimental convergence rates for $\| \bar{q} - \bar{q}_h\|_{L^{2}(\Omega)}$ obtained for both the fully discrete and the semidiscrete schemes. We note that the experimental convergence rates coincide with those reported in Example 1. This particular result can be attributed to the following fact: the value of $a$ is so small that the singular behaviour of both $\bar{u}$ and $\bar{p}$ described in \eqref{eq:u_and_p} remains present with the computational resources at our disposal. This singular behaviour is inherited by the projection formula \eqref{eq:projection_control} on $\bar{q}$.

\subsection{Example 3}
We consider $a = 0.95 \|\bar{u}\bar{p}\|_{L^{\infty}(\Omega)}$ and $b = 1.5$.


\begin{figure}[!ht]
\centering
\psfrag{s=01}{{\normalsize $s = 0.1$}}
\psfrag{s=02}{{\normalsize $s = 0.2$}}
\psfrag{s=03}{{\normalsize $s = 0.3$}}
\psfrag{s=04}{{\normalsize $s = 0.4$}}
\psfrag{s=05}{{\normalsize $s = 0.5$}}
\psfrag{s=06}{{\normalsize $s = 0.6$}}
\psfrag{s=07}{{\normalsize $s = 0.7$}}
\psfrag{s=08}{{\normalsize $s = 0.8$}}
\psfrag{s=09}{{\normalsize $s = 0.9$}}
\psfrag{h}{{\normalsize $h$}}
\psfrag{o(h05)}{{\normalsize $h^{0.5}$}}
\psfrag{o(h06)}{{\normalsize $h^{0.6}$}}
\psfrag{o(h07)}{{\normalsize $h^{0.7}$}}
\psfrag{o(h08)}{{\normalsize $h^{0.8}$}}
\psfrag{o(h09)}{{\normalsize $h^{0.9}$}}
\psfrag{o(h10)}{{\normalsize $h^{1.0}$}}
{\large \hspace{0.7cm} $\|\bar{q} - \bar{q}_{h}\|_{L^{2}(\Omega)}$ for $a = 0.95\|\bar{u}\bar{p}\|_{L^{\infty}(\Omega)}$}
\begin{minipage}[c]{0.545\textwidth}\centering
\psfrag{control norm error}{}
\includegraphics[trim={0 0 0 0},clip,width=6.30cm,height=4.6cm,scale=0.4]{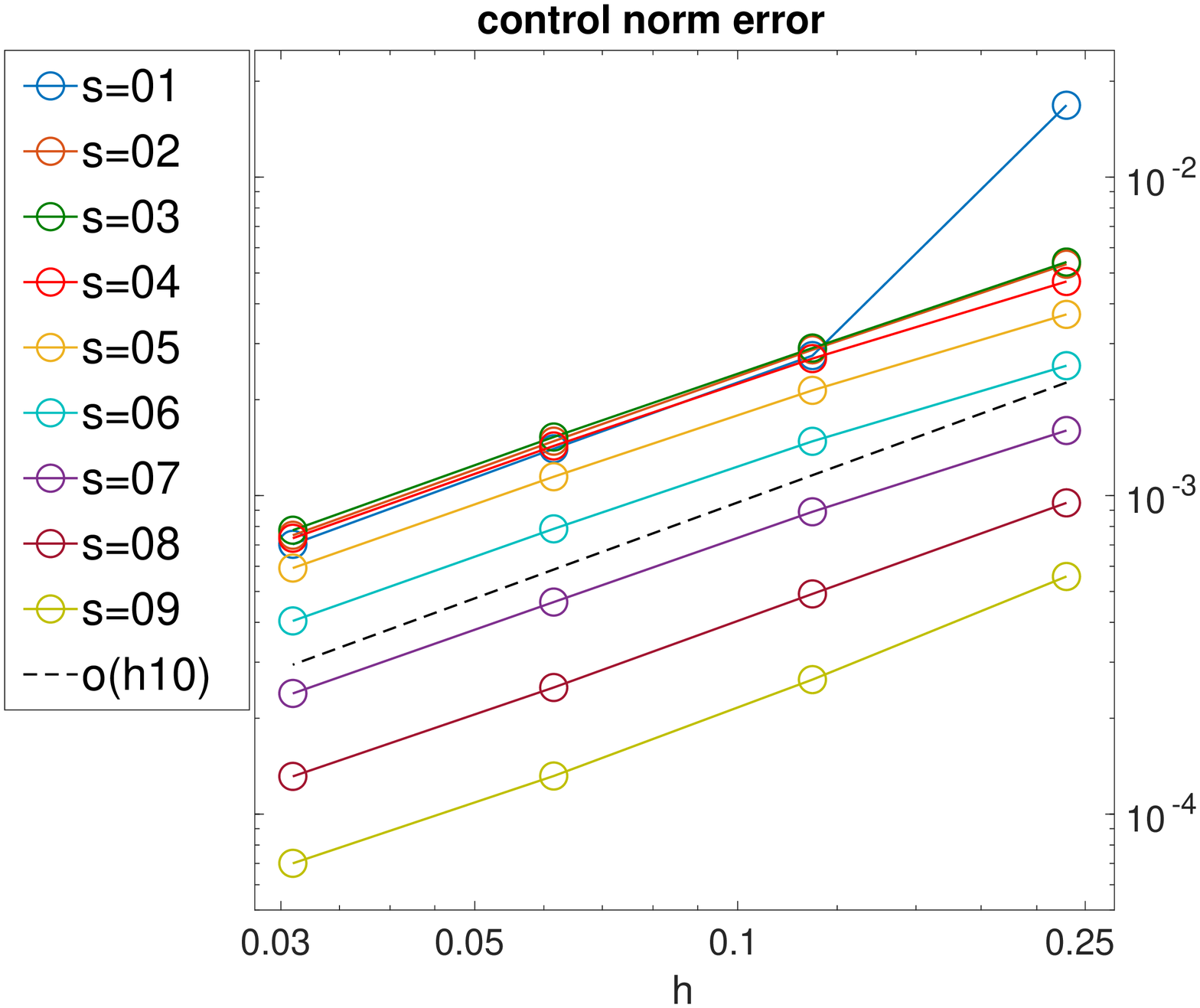} \\
\hspace{0.75cm} \tiny{(H.1)}
\end{minipage}
\begin{minipage}[c]{0.445\textwidth}\centering
\psfrag{control norm error}{}
\includegraphics[trim={0 0 0 0},clip,width=5.10cm,height=4.6cm,scale=0.4]{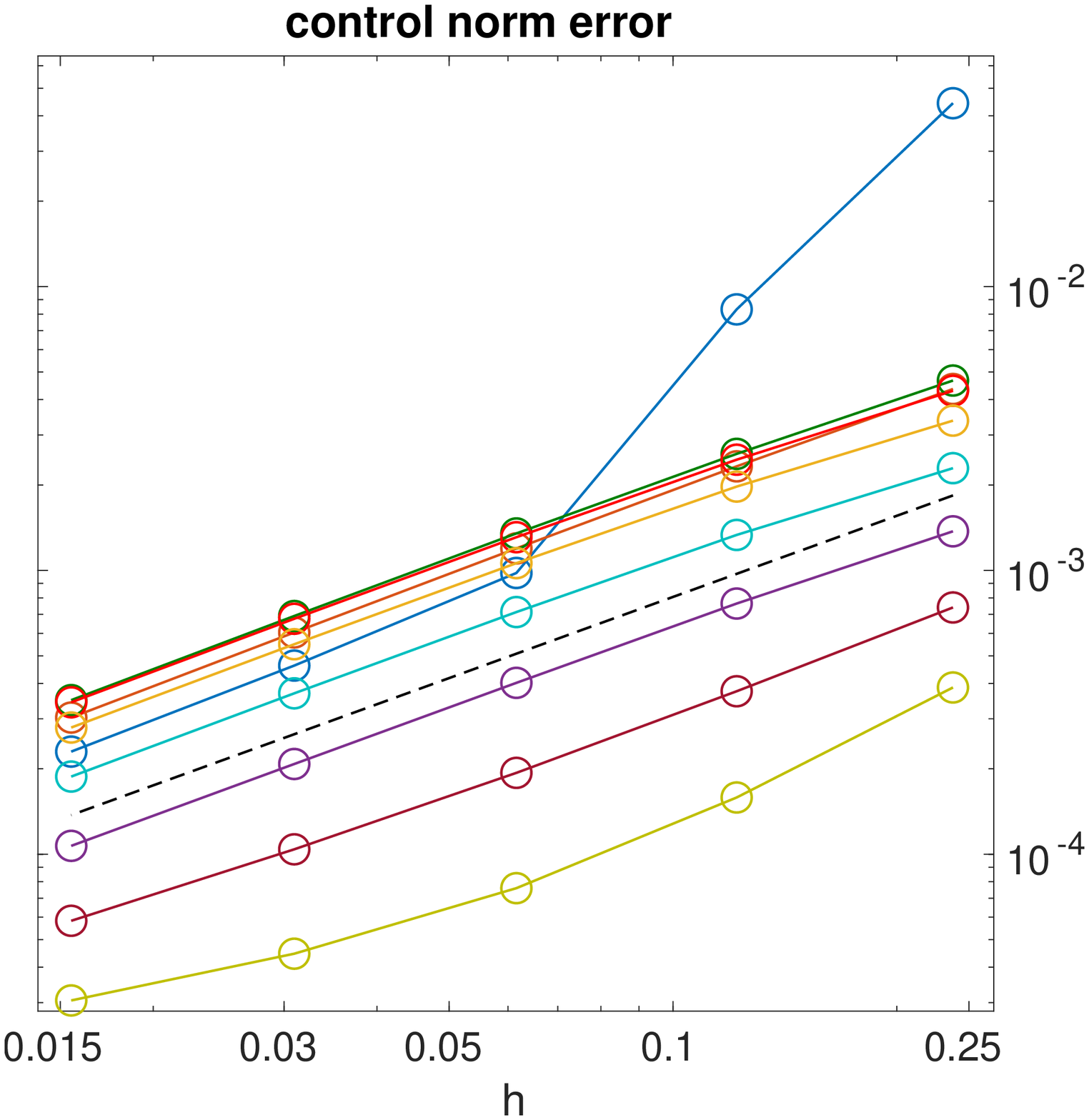}\\
\hspace{-0.4cm}\tiny{(H.2)}
\end{minipage}
\caption{Experimental rates of convergence for $\| \bar{q} - \bar{q}_h\|_{L^{2}(\Omega)}$ considering the fully discrete (H.1) and semidiscrete schemes (H.2) for $a = 0.95\|\bar{u}\bar{p}\|_{L^{\infty}(\Omega)}$ and $s \in \{0.1,0.2,..., 0.9\}$.}
\label{fig:ex-2.5}
\end{figure}

In Figure \ref{fig:ex-2.5} we present for $s \in \{0.1,0.2,\ldots,0.9\}$ the experimental convergence rates for $\| \bar{q} - \bar{q}_h\|_{L^{2}(\Omega)}$ obtained for both the fully discrete and the semidiscrete schemes. We note that $\| \bar{q} - \bar{q}_h\|_{L^{2}(\Omega)}$ achieves the experimental convergence rate $\mathcal{O}(h)$ for both methods and for all considered values of $s$. In contrast to Example 2, we have here that the singular behaviour near the boundary of $\Omega$ disappears when the restriction $a$ is large enough.

\begin{remark}[fully discrete versus semidiscrete approximation]
\DQ{In the following, we present what we consider to be the most important advantages and disadvantages of each discretization scheme.}
\\
\DQ{\textbf{Advantages(A)/disadvantages(D) of the fully discrete scheme:}
\begin{itemize}[leftmargin=*]
\item[(A1)] The scheme provides an explicit discrete control variable.
\item[(D1)] The scheme incurs additional computational costs due to the additional degrees of freedom required to discretize the admissible control set.
\item[(D2)] If the control set is discretized with piecewise constant functions, the expected convergence rate for the control approximation is always limited to $\mathcal{O}(h)$.
\end{itemize}
\textbf{Advantages(A)/disadvantages(D) of the semidiscrete scheme:}
\begin{itemize}[leftmargin=*]
\item[(A1)] A discrete control is not explicitly used in the computational implementation.
\item[(A2)] If the state and adjoint equations are discretized with piecewise linear functions, the data are smooth, and $\Omega$ is convex, then the rate $\mathcal{O}(h^2)$ can be obtained for the control approximation in the case $s =\ 1$. For the case $s \in (0,1)$, numerical evidence shows that such a rate is restricted to $\mathcal{O}(h)$.
\item[(D1)] An additional effort has to be made to compute an explicit discrete control by the projection formula. This can be interpreted as a post-processing step.
\end{itemize}}
\end{remark}

\bibliographystyle{siamplain}
\bibliography{bilinear_ref_sin_url}

\end{document}